\documentclass[hidelinks,11pt]{article}

\newcommand{\bol}{\boldsymbol}

\newcommand{\nex}{\boldsymbol{x}}

\newcommand{\de}{\,\mathrm{d}}                               
\newcommand{\e}{\operatorname{e}}

\newcommand{\andtext}{\quad\mbox{and}\quad}

\newcommand{\p}{\partial}

\newcommand{\real}{\mathrm{Re}\,}    
                                   
\newcommand{\imag}{\mathrm{Im}\,}

\newcommand{\lf}{\left}
\newcommand{\rg}{\right}

\newcommand{\Z}{\mathbb{Z}}     
\newcommand{\R}{\mathbb{R}}       
\newcommand{\C}{\mathbb{C}}

\usepackage{amsmath}
\usepackage{amsmath}
\usepackage{bm}
\usepackage{multirow}
\usepackage{amsmath}
\usepackage{amsfonts}
\usepackage{amsmath}
\usepackage{amssymb}  
\usepackage{graphicx}
\usepackage{caption}
\usepackage{anyfontsize}
\usepackage[section]{placeins}
\usepackage{mathrsfs}
\usepackage{upgreek}
\usepackage{amsthm}
\usepackage{subfig}
\usepackage{booktabs}
\usepackage{authblk}
\usepackage{xcolor}
\usepackage{cite}
\usepackage{soul}
\newcommand{\Zn}{\mathbb{Z}}
\newcommand{\Cn}{\mathbb{C}}
\newcommand{\sF}{\mathsf} 
\newcommand{\sspan}{\operatorname{span}}
\usepackage[]{algorithm2e}
\newtheorem{theorem}{Theorem}[section]
\newtheorem{lemma}[theorem]{Lemma}
\newtheorem{proposition}[theorem]{Proposition}

\newtheorem{remark}[theorem]{Remark}
\newtheorem{definition}[theorem]{Definition}

 \usepackage{xcolor}
 \usepackage[normalem]{ulem}

\usepackage{hyperref}
\usepackage{cleveref}

\title{Reflectionless discrete perfectly matched layers for higher-order finite difference schemes}

 \author[1]{Vicente A. Hojas\thanks{V.A. Hojas was supported by Beca ANID de Magíster Nacional N. 22230599.}}
\affil[1]{\small{School of Engineering, Pontificia Universidad Cat\'olica de Chile, Santiago, Chile}}
 \author[2]{Carlos P\'erez-Arancibia}
 \affil[2]{\small{Department of Applied Mathematics and MESA+ Institute, University of Twente, Enschede,  The Netherlands}}
\author[3]{Manuel A. S\'anchez\thanks{M.A. S\'anchez was supported by FONDECYT Regular  N. 1221189 and by Centro Nacional de Inteligencia Artificial CENIA, FB210017, Basal ANID Chile.}}
 \affil[3]{\small{Institute for Mathematical and Computational Engineering, Pontificia Universidad Cat\'olica de Chile, Santiago, Chile}}

\date{\today}

\begin{document}

\maketitle

%*************************************************************

\begin{abstract}

This paper introduces discrete-holomorphic Perfectly Matched Layers (PMLs) specifically designed for high-order finite difference (FD) discretizations of the scalar wave equation.  In contrast to standard PDE-based PMLs, the proposed method achieves the remarkable outcome of completely eliminating numerical reflections at the PML interface, in practice achieving errors at the level of machine precision. Our approach builds upon the ideas put forth in a recent publication [\emph{Journal of Computational Physics} 381 (2019): 91-109] expanding the scope from the standard second-order FD method to arbitrary high-order schemes. 
This generalization uses additional localized PML variables to accommodate the larger stencils employed. 
We establish that the numerical solutions generated by our proposed schemes exhibit a geometric decay rate as they propagate within the PML domain. To showcase the effectiveness of our method, we present a variety of numerical examples, including waveguide problems. These examples highlight the importance of employing high-order schemes to effectively address and minimize undesired numerical dispersion errors, emphasizing the practical advantages and applicability of our approach.


\end{abstract}

\maketitle %% required

\section{Introduction}

Problems of wave propagation in unbounded domains play a paramount role in many applications, especially in cases where the behavior of the underlying physical phenomena is governed by wave equations in linear acoustics, electromagnetics, or elastodynamics. The infinite propagation of waves within unbounded domains is mathematically expressed through the use of an appropriate \emph{radiation condition}. This condition serves as a boundary condition at infinity, ensuring not only the accurate representation of the physical behavior of the solution but also the well-posedness of the problem. When not expressed as a  uniform limit, this condition, which is often referred to as the exact transparent boundary condition, entails the evaluation of rather involved non-local pseudo-differential operators acting on the wavefield traces at a (bounded) artificial surface enclosing the region of interest. The effective approximation and feasible numerical implementation of such condition---which is essential to effectively reduce the problem to a bounded computational domain where, e.g., off-the-shelf finite difference (FD) and finite element (FE) PDE solvers can be applied---poses important theoretical and practical challenges on which much has been written about over the last half a century; see, e.g.,~\cite{givoli2013numerical,taflove} and references therein. 

In a nutshell, domain truncation typically involves two primary approaches: Absorbing Boundary Conditions (ABCs) (also referred to as non-reflecting boundary conditions) and Perfectly Matched Layers (PMLs). Classical local ABCs, which rely on Pad\'e~\cite{Engquist1977} or more sophisticated expansions of the underlying dispersion relation~\cite{halpern1988wide,bayliss1980radiation,bayliss1982boundary}, lead to boundary conditions that although local in nature, involve evaluation of high-order time- and space-derivatives of the solution or, alternatively, the introduction of auxiliary variables; see~
\cite{higdon1986absorbing} for a general framework that encompasses many local ABCs. Owing to the fact that they rely on approximations of the exact transparent boundary condition, local ABCs inherently introduce errors that manifest themselves as spurious reflections, which exist at the continuous level even before any discretization is applied~\cite{halpern1987error}. Non-local ABCs based on retarded potentials~\cite{ting1986exact} as well as boundary integral equation formulations~\cite{greenspan1966numerical,johnson1980coupling} (e.g., finite element and boundary element coupling techniques) and spectral series representations of Dirichlet-to-Neumann maps~\cite{fix1978variational,maccamy1980finite,goldstein1982finite,canuto1985spectral,keller1989exact,nicholls2004exact}, have also been extensively investigated; see, e.g.,\cite{givoli2013numerical,ihlenburg1998finite}. However, as local ABCs, they suffer from spurious reflections stemming from various sources, such as the discretization of the associated singular integral operators and truncation of Fourier-Bessel and spherical harmonics series expansions.

The most widely adopted approach for domain truncation is, nonetheless, the PML method, initially proposed by J.-P. Berenger in \cite{Berenger1994}.  Unlike ABCs that introduce boundary conditions to absorb outgoing waves at the artificial boundary, PMLs consist of a finite-thickness artificial-material layer surrounding the truncated portion of the physical domain. This artificial material is specifically designed to absorb outgoing waves impinging on the interface between the physical and PML domains without generating undesired reflections. At the continuous level, this is achieved by formulating a material that perfectly matches the ambient medium and exponentially dampens outgoing waves as they propagate through the absorbing layer; a comprehensive review can be found in \cite[Ch. 7]{taflove}.
The appeal of the PML method stems from its ease of implementation in existing PDE solvers, as it does not require the evaluation of high-order derivatives or the computation of non-local operators. However, at the discrete level, the ``perfect matching" is compromised due to the discrete representation of material properties. Consequently, challenging-to-control spurious numerical reflections emerge, ultimately contaminating the sought-after approximate numerical solution. Extensive research efforts have been dedicated to developing optimized PML formulations that minimize such errors (see, e.g., \cite{asvadurov2003,Berenger1999,Berenger2002,katz1994validation,bermudez2006optimal,bermudez2007optimal,michler2007improving,modave2014optimizing,collino1998optimizing}). Interestingly, in~\cite{asvadurov2003}, an optimal PML is developed, leading to a purely imaginary complex stretching which is proven to be equivalent to an ABC approach based on Pad\'e approximations of the square root. Nevertheless, it is worth noting that the presence of spurious numerical reflection errors in the PML discretization approach is essentially unavoidable \cite{chew1996perfectly}.

An important advancement in this field is the introduction of the reflectionless discrete Perfectly Matched Layer (RDPML) scheme, developed entirely at the discrete level \cite{Chern_2019}.  Although other discrete-based approaches exist to construct ABCs and non-reflecting boundary conditions~\cite{arnold2003,rowley2000}, as far as the authors are aware, the second-order FD scheme \cite{Chern_2019} for the scalar wave equation stands as the sole existing PML-based method that achieves effective reflectionlessness at the discrete level. The key concept underlying its formulation lies in extending the complex coordinate stretching approach of PMLs \cite{Chew1994} to the discrete complex plane, employing discrete complex analysis tools \cite{duffin1956basic}. To achieve the discrete holomorphicity of the numerical solution, this scheme incorporates two additional unknown vectors that represent the discrete complex derivatives of the solution at the grid points located within the PML domain.

In this work, we build upon Chern's second-order scheme \cite{Chern_2019} to develop RDPMLs for arbitrary higher-order FD spatial discretizations. Our primary motivation for this research is to tackle the challenge of dispersion errors, which are well-known limitations that affect the accuracy of both time-domain and frequency-domain FD schemes for wave-type problems, particularly when dealing with elongated domains such as waveguides (see, e.g., \cite{trefethen1982group} and \cite[ch. 5]{Trefethen1996}). 
Similar to Chern's scheme, our proposed schemes incorporate additional variables that correspond to the discrete complex derivatives of the approximate solution at the grid points within the PML domain. The number of these additional variables depends on the chosen stencil size or the order of the FD scheme employed.

To establish a clear focus, we concentrate on solving the linear scalar wave equation in free space:
\begin{align}
    \frac{\partial^2 u}{\partial t^2}(\nex, t)-
    \Delta u(\nex, t) & = 0, \quad 
    t>0, \; \nex \in \mathbb{R}^d,\quad d=1,2,
    \label{eqn:wave_1d}
\end{align}
 subject to the initial conditions
\begin{equation*}
u(\nex, 0) = g_0(\nex) \andtext \frac{\partial u}{\partial t}(\nex, 0) = g_1(\nex),    \qquad \nex\in\R^d,
\end{equation*}
where $g_0$ and $g_1$ are given smooth functions. (In \Cref{sec:numerical_results} we also consider a waveguide problem with Dirichlet boundary conditions.)  While we provide explicit expressions for RDPML schemes in one and two spatial dimensions ($d=1,2$), our results seamlessly extend to higher spatial dimensions through the use of tensor products. 

The paper is structured as follows: \Cref{sec:one_dim} revisits the definition of some relevant properties of the class of one-dimensional high-order FD schemes for which we develop RDPMLs. \Cref{sec:complex} then establishes the theoretical framework for the  discrete complex stretching method and proves the reflectionless property of the schemes (see \Cref{thm:reflectionless} and \Cref{rem:afterThrm34}). In \Cref{sec:PML_eqs}, we derive the PML equations for the one-dimensional wave equation, and in \ref{sec:two_dims}, we extend these derivations to the two-dimensional wave equation. \Cref{sec:linear_system} focuses on the reduced (time-harmonic) wave equations, providing expressions for implementing the schemes without auxiliary variables and discussing the structure of the resulting linear systems. Numerical examples are presented in \Cref{sec:numerical_results}, and the conclusions of this work are provided in~\Cref{sec:conclusion}.

\section{High-order finite-difference schemes in one dimension}\label{sec:one_dim}
% ---------------------------------------------
We start off this section by establishing a few elementary results on high-order finite difference operators. Throughout this paper, we employ centered and symmetric finite difference approximations of the second-order space derivative ($1d$-Laplacian) at grid points $x_j = jh$, $j\in\Zn$, $h>0$, which give rise to the following $(2p+1)$-points wide stencils specified by  coefficients $\{a_r\}_{r=-p}^p\subset\mathbb R$:
\begin{align}\label{eqn:finitedifferenceoperator0}
    \frac{\de^2v}{\de x^2}\Big|_{x_j} & =
    \sum_{r=-p}^p a_r v(x_{j+r}) +\mathcal O(h^{2p}),\quad j\in\mathbb Z,
\end{align}
where we have assumed that $v\in C^{2p}(\R)$, $p\in \mathbb N$.
 The schemes~\eqref{eqn:finitedifferenceoperator0} that we consider are assumed of optimal order of accuracy which means that their order is at least $2p$. Furthermore, we assume that the stencil is symmetric which means that $a_r=a_{-r}$ for all $r=1,\ldots p$. 
 
As it turns out, arbitrarily high-order symmetric stencils satisfying~\eqref{eqn:finitedifferenceoperator0} can be easily obtained by plugging in a Fourier mode $v(x) = \e^{i\xi x}$ into~\eqref{eqn:finitedifferenceoperator0} which yields the relation
 $$
    \sum_{r=-p}^p a_r v(x_{j+r}) =\e^{ijh\xi}\left(a_0+2\sum_{\ell = 1}^p a_\ell \cos(\ell h\xi)\right).
 $$
To find the coefficients $\{a_r\}_{r=0}^p$, we expand  $\cos(\ell h\omega)$ above in Taylor's series of order $p$ about $h\xi =0$ and make the resulting approximation equal to the exact second-order derivative, i.e., $v''(x_j)=-\xi^2v(x_j)$. This gives rise to the equation
\begin{equation*}%\label{eq:disc_disp_rel_2}
a_0 +2\sum_{\ell=1}^p a_\ell \sum_{n=0}^p\frac{(-1)^n(\ell h\xi)^{2n}}{(2n)!} = -\xi^2,
\end{equation*}
from where the linear system for the coefficients $\{a_r\}_{r=0}^p$ is obtained, $\delta_{j,k}$ the Kronecker delta,
\begin{equation*}%\label{eq:stencil_coeff_eqn}
\frac{a_0\delta_{n,0}}{2}+\sum_{r=1}^p (rh)^{2n}a_r=\delta_{n,1},\quad n=0,\ldots,p.
\end{equation*}
We mention in passing that although we only address in detail the derivations of the RDPML method for symmetric even-order finite difference schemes,  our findings can be extended to more sophisticated schemes tailored to wave-type problems such as  dispersion-relation-preserving~\cite{tam1993dispersion} and joint time–space domain~\cite{liu2009new} schemes. 

\begin{table}[htp]
\centering\small
\begin{tabular}{cl@{\hskip .0in}c@{\hskip .3in}c@{\hskip .3in}c@{\hskip .3in}c@{\hskip .3in}c@{\hskip .3in}c@{\hskip .3in}c@{\hskip .3in}c@{\hskip .3in}c}
\toprule
Order & & \multicolumn{9}{c}{Coefficients $a_{r}\cdot h^2$, $r=-p,\ldots,p$} \\[0.2em]
$(2p)$& & $-4$ & $-3$ & $-2$ & $-1$ & $0$ & $1$ & $2$ & $3$ & $4$ \\ 
\midrule \\  
$2$ && & & &$1$ &$-2$ & $1$ & & & \\[1em]
$4$ && & &$ -\frac{1}{12}$ & $\frac{4}{3}$ & $-\frac{5}{2}$& $\frac{4}{3}$ &$ -\frac{1}{12}$  & &  \\[1em]
$6$ && & $\frac{1}{90}$ & $-\frac{3}{20}$ & $\frac{3}{2}$ & $-\frac{49}{18}$ & 
        $\frac{3}{2}$ & $-\frac{3}{20}$ & $\frac{1}{90}$ & \\[1em]
$8$ && $-\frac{1}{560}$ & $\frac{8}{315}$ & $-\frac{1}{5}$ & $\frac{8}{5}$ & 
        $-\frac{205}{72}$ & $\frac{8}{5}$ & $-\frac{1}{5}$ & $\frac{8}{315}$ & 
        $-\frac{1}{560}$\\[1em]
\bottomrule
\end{tabular}\vskip2mm
\caption{Coefficients of the centered finite difference approximation operator to the Laplacian in one dimension; see e.g. \cite{Fornberg1988}.}
\label{tab:fd_examples}
\end{table}
 
In the derivations of the RDPML method presented in the next sections, we apply high-order finite difference schemes of the form~\eqref{eqn:finitedifferenceoperator0} to the spatial discretization of the wave equation~\eqref{eqn:wave_1d} while keeping the time variable continuous. To simplify the analysis of the spatial discretization, for the time being, we consider the one-dimensional ($d=1$) wave equation in the frequency domain whereby applying Fourier transform, 
$$
u(x,t)=\frac{1}{\sqrt{2 \pi}} \int_{-\infty}^{\infty} \e^{-i \omega t} v(x,\omega) \de \omega,
$$
 we arrive at the reduced  wave (Helmholtz) equation 
 \begin{equation}\label{eq:helm}
\frac{\de^2 v}{\de x^2}(x) +\omega^2 v(x) = 0,\quad x\in\R,
\end{equation}
for $v=v(\cdot,\omega):\R\to\C$. Upon applying the high-order finite difference scheme~\eqref{eqn:finitedifferenceoperator0} to~\eqref{eq:helm}  we then obtain 
\begin{equation}\label{eq:disc_Helm}
\mathcal L_p[\mathsf v] +\omega^2\mathsf v=\sF 0,
\end{equation}
where \begin{equation}\label{eqn:finitedifferenceoperator}
\mathcal L_p [\mathsf v]_j :=  \sum_{r=-p}^p a_r v_{j+r},\quad j\in\Zn,
\end{equation} with $\mathsf v=\mathsf v(\omega)\in\mathbb C^{\mathbb Z}$ containing the approximate grid values of the smooth function~$v(\cdot,\omega)$ in~\eqref{eq:helm}. Borrowing the notation for sequences, here and in the sequel we sometimes denote by $\mathbb K^\mathbb Z$, with $\mathbb K=\mathbb R$ or $\mathbb C$, the set of discrete functions $\sF v:\mathbb Z\to \mathbb K$,  and by $\mathbb K^{\mathbb Z\times \mathbb Z}$ the set of functions $\sf V:\mathbb Z^2\to\mathbb K$. The difference operator $\mathcal L_p$ can then be understood as a linear mapping from $\mathbb K^\mathbb Z$ to $\mathbb K^\mathbb Z$.

As it turns out, unlike the continuous problem~\eqref{eq:helm}, which admits only two linearly independent solutions (modes) $\e^{i\omega x}$ and $\e^{-i\omega x}$, the discretized equation~\eqref{eq:disc_Helm} may admit more linearly independent solutions depending on the order of the scheme. To establish the effectiveness of the proposed discrete PML, we need to properly characterize such discrete modes. 
{ In order do so, we first define 
$$
\mathsf w(\xi)\in \mathbb C^{\mathbb Z},\quad \mathsf w(\xi):=(\e^{i j\xi h}), \quad j\in\mathbb Z,
$$ and for $\lambda \in \C$ and $r\in\{0,\ldots,p\}$ we let 
\begin{equation} \quad b_{r}(\lambda) :=
    \begin{cases}
    \displaystyle (a_0+\lambda) + 2\sum_{l=1}^{\lfloor p/2 \rfloor}(-1)^{l}a_{2l},& \text{if }r = 0,\medskip\\
    \displaystyle \sum_{l=k}^{\lfloor p/2 \rfloor} (-1)^{l-k}\binom{l+k-1}{2k-1}\frac{l}{k}a_{2l}, & \text{if } r= 2k\neq 0, \medskip\\
    \displaystyle \sum_{l=k}^{\lceil p/2 \rceil-1} (-1)^{l-k} \binom{l+k}{2k}\frac{2l+1}{2k+1}a_{2l+1},& \text{if } r = 2k+1.
    \end{cases}\label{eq:pol_coeff}
    \end{equation}
where $\{a_{r}\}_{r=-p}^p$ are the coefficients of the stencil~\eqref{eqn:finitedifferenceoperator0}. With these definitions at hand, we introduce the following lemma:

\begin{lemma}\label{thm:operator_sols}
Let $\lambda\in\mathbb C$ and $\{z_r(\lambda)\}_{r=1}^p\subset\mathbb C$ be the roots of the $p$th-degree polynomial $P_p(\cdot;\lambda)\in \mathcal P^{p}(\mathbb C)$  defined as
    \begin{equation}\label{eqn:polyPc}
    P_{p}(z;\lambda) = \sum_{r=0}^{p} b_{r}(\lambda) z^{r},  \end{equation} which are assumed to be distinct. 
Then,  the general solution $0\neq \mathsf  v\in \mathbb C^{\mathbb Z}$ of
\begin{equation}\label{eq:eigen}
\mathcal L_p[\mathsf v] +\lambda \mathsf v=\sF 0,
\end{equation}
satisfies $\mathsf v\in E_\lambda$, where $E_\lambda\subset \mathbb C^\mathbb Z$ is defined as
 \begin{align*}
            E_{\lambda} = 
            \begin{cases}
            \displaystyle
                \sspan\{\mathsf w(\xi_r), \mathsf  w(-\xi_r)\}_{r=1}^p, & 
                 z_r \neq \pm 2, r=1,\ldots,p, \\
                \sspan\{\{1, j\}\cup\{\mathsf  w(\xi_r), \mathsf  w(-\xi_r)\}_{r=2}^ p\},
                &  z_1 = 2, z_r\neq -2, r=2,\ldots,p,\\
                \sspan\{\{(-1)^j, j(-1)^j\}\cup\{\mathsf  w(\xi_r), \mathsf  w(-\xi_r)\}_{r=2}^p\},
                &  z_1 = -2, z_r\neq 2, r=2,\ldots,p,
            \end{cases}
    \end{align*}
in terms of   \begin{equation} 
\xi_r=\xi_r(\lambda) := \frac{1}{h}\cos^{-1}\left(\frac{z_r(\lambda)}{2}\right)\in\left\{z\in\Cn:0\leq\real z\leq \frac{\pi}{h}\right\}.\label{eq:cos_rel}
   \end{equation}
\end{lemma}}

\begin{proof}
Given $\lambda\in \mathbb C$, we seek non-trivial solutions  $0\neq\mathsf v=(v_j)$, $j\in\mathbb Z$, of \eqref{eq:eigen}.
% \begin{equation}\label{eqn:eig_problem0}
% \mathcal L_p[\mathsf v]_j =    \sum_{r=-p}^p a_r v_{j+r}  =-\lambda v_j , \quad j \in \mathbb{Z}.
% \end{equation}
To do so, we resort to the ansatz  $v_j = y^j$ for some $y \in \mathbb{C}\setminus \{0\}$
(see e.g. \cite[Chapter~2]{Bender2010}). Since $\mathcal L_p$ is a symmetric difference operator,~%\eqref{eqn:eig_problem0} 
\eqref{eq:eigen} can be expressed as
\begin{equation}\label{eqn:rational_equation0}
    (a_0+\lambda )+ \sum_{r=1}^p  a_r(y^r + y^{-r})  = 0. 
\end{equation}
Setting $z = y + 1/y$ and applying the binomial theorem, we obtain 
\begin{equation*}
    z^{r} = \sum_{k=0}^r{r \choose k} y^{r-k}y^{-k} = 
    \begin{cases}
    \displaystyle \binom{2\ell}{\ell}+\sum_{k=1}^{\ell} \binom{2\ell}{\ell-k}(y^{2k} + y^{-2k}), & \mbox{if }r=2\ell, \\
    \displaystyle \sum_{k=0}^{\ell} \binom{2\ell+1}{\ell-k}(y^{2k+1} + y^{-2k-1}), & \mbox{if }r=2\ell+1. 
    \end{cases}
\end{equation*}
Inverting the linear relations above we have the expression
\begin{equation}\label{eqn:yrs0}
    y^{r}+y^{-r} =
    \begin{cases}
    \displaystyle 2(-1)^{\ell} + \sum_{k=1}^{\ell} (-1)^{\ell-k} \binom{\ell+k-1}{2k-1}\frac{\ell}{k} z^{2k}, & \mbox{if }r=2\ell, \\
    \displaystyle \sum_{k=0}^{\ell} (-1)^{\ell-k}\binom{\ell+k}{2k}\frac{2\ell+1}{2k+1}z^{2k+1}, & \mbox{if }r=2\ell+1. 
    \end{cases}
\end{equation}
Thus, replacing \eqref{eqn:yrs0} in \eqref{eqn:rational_equation0}, we get that \eqref{eqn:rational_equation0} can be recast as the  polynomial equation 
\begin{equation*}%\label{eqn:Pc0}
    P_p(z;\lambda) = \sum_{r=0}^p b_r(\lambda)z^r = 0,
\end{equation*}
where the coefficients $b_r(\lambda)$ are given by~\eqref{eq:pol_coeff}.

Now, let $z_r(\lambda)$, $r=1,\ldots,p,$ be the assumed $p$ distinct roots of $p$th-degree polynomial $P_p(\cdot;\lambda)$% in \eqref{eqn:Pc0}
, and let $y_{\pm r}(\lambda)$ be the corresponding solutions of \eqref{eqn:rational_equation0}, which can be expressed in terms of $z_r(\lambda)$ as
\begin{align}
    y_{\pm r}(\lambda) & = \frac{z_r(\lambda)}{2} \pm i\sqrt{1 - \left(\frac{z_r(\lambda)}{2}\right)^2}.
    \nonumber
\end{align}
Assuming that $z_r(\lambda)\neq \pm 2$, we can write the solutions $\sF v = (v_j)$, $j\in\mathbb Z$, of %\eqref{eqn:eig_problem0}  
\eqref{eq:eigen} as a linear combination of $w_j(\xi)=\e^{ij\xi h}$:
$$
 v_j = \sum_{r=1}^{p}\left( A_{r} y_{r}^{j} + B_{r} y_{-r}^{j}\right)  = \sum_{r=1}^p C_rw_j(\xi_r) + C_{-r}w_j(-\xi_r), 
$$
for constants $A_r, B_r,C_r,C_{-r}$, for $r=1,..,p$,  where
\begin{equation*}%\label{eq:loc_wn}
    \xi_{r} =\xi_{r} (\lambda)= -\frac{i}{h} \ln\left(\frac{z_r(\lambda)}{2} + i\sqrt{1-\left(\frac{z_r(\lambda)}{2}\right)^{2}}\right) = \frac{1}{h} \cos^{-1}\left(\frac{z_r(\lambda)}{2}\right),
\end{equation*}
and where we have used the expression  $-i\ln(x + i\sqrt{1-x^2})$ for the complex inverse cosine function.

On the other hand, in the case when $z_r(\lambda)=2$, we set $r=1$ for simplicity,
and we get  that  the solution $y_{1}(\lambda) = 1$ of %\eqref{eqn:eig_problem0}
\eqref{eq:eigen} has multiplicity two.  The general solution of 
\eqref{eq:eigen} is in this case given by
    $$
    v_j  = A_1 + jB_1 + \sum_{r=2}^{p}A_rw_j(\xi_r) + B_{r}w_j(-\xi_r), \quad j\in\mathbb Z,
    $$
for constants $A_1, B_1$ and  $A_r, B_r$, for $r=2,\ldots,p$.
Similarly, if $z_1(\lambda)=-2$, we have that $y_1(\lambda)= -1$ is a repeated solution and the general solution of \eqref{eq:eigen} is given by  
$$
v_j  = (-1)^jA_1 + j(-1)^jB_1 + \sum_{r=2}^p A_r w_j(\xi_r) +B_{r}w_j(-\xi_r),\quad j\in\mathbb Z,
$$
for constants $A_1, B_1$ and  $A_r, B_r$, for $r=2,\ldots,p$.  \end{proof}

\begin{remark}
The coefficients~\eqref{eq:pol_coeff} of the polynomial $P_p(z;\lambda)$ in  \eqref{eqn:polyPc} are easily computed from the stencil coefficients $\{a_r\}_{r=-p}^p$. Indeed, the first four polynomials are 
\begin{alignat*}{2}
P_1(z;\lambda) & = (a_0+\lambda) + a_1 z, \\
P_2(z;\lambda) & = (a_0+\lambda-2a_2) + a_1 z +a_2 z^{2}, \\
P_3(z;\lambda) & = (a_0+\lambda-2a_2) + (a_1-3a_3)z + a_2 z^{2} + a_{3} z^{3}, \\
P_4(z;\lambda) & = (a_0+\lambda-2a_2+2a_4) + (a_1-3a_3)z + (a_2-4a_4)z^{2} + a_3z^3 + a_4 z^{4}.
\end{alignat*}
In particular, for the coefficients displayed on Table \ref{tab:fd_examples}, the resulting polynomials are:
\begin{subequations}\begin{alignat}{2}
h^{2}P_1(z;\lambda) & = \left(-2+\lambda h^{2}\right) + z, \\
h^{2}P_2(z;\lambda) & = \left(-\frac{7}{3}+\lambda h^{2}\right) + \frac43 z - \frac{1}{12} z^{2},\\
h^{2}P_3(z;\lambda) & = \left(-\frac{109}{45}+\lambda h^{2}\right) + \frac{22}{15} z -\frac{3}{20} z^{2} + \frac{1}{90} z^{3}, \\
h^{2}P_4(z;\lambda) & = \left(-\frac{772}{315}+\lambda h^{2}\right) + \frac{32}{21} z -\frac{27}{140} z^{2} + \frac{8}{315} z^3 -\frac{1}{560} z^{4}.
\end{alignat}\label{eq:exp_pol}\end{subequations}
\end{remark}

% \begin{figure}[ht!]
%         \centering
% 	\subfloat[$|\log(P_1(2\cos(\xi h);\lambda))|$]{\includegraphics[width=0.24\linewidth]{Figures/fig_1a.pdf}}
% 		\label{fig:P1}
% 	\subfloat[$|\log(P_2(2\cos(\xi h);\lambda))|$]{\includegraphics[width=0.24\linewidth]{Figures/fig_1b.pdf}}
% 		\label{fig:P2}
%         	\subfloat[$|\log(P_3(2\cos(\xi h);\lambda))|$]{\includegraphics[width=0.24\linewidth]{Figures/fig_1c.pdf}}
% 		\label{fig:P3}
%       	\subfloat[$|\log(P_4(2\cos(\xi h);\lambda))|$]{  \includegraphics[width=0.24\linewidth]{Figures/fig_1d.pdf}}
% 	        \label{fig:P4}
%         \caption{Location of discrete modes.}
%         \label{fig:disp_relation}
%     \end{figure}

        \begin{figure}[ht!]
        \centering
	\subfloat[$|\log(P_1(2\cos(\xi h);\lambda))|$]{\includegraphics[width=0.35\linewidth]{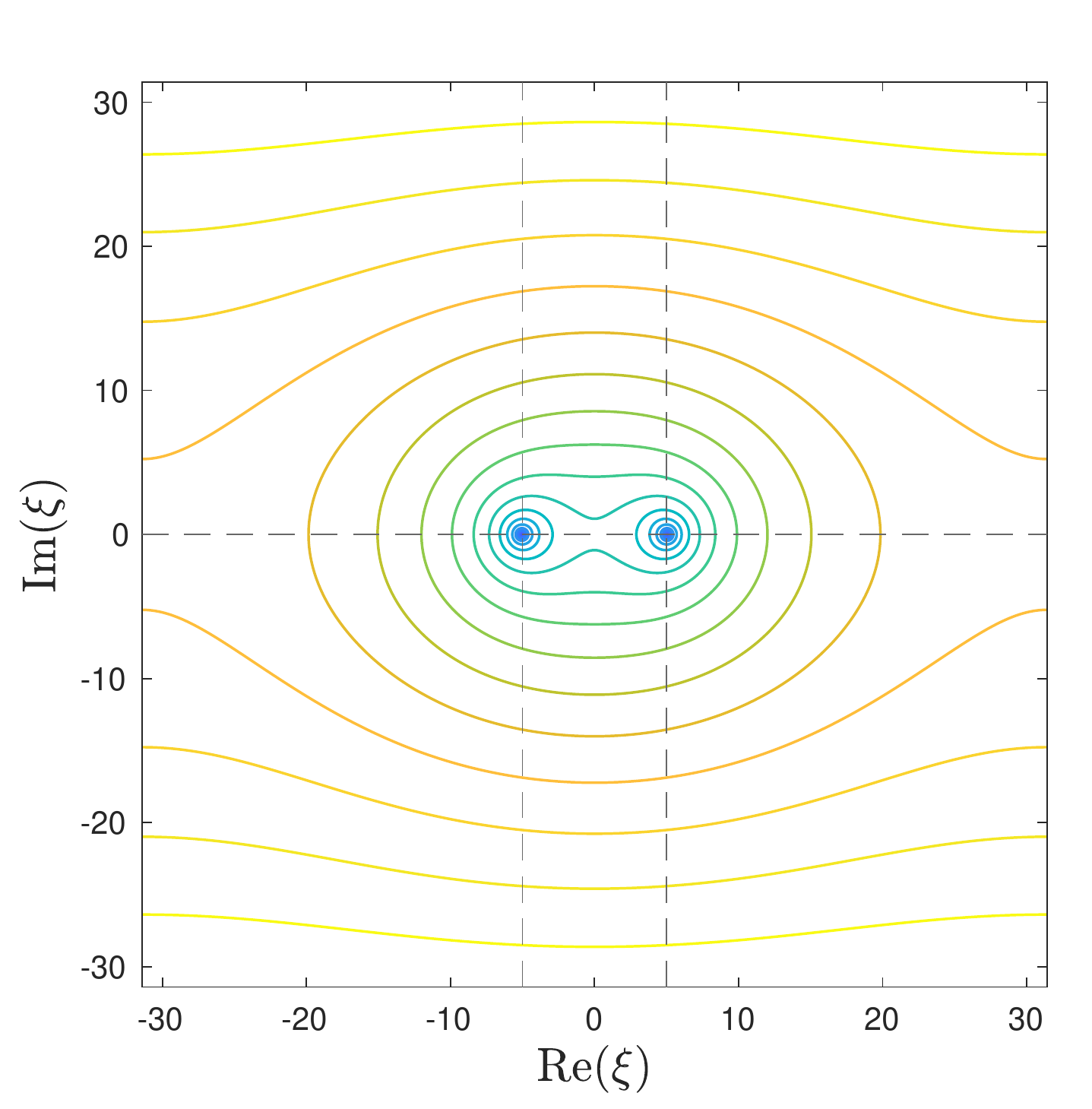}}
		\label{fig:P1}
	\subfloat[$|\log(P_2(2\cos(\xi h);\lambda))|$]{\includegraphics[width=0.35\linewidth]{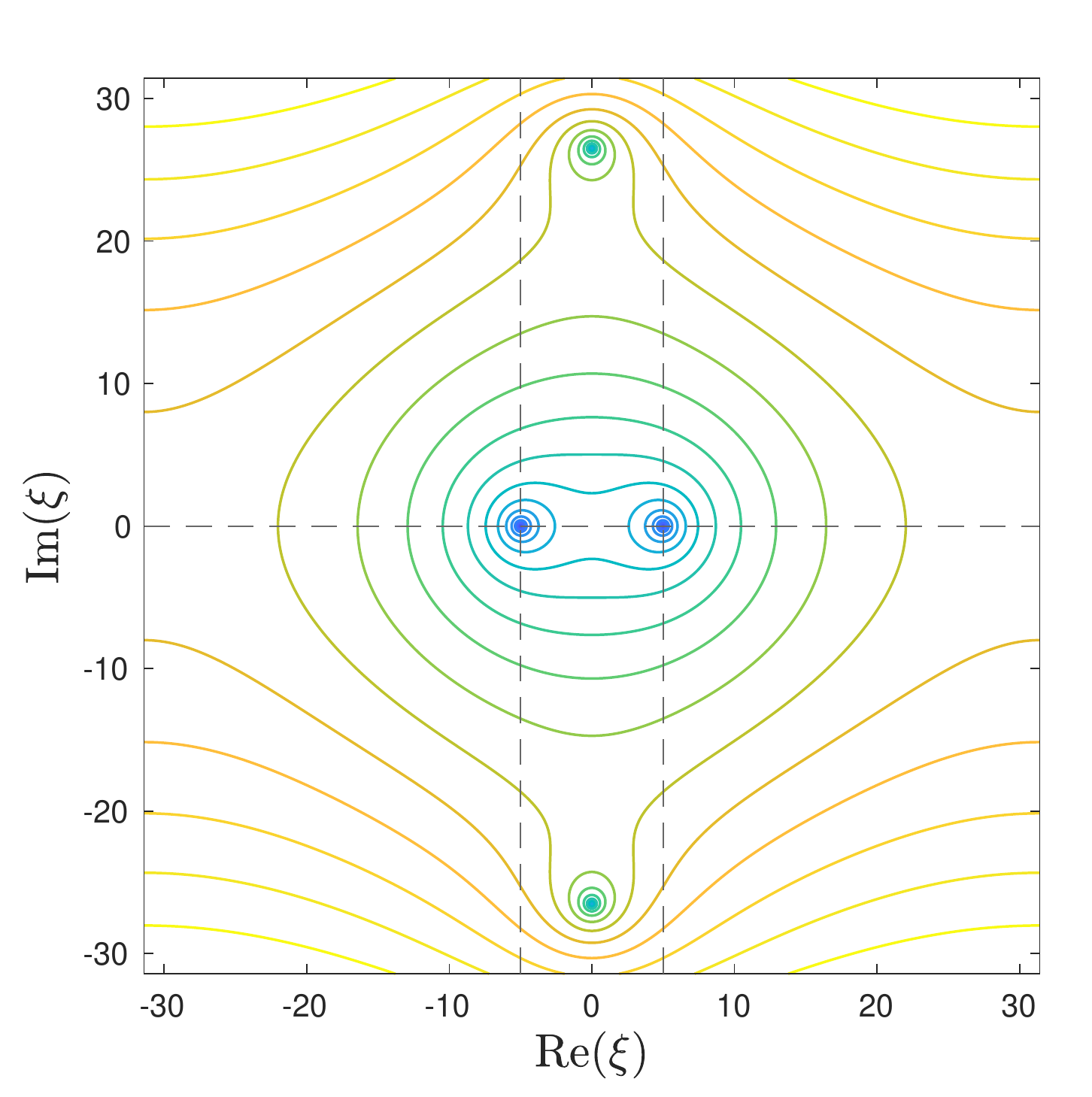}}
		\label{fig:P2}\\
        	\subfloat[$|\log(P_3(2\cos(\xi h);\lambda))|$]{\includegraphics[width=0.35\linewidth]{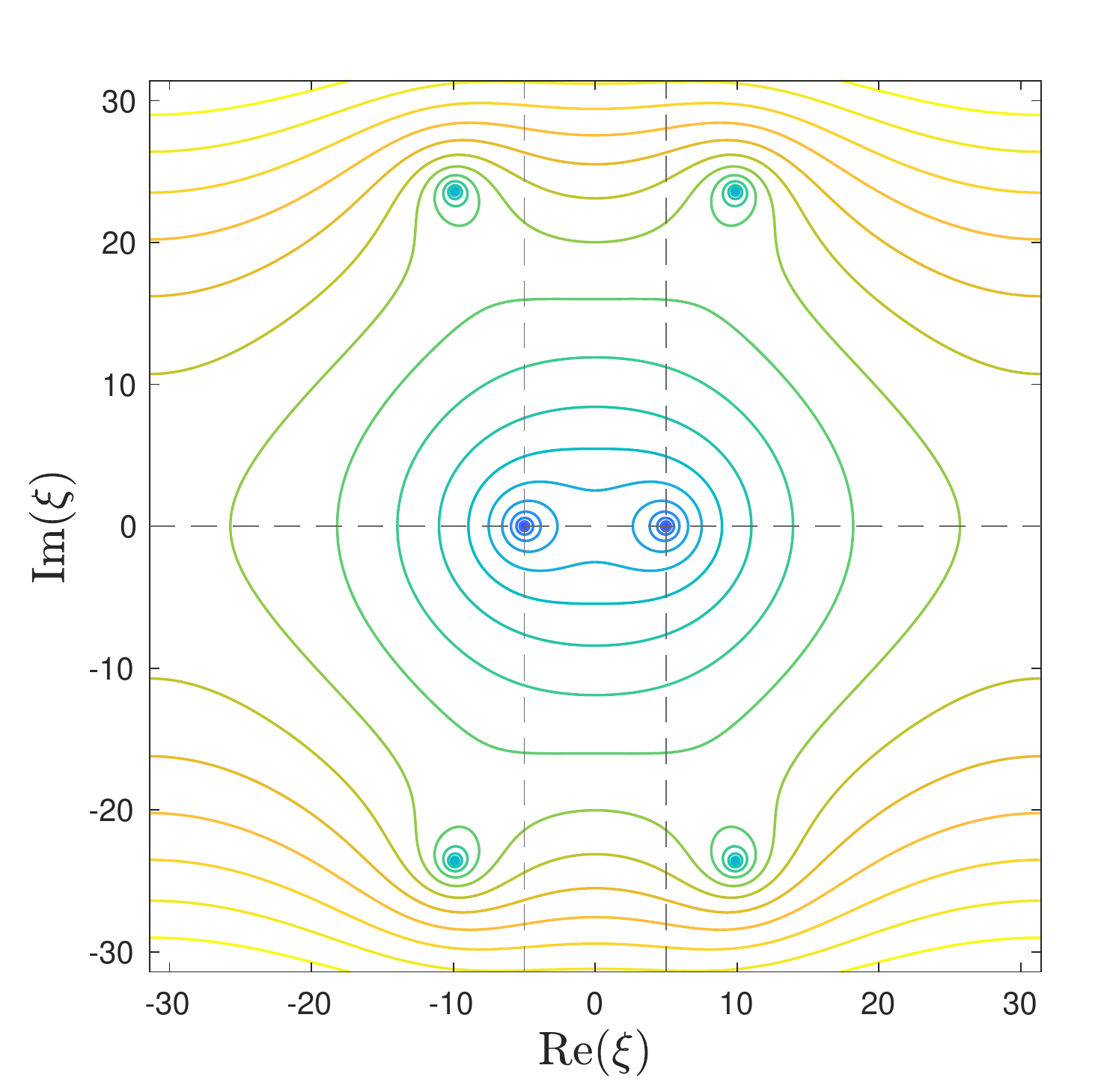}}
		\label{fig:P3}
      	\subfloat[$|\log(P_4(2\cos(\xi h);\lambda))|$]{  \includegraphics[width=0.35\linewidth]{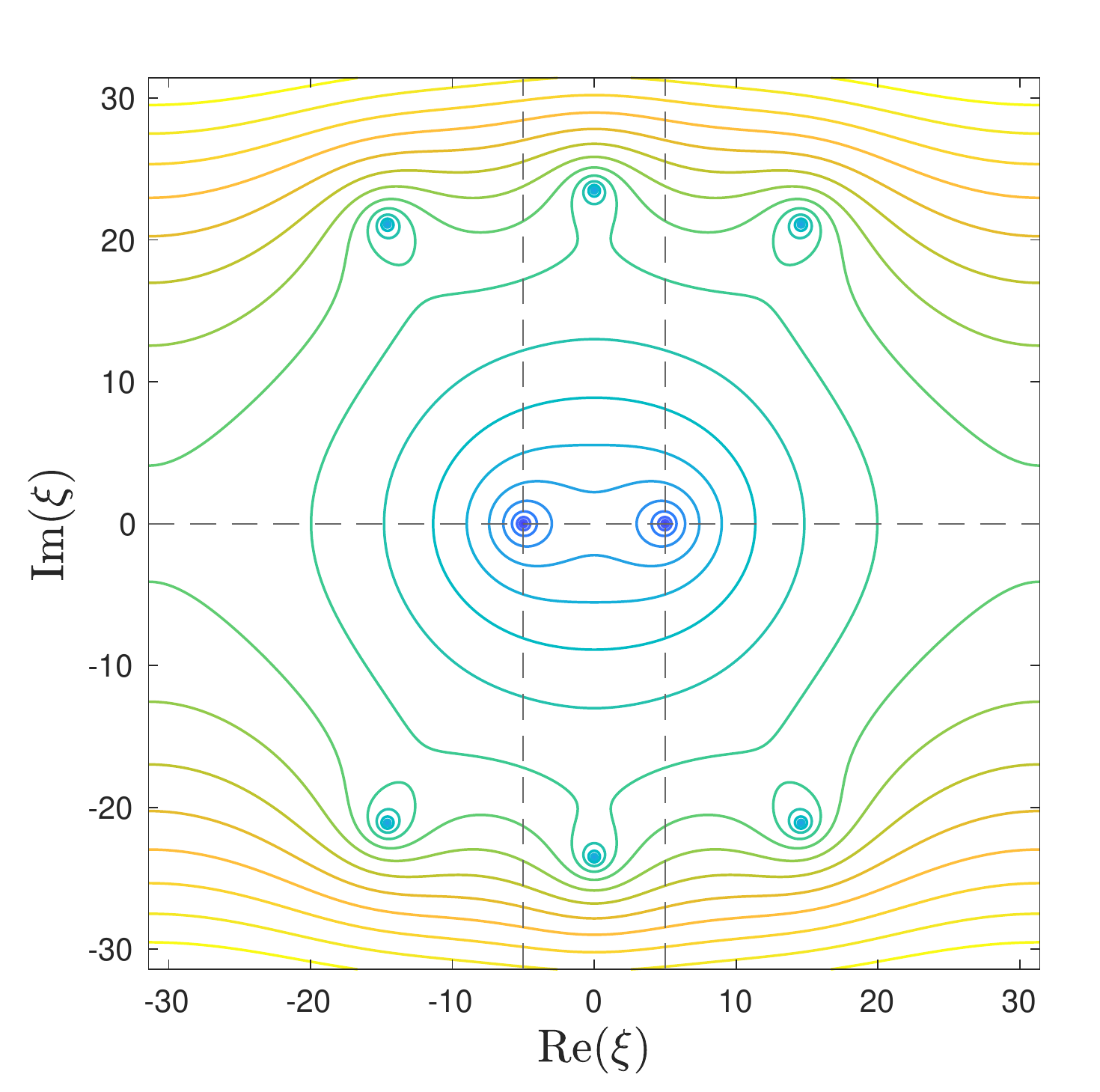}}
	        \label{fig:P4}
        \caption{Discrete mode locations.}
        \label{fig:disp_relation}
    \end{figure}

\begin{remark} Note that in order to avoid aliasing, we consider modes $\sF w(\xi)=(\e^{ij\xi h})\in\Cn^\Zn$, %in~\eqref{eq:modes}
with $\real\xi\in (-\frac{\pi}{h},\frac{\pi}{h}]$. Letting
$z=2\cos(\xi h)$ and in light of~\eqref{eq:cos_rel} we then have that the discrete wavenumbers associated with the $(2p)$th-order scheme~\eqref{eqn:finitedifferenceoperator0} satisfy the dispersion relation
\begin{equation}\label{eq:dispersion_relation}
P_p(2\cos(\xi h);\lambda)=0,\quad -\frac{\pi}{h}< \real\xi \leq \frac{\pi}{h}.\end{equation}
In particular, for the classical second-order scheme we have $h^2P_1(2\cos(\xi_r h) =0$ from where we obtain the well-known dispersion relation:
$$
4\sin^2\left(\frac{\xi h}{2}\right)=\lambda h^2, \quad -\frac{\pi}{h}<\real\xi\leq \frac{\pi}{h}.
$$ 
To help visualize the locations of $\xi_r$ for higher-order schemes, we consider \Cref{fig:disp_relation} which displays the level curves of the complex function $f(\xi)=|\log \lf(P_p(2\cos(\xi_r h);\lambda)\rg)|$  within the domain  $\{z \in \Cn: -\frac{\pi}{h}\leq \real z\leq \frac{\pi}{h},-\frac{\pi}{h}\leq \imag z\leq \frac{\pi}{h}\}$ corresponding to the polynomials $P_p$, $p=1,\ldots,4$, defined in~\eqref{eq:exp_pol} for the values $h=0.1$ and $\lambda =25$. The locations of the discrete mode wavenumbers at $\xi_r$ and  $-\xi_r$ can be easily identified in these figures as $f$ develops singularities at those points. The vertical lines mark the location of the exact wavenumbers at $\pm\sqrt{\lambda}=\pm 5$ in this example. 
\end{remark}

\section{Discrete complex stretching}\label{sec:complex} 

This section presents the elements of discrete complex analysis necessary for the derivation of the high-order RDPML schemes. Before proceeding with the derivations, we need to introduce certain definitions from discrete complex analysis; see~\cite{Chern_2019,duffin1956basic}.

 Let $\Lambda$ be an infinite quadrilateral lattice with sets of vertices, edges, and quadrilateral faces given by 
 \begin{align*}
 \mathcal V(\Lambda)&=\mathbb Z^2,\\
     \mathcal E(\Lambda) & = \{((j, k), (j+1, k)) | (j, k) \in \mathbb{Z}^2\} \cup 
     \{((j, k), (j, k+1)) | (j, k) \in \mathbb{Z}^2\},\quad\text{and}\\
     \mathcal F(\Lambda) & = \left\{ ((j, k), (j+1, k), (j+1, k+1), (j, k+1))| 
     (j, k) \in \mathbb{Z}^2\right\},
 \end{align*}
 respectively, where $((j,k),(l,m))$ denotes the edge connecting the vertices at $(j,k)$ and $(l,m)$ and $((j, k), (j+1, k), (j+1, k+1), (j, k+1))$ denotes the face with vertices $(j, k)$, $(j+1, k)$, $(j+1, k+1)$, and $(j, k+1))$.
 
 A complex-valued function defined on $\Lambda$ is then a mapping $\mathsf F:\mathcal V(\Lambda) \rightarrow \mathbb{C}$, with $\mathsf F=(F_{j,k})$ for $(j,k)\in \mathcal V(\Lambda)=\mathbb{Z}^2$, i.e., we can characterize the mappping as $\sF F\in\mathbb C^{\mathbb Z\times\mathbb Z}$. In particular, we 
 can define a (discrete) complex domain $(\Lambda,\mathsf Z)$ through a mapping 
 $\mathsf Z:\mathcal V(\Lambda) \rightarrow \mathbb{C}$. This allows us to 
introduce the concept of discrete holomorphicity.
 \begin{definition}[Discrete holomorphicity] \label{def:discrete_holo}
A function $\mathsf F: \mathcal V(\Lambda)
     \rightarrow \mathbb{C}$ is said to be complex-differentiable 
     with respect to $\mathsf Z:\mathcal V(\Lambda) \rightarrow \mathbb{C}$ at a face $Q_{j+\frac12, k+\frac12}=((j, k), (j+1, k), (j+1, k+1), (j, k+1)) \in \mathcal F(\Lambda)$ if it satisfies
     the discrete Cauchy-Riemann equation:
     \begin{equation}\label{eqn:discrete_holo}
         \frac{F_{j+1, k+1}-F_{j, k}}{Z_{j+1,k+1}-Z_{j, k}} = 
         \frac{F_{j, k+1}-F_{j+1, k}}{Z_{j, k+1}-Z_{j+1, k}},\quad (j,k)\in \mathcal V(\Lambda),
     \end{equation}
     where $\mathsf F = (F_{j,k})$ and $\mathsf Z = (Z_{j,k})$.
     In which case we denote by $D_{\mathsf Z}\mathsf F(Q_{j+\frac12, k+\frac12})$  the discrete complex derivative of $\mathsf F$ with respect to $\mathsf Z$ at $Q_{j+\frac12,k+\frac12}\in\mathcal F(\Lambda)$ which corresponds to either of the quotients in~\eqref{eqn:discrete_holo}. Furthermore, the function $\mathsf F$ is said to be discrete-holomorphic if it is complex-differentiable with respect to every face $Q \in \mathcal F(\Lambda)$.
 \end{definition}

% \subsection{Coordinate stretching}
With this definition at hand, we can proceed with the construction of the \emph{discrete complex coordinate stretching} that will be used to derive the PML. Following~\cite{Chern_2019}, we begin by  extending  the one-dimensional grid  $\{x_{j}\}_{j\in\mathbb Z}\subset\mathbb R$ to a specific discrete complex domain $(\Lambda,\mathsf Z)$ with $\mathsf Z=(Z_{j,k}): \mathcal V(\Lambda) \mapsto 
\mathbb{C}$  defined by
\begin{equation}\label{eq:complex_dom}
    Z_{j, k}  = (j+k)h + i\frac{h}{\omega}\sum_{\ell=0}^{k-1}\sigma_\ell, \quad (j, k) \in \mathbb{Z}^2,
\end{equation}
where  the \textit{damping coefficient} is defined as  
\begin{equation}\label{eq:sigma_def}
    \sigma_j \geq 0\ \ \text{if}\ \  j\geq0\ \ \text{and}\ \  \sigma_j= 0 \ \ \text{if}\ \ j<0. 
\end{equation}
It is important to note that $\mathsf Z\in\mathbb C^{\mathbb Z\times\mathbb Z}$ is constructed in such a way that the \emph{path}
\begin{align}\label{eq:complex_path}
    Z_{0, j} & = jh + i\frac{h}{\omega}\sum_{\ell=0}^{j-1}\sigma_{\ell}, \quad j \in \mathbb Z,
\end{align}
satisfies  $Z_{0,j}=x_{j}=jh$ for $j<0$ and $\imag Z_{0,j}\geq0$ for $j\geq 0$. In other words, the portion of the path $Z_{0,j}$, $j<0$, lies on the (real) physical domain, while $Z_{0,j}$, $j\geq0$, lies on what we will consider our PML domain. Figure \ref{fig:complex_domain}  depicts the discrete complex domain $(\Lambda,\mathsf Z)$ as well as the stretching path. In particular, the variables involved in the derivation of the fourth-order scheme are depicted in that figure.

\begin{figure}[ht]
    \centering
    \includegraphics[width=0.9\linewidth]{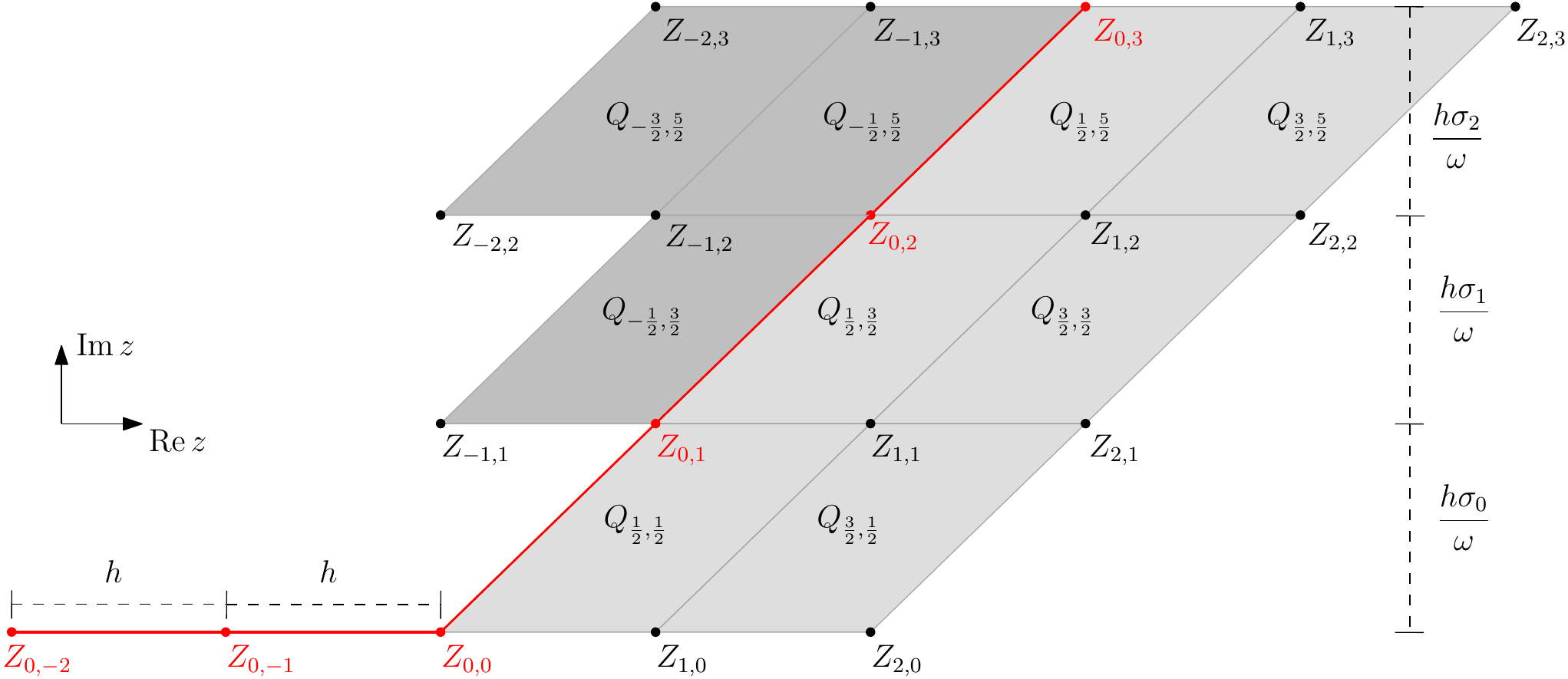}
    \caption{Depiction of the discrete complex domain $(\Lambda,\mathsf Z)$ defined by~\eqref{eq:complex_dom} and stretching path~\eqref{eq:complex_path}, which is marked by the red dots.}
    \label{fig:complex_domain}
\end{figure}

Analogous to the continuous setting, we now consider the holomorphic extension of solutions $\sF v\in\mathbb C^{\mathbb Z}$  of~\eqref{eq:disc_Helm} to the complex domain $(\Lambda,\mathsf Z)$ defined in~\eqref{eq:complex_dom}.  As pointed out by Chern~\cite{Chern_2019}, there are infinitely many ways to extend $\mathsf v\in\mathbb C^{\mathbb Z}$ to $\mathsf V\in\mathbb C^{\mathbb Z\times\mathbb Z}$ in such a way that $ 
 V_{0,j} = v_{j}$, $j<0$, and
\begin{equation}\label{eqn:discrete_extended}
\mathcal L_p[\mathsf V_{*,k}] +\omega^2\mathsf V_{*,k}= \mathsf 0,
   \quad 
     k \in \mathbb{Z},
\end{equation}
where $\mathsf V_{*,k}\in\mathbb C^{\mathbb Z}$ denotes the $k$th ``column" of $\mathsf V$. (For example, for any  $\{c_k\}_{k\in\Zn}\subset\Cn$ such that $c_0=1$ we have that $\sF V(\xi_r):\mathcal V(\Lambda)\to \Cn$, $\sF V(\xi_r): = (c_kw_j(\xi_r))$ is an extension of the mode $\sF w(\xi_r)$ defined in \Cref{thm:operator_sols} for $\xi_r=\xi_r(\omega^2)$ which satisfies~\eqref{eqn:discrete_extended}.) As it turns out, to make such an extension unique, it suffices to request  $\mathsf V$ to be discrete holomorphic in the sense of Definition~\ref{def:discrete_holo}. Indeed, the following theorem, whose proof can be found in~\cite{Chern_2019}, establishes the uniqueness of the holomorphic extension of the relevant modes of $\mathcal L_p$, and therefore the uniqueness of the extension of the general solution $\sF v\in E_{\lambda}$, $\lambda = \omega^2$, of~\eqref{eq:disc_Helm}.

\begin{lemma}
\label{thm:extension}
Let ${\sf w}(\xi)=(w_j(\xi))=(\e^{i\xi jh}):\mathbb Z\to \mathbb C$ with $-\pi/h<\real\xi\leq\pi/h$ be such that  $\mathcal L_p[\sf w]=-\lambda \sf w$ for some $\lambda>0$.  Then, there exists a unique discrete holomorphic 
    function ${\sf W}({\xi})=(W_{j,k}(\xi)): \mathcal V(\Lambda) \rightarrow \mathbb{C}$ such that  $\mathcal{L}_p[\sF W_{*,k}]=-\lambda {\sf W}_{*,k}$ for all $k\in\mathbb Z$, and 
    ${\sf W}_{*,0}(\xi) = {\sF w}(\xi)$. Moreover, 
    \begin{equation*}
        W_{j,k}(\xi) = 
        \begin{cases} 
        \displaystyle
        \left(\prod_{m=0}^{k-1}\rho(\sigma_m, \xi, \omega)
        \right)w_{j+k}(\xi), & \text{for}\quad k>0, \\
        w_{j+k}(\xi), & \text{for}\quad k<0
        \end{cases}
    \end{equation*}
    where 
    \begin{equation}\label{eq:rho}
        \rho(\sigma, \xi, \omega) := \frac{2+i\frac{\sigma}{\omega}
        (1-\e^{-i\xi h})}{2+i\frac{\sigma}{\omega}(1-\e^{i\xi h})}.
    \end{equation}
\end{lemma}

The following results establish that all physically meaningful solutions of the discretized time-harmonic wave equation~\eqref{eq:disc_Helm} exhibit geometric decay when restricted to the stretching path. Physically relevant modes refer specifically to non-constant solutions $\mathsf v=(v_j)$, $j\in\Z$, of~\eqref{eq:disc_Helm} that remain bounded within the physical domain, i.e., $|v_j|\leq c$ for all $j\in\Z$, $j<0$, where $c>0$ is a constant. Based on \Cref{thm:operator_sols}, such solutions can be represented as $\mathsf v=\mathsf w(\xi)$, $\xi\in\Xi$, where $\Xi$ is defined as
\begin{equation*}%\label{eq:mode_set}
\Xi = \left\{\xi\in\C: -\frac{\pi}h< \real\xi\leq \frac{\pi}h,\ \imag\xi\leq 0\right\}\setminus\{0\}. 
\end{equation*}
These physically admissible time-harmonic solutions $\mathsf w(\xi)\e^{-i\omega t}$ encompass constant-amplitude left- and right-propagating modes ($\omega\real\xi\neq 0$, $\imag \xi=0$) as well as modes with exponentially growing amplitude ($\imag\xi<0$).

\begin{lemma}
\label{thm:rho_decay}
For every
$\sigma > 0$, $\omega \in \mathbb{R}$, $\omega \not = 0$, and 
$\xi\in\Xi$, we 
have that the geometric decay rate $\rho$, defined in~\eqref{eq:rho}, satisfies $|\rho(\sigma, \xi, \omega)| < 1$.
\end{lemma}

\begin{proof}
    Once again, we refer to \cite{Chern_2019} for the proof in the case when $ \real\xi \not = 0$ and $\imag\xi < 0$, and we only complete it by considering the additional cases:  $0<{\rm sgn}(\omega)\real \xi< \pi/h$, $\imag \xi = 0$, and  $\real \xi = 0$ and $\imag \xi< 0$.  
    
    In the first case, when ${\rm sgn}(\omega)\xi\in(0,\pi/h)$ we define
    $ \varrho( \xi) = |\rho(\sigma, \xi, \omega)|$. A direct calculation of the derivative of $\varrho$ yields
    \begin{align}
        \frac{\de \varrho}{\de \xi}(\xi) = 
        \frac{-(2h\sigma\omega(-\sigma^2 + (\sigma^2+2
        \omega^2)\cos(\xi h)))}{(\sigma^2+2\omega^2-
        \sigma^2\cos(\xi h)+2\sigma\omega 
        \sin(\xi h))^2\sqrt{\frac{\sigma^2+2\omega^2-\sigma^2\cos(\xi h)-2\sigma\omega\sin(\xi h)}{\sigma^2+2\omega^2-\sigma^2\cos(\xi h)+2\sigma\omega\sin(\xi h)}}}. \label{eq:der_rho}      
    \end{align}
    Clearly, the sign of the derivative depends only on the sign of
    \begin{align}
        q( \xi) = \omega (\sigma^2 - (\sigma^2 + 2\omega^2)\cos(\xi h)), \nonumber
    \end{align}
    and the derivative is well-defined since the denominator in~\eqref{eq:der_rho} does not
    vanish for  $\omega \neq 0$ and $\xi \in \mathbb{R}$. We then distinguish two cases, namely:
    \begin{enumerate}
        \item[a)] $\omega > 0$, $ \xi \in (0, \pi/h)$: We have 
         $\varrho( 0) = \varrho(\pi/h) 
        = 1$ and $q(0) < 0$. Since $\cos(\xi h)$ is 
        monotonically decreasing for  $\xi\in(0, \pi/h)$, $q$ is monotonically increasing in that interval. This shows that $\de\varrho/\de\xi$ changes sign at most once over the interval $\xi\in(0,\pi/h)$, which allows us to conclude that
        $ \varrho(\xi) < 1$ for all $\xi \in (0, \pi/h)$. 

        \item[b)] $\omega < 0$, $\xi \in (-\pi/h, 0)$:  We have
         $\varrho( -\pi/h) = \varrho( 0) = 1$, 
        and $q(-\pi/h) < 0$. Since $\cos(\xi h)$ is monotonically increasing 
        in this interval  and $\omega<0$, then so is $q(\xi)$. Therefore, $\varrho(\sigma,\xi,\omega) < 1 $ for 
        $\xi \in (-\pi/h, 0)$.
    \end{enumerate}

    Finally, when $\real \xi = 0$, $\imag \xi < 0$, we set $\alpha = -\imag (\xi) > 0$. Then $\rho$ can be rewritten as 
        \begin{align}
            \rho(\sigma, \xi, \omega) = \frac{2 + i\frac{\sigma}{\omega}(1-
            \e^{-\alpha h})}{2 + i\frac{\sigma}{\omega}(1-\e^{\alpha h})}.
            \nonumber
        \end{align}
        Showing that $|\rho| < 1$ is then equivalent to show that $  |1-\e^{-\alpha h}| < |1-\e^{\alpha h}|$ which for $\alpha > 0$ is equivalent to
        \begin{align}
            1 < \frac{\e^{\alpha h} + \e^{-\alpha h}}{2} = \cosh(\alpha h) \nonumber
        \end{align}
    which is in turn true for all $\alpha > 0$, regardless the sign of $\omega$.
    This concludes the proof.
\end{proof}

The next and final theorem of this section summarizes the theoretical results regarding the discrete PML: 
\begin{theorem}\label{thm:reflectionless} Let $\mathsf u(t) = (u_j(t)) = 
\e^{-i\omega t}\e^{i\xi hj}$  for $h>0$, $\xi \in \Xi$ 
and $\omega \in \mathbb{R} \setminus \{0\}$, be a right-traveling wave that satisfies the semi-discrete wave equation:
\begin{align}
    \frac{\de^2 \mathsf u}{\de t^2} = \mathcal{L}_p[\mathsf u].
    \label{eqn:wave}
\end{align}
Then $\mathsf u(t):\Z\to\C$ can be uniquely extended 
to a discrete holomorphic function $\mathsf U(t):  
\Z^2 \to \mathbb{C}$ whose $k$th-column  satisfies
\begin{align}
    \frac{\de^2 \mathsf U_{*, k}}{\de t^2} = 
    \mathcal{L}_p[\mathsf U_{*, k}].
    \label{eqn:extended_wave}
\end{align}
Moreover, the restriction $U_{0,j}$ of $\mathsf U$ to the discrete complex path~\eqref{eq:complex_path} is given explicitly by
$$U_{0, j}(t) = A(j)u_j(t),\quad j\in\Z,$$
where
\begin{align}
    A(j) = \begin{cases}
        1, & j \leq 0\\
        \prod_{l=0}^{j-1}\rho(\sigma_j, \xi, \omega), & 
        j > 0
    \end{cases} ,\qquad  \rho(s, \xi, \omega) = \frac{2 + \frac{is}{\omega}
    (1-\e^{-i\xi h})}{2 + \frac{is}{\omega}(1-\e^{i\xi h})},
    \nonumber
\end{align}
with $|\rho(s, \xi, \omega)| < 1$ for $s > 0$, and $\omega \not = 0$.
\end{theorem}
\begin{proof} The proof follows directly from \Cref{thm:rho_decay,thm:extension}.
\end{proof}
\begin{remark}\label{rem:reflectionless_property}
Note that the fact that amplitude factor $A$ satisfies $A(j)=1$ for $j\leq 1$, encodes the reflectionless property of the RDPML. It demonstrates that the RDPML solution aligns perfectly with the exact discrete solution on the physical domain grid.
\end{remark}
\begin{remark}\label{rem:afterThrm34}
    Note that equations (\ref{eqn:wave}) and 
    (\ref{eqn:extended_wave}) coincide in the physical domain as the identity $Z_{0, j} = Z_{j, 0}$ holds for all $j<0$. Therefore, in particular, we have that $u_j=U_{0, j}$ for $j<0$. This means that any solution $\mathsf u$
    of the discretized wave equation \eqref{eqn:wave} is uniquely extended 
    to a holomorphic solution $\mathsf U$ of \eqref{eqn:extended_wave} which decays 
    geometrically in the PML region, i.e., when it is evaluated along the stretching path~$Z_{0, j}$.
\end{remark}

Having proven the geometric decay of the PML, in the next section 
we explain how to explicitly derive the equations for this unique 
holomorphic extension.

\section{RDPML equations in one dimension\label{sec:PML_eqs}}
We are now in a position to derive explicit equations for the unique holomorphic extension $\mathsf V$ of the frequency-domain solution $\sF v$ in~\eqref{eq:disc_Helm}, along the discrete complex path~\eqref{eq:complex_path}. To derive the desired PML equations, we start off by writing the difference equation~\eqref{eqn:discrete_extended} for $\sF V$ in terms of its restriction $\tilde{\mathsf v}=(V_{0,j}):\Zn\to \Cn$ to the path $(Z_{0,j}):\Z\to\C$ introduced in~\eqref{eq:complex_path}. 

Since~$\sF V$ is discrete holomorphic, it follows from Definition~\ref{def:discrete_holo} that its complex derivatives $D_{\sF Z} \sF V$ on the faces $Q_{-r+\frac12, j+\frac12}$ and $Q_{r-\frac12, j-\frac12}$, $r=1,\ldots,p$, are respectively given by 
\begin{subequations}\label{eqn:derivatives}
\begin{align}
D_{\sF Z} \sF V\lf(Q_{-r+\frac12,j+\frac12}\rg)&:=i\omega\hat{\Phi}^{(r)}_{j}  = \frac{ V_{-r+1, j+1}- V_{-r, j}}{Z_{-r+1,j+1}-Z_{-r, j}} = \frac{V_{-r+1, j+1} -  V_{-r, j}}{2h + i\frac{\sigma_j h}{\omega}} \label{eqn:derivatives_a}\\
D_{\sF Z} \sF V\lf(Q_{r-\frac12,j-\frac12}\rg)&:=i\omega\hat{\Psi}^{(r)}_{j}   = \frac{V_{r, j}-V_{r-1, j-1}}{Z_{r,j}-Z_{r-1, j-1}} = \frac{V_{r, j} - V_{r-1, j-1}}{2h + i\frac{\sigma_{j-1} h}{\omega}}\label{eqn:derivatives_b}
\end{align}
\end{subequations} for $j\geq 0$. This leads to the definition of the auxiliary discrete complex functions $\hat{\bol\Phi}^{(r)}=
(\hat\Phi^{(r)}_{j}):\Z^+_0\to\C$ and $\hat{\bol\Psi}^{(r)}=(\hat\Psi^{(r)}_{j}):\Z^+_0\to\C$ that, as we will see, play an important role in the derivation of the discrete PML equations. The next proposition presents expressions for the values of $\sF V$ appearing in \eqref{eqn:discrete_extended} in terms of $\tilde {\sF v}$ and $\hat{\boldsymbol \Phi}^{(r)}$ and $\hat {\bol\Psi}^{(r)}$, as well as the discrete PML equations.

\begin{proposition}\label{prop:identities}
Let $\tilde {\sF v}=(\tilde v_j)=(V_{0,j})$ where ${\sf V}=(V_{j,k}):\Z^2\to \C$ satisfies~\eqref{eqn:discrete_extended} and is discrete holomorphic with respect to ${\sf Z}:\Z^2\to \C$ defined in~\eqref{eq:complex_dom}. Then
\begin{subequations}\label{eqn:identities_u_r}
\begin{align}
    V_{-r, j} & = \tilde v_{j-r} - 
    h\sum_{\ell=1}^r\sigma_{j-\ell}\hat \Phi^{(r+1-\ell)}_{j-\ell},
    \label{eqn:identities_u_r:a}\\
    V_{r, j} & = \tilde v_{j+r} + h\sum_{\ell=1}^r 
    \sigma_{j+\ell-1}\hat 
    \Psi^{(r+1-\ell)}_{j+\ell}, \label{eqn:identities_u_r:b}
\end{align}
for $j\in \mathbb Z$ and $r=1,\ldots,p$, where $\sigma_j$ is defined in~\eqref{eq:sigma_def}.
\end{subequations}
Furthermore, $\tilde {\sF v}$, $\hat {\boldsymbol \Phi}^{(r)}=
(\hat\Phi^{(r)}_{j})$, and $\hat{\boldsymbol \Psi}^{(r)}=(\hat\Psi^{(r)}_{j})$ satisfy the following system of equations:
\begin{subequations}\label{eqn:PML1Dfreq}
\begin{align}
\label{eqn:PML1Dfreq:a}
  -\omega^2\tilde  v_j & =  \mathcal L_p[\tilde {\sf v} ]_{j} + h\sum_{r=1}^{p}\sum_{\ell=1}^{r}\left(a_{r} \sigma_{j+\ell-1}\hat{\Psi}^{(r+1-\ell)}_{j+\ell} - a_{-r}\sigma_{j-\ell}\hat{\Phi}^{(r+1-\ell)}_{j-\ell}\right),\\
\label{eqn:PML1Dfreq:b}
i\omega \hat \Phi^{(r)}_{j} & = \frac{\tilde  v_{j-r+2} - \tilde  v_{j-r}}{2h} + 
    \frac{\sigma_j \hat \Phi^{(r)}_{j}+ \sigma_{j-1}\hat \Phi^{(r)}_{j-1}}{2} + 
    \sum_{\ell=1}^{r-1}\frac{\sigma_{j-1-\ell}\hat \Phi^{(r-\ell)}_{j-1-\ell} - 
    \sigma_{j+1-\ell}\hat \Phi^{(r-\ell)}_{j+1-\ell}}{2}, \\
\label{eqn:PML1Dfreq:c}
    i\omega \hat \Psi^{(r)}_{j} & = \frac{\tilde  v_{j+r} - \tilde v_{j+r-2}}{2h} + 
    \frac{\sigma_{j-1}\hat \Psi^{(r)}_{j} + \sigma_{j}\hat \Psi^{(r)}_{j+1}}{2} + 
    \sum_{\ell=1}^{r-1}\frac{\sigma_{j+\ell}\hat \Psi^{(r-\ell)}_{j+\ell+1} - 
    \sigma_{j+\ell-2}\hat \Psi^{(r-\ell)}_{j+\ell-1}}{2},
\end{align}
\end{subequations}
with the difference operator $\mathcal L_p$  defined in~\eqref{eqn:finitedifferenceoperator} and the discrete complex derivatives  $\hat\Phi^{(r)}_{j}$ and $\hat\Psi^{(r)}_{j}$  defined in~\eqref{eqn:derivatives}.
\end{proposition}

\begin{proof}
 We first prove the identity \eqref{eqn:identities_u_r:a} by induction on the index $r=1,\ldots,p$. The proof of identity \eqref{eqn:identities_u_r:b} is analogous.  
For $r=1$, replacing  $j$ by $j-1$ in~\eqref{eqn:derivatives_a} , we have 
\begin{equation*}
    i\omega \hat \Phi^{(1)}_{j-1} = \frac{V_{0, j} -   V_{-1, j-1}}{2h + i\frac{\sigma_{j-1} h}{\omega}} = \frac{ V_{-1, j}-V_{0, j-1}}{\frac{i\sigma_{j-1} h}{\omega}},
\end{equation*} where the last identity followed from  the discrete holomorphicity of $\mathsf V$.
Then, in view of the fact that $V_{0,j-1}= \tilde v_{j-1}$, we obtain
\begin{equation*}
    V_{-1, j}   = V_{0,j-1} - h\sigma_{j-1}\hat \Phi^{(1)}_{j-1} = \tilde v_{j-1} - h\sigma_{j-1}\hat \Phi^{(1)}_{j-1},
\end{equation*} which corresponds to~\eqref{eqn:identities_u_r:a} for $r=1$. 

Assume now that \eqref{eqn:identities_u_r:a}  holds for some $r=m-1$. Once again using the discrete holomorphicity of $\sF V$ and~\eqref{eqn:derivatives_a}, we arrive at 
\begin{align*}
    V_{-m, j}  = V_{-(m-1), j-1} - h\sigma_{j-1}\hat \Phi^{(m)}_{j-1}.
\end{align*}
Then, using the induction hypothesis, we can replace $V_{-(m-1), j-1}$ by its formula in  
~\eqref{eqn:identities_u_r:a} to get
\begin{align*}
V_{-m, j}   = \tilde v_{j-m} - h\sum_{\ell=1}^{m-1}\sigma_{j-1-\ell}\hat 
    \Phi^{(m-\ell)}_{j-1-\ell}- h\sigma_{j-1}\hat \Phi^{(m)}_{j-1}
     = \tilde v_{j-m} - h\sum_{\ell=1}^m \sigma_{j-\ell}\hat \Phi^{(m+1-\ell)}_{j-\ell},
\end{align*}
which proves identity \eqref{eqn:identities_u_r:a}. 

Now, to show  the recursive relation \eqref{eqn:PML1Dfreq:b}  we resort to the definition of $\hat\Phi^{(r)}_{j}$   to obtain 
\begin{equation*}
\begin{split}  \lf(2h+i \frac{\sigma_{j}h}{\omega}\rg)i\omega \hat \Phi^{(r)}_{j} =  V_{-r+1,j+1} - V_{-r,j} 
    =\hspace{6cm}\\ \tilde v_{j-r+2} - \tilde v_{j-r} + h\sum_{\ell=1}^{r-1} \left(\sigma_{j-1-\ell}\hat \Phi^{(r-\ell)}_{j-1-\ell} - 
    \sigma_{j+1-\ell}\hat \Phi_{j+1-\ell}^{(r-\ell)}\right)  + h\sigma_{j-1}\hat \Phi^{(r)}_{j-1},\end{split}
\end{equation*}
where~\eqref{eqn:identities_u_r:a} was applied twice to $ V_{-r+1,j+1}$ and  $V_{-r,j}$. It follows from here that
\begin{align*}
    i\omega \hat \Phi^{(r)}_{j}& = \frac{\tilde v_{j-r+2} - \tilde v_{j-r}}{2h} + 
    \frac{\sigma_j \hat \Phi^{(r)}_{j} + \sigma_{j-1}\hat \Phi^{(r)}_{j-1}}{2} + 
    \sum_{\ell=1}^{r-1}\frac{\sigma_{j-1-\ell}\hat \Phi^{(r-\ell)}_{j-1-\ell} - 
    \sigma_{j+1-\ell}\hat \Phi^{(r-\ell)}_{j+1-\ell}}{2}.
\end{align*}
Similarly,  the relation \eqref{eqn:PML1Dfreq:c} is obtained for the definition of  $\Psi^{(r)}_{j}$ and identity~\eqref{eqn:identities_u_r:b}. 

Finally, to prove \eqref{eqn:PML1Dfreq:a} we replace formulas \eqref{eqn:identities_u_r:a} and \eqref{eqn:identities_u_r:b} in \eqref{eqn:discrete_extended}, obtaining the PML equation
\begin{align*}
    -\omega^2\tilde v_j\!=& \mathcal L_{p}[{\sF V_{*,0}}]_{j} =  \sum_{r=-p}^p a_r V_{r, j}\\
    &\!\!\!\!\!\!\!\!\!\!\!\!\!\!\!\!=  a_0 \tilde v_0 + \sum_{r=1}^{p}a_{r}(\tilde v_{j+r} + h\sum_{l=1}^r \sigma_{j+l-1}\hat 
    \Psi^{(r+1-l)}_{j+l}) + \sum_{r=1}^{p} a_{r} (\tilde v_{j-r} - h\sum_{l=1}^r \sigma_{j-l}\hat    \Phi^{(r+1-l)}_{j-l} ) \\
    &\!\!\!\!\!\!\!\!\!\!\!\!\!\!\!\!=  \sum_{r=-p}^{p}a_{r}\tilde v_{j+r} + h\sum_{r=1}^{p}a_{r}\sum_{\ell=1}^{r} \sigma_{j+\ell-1}\hat{\Psi}^{(r+1-\ell)}_{j+\ell} - h\sum_{r=1}^{p}a_{r}\sum_{\ell=1}^{r}\sigma_{j-\ell}\hat{\Phi}^{(r+1-\ell)}_{j-\ell} \\
    &\!\!\!\!\!\!\!\!\!\!\!\!\!\!\!\!=  \mathcal L_p[\tilde {\sF v}]_{j} + h\sum_{r=1}^{p}\sum_{\ell=1}^{r}a_{r}\left( \sigma_{j+\ell-1}\hat{\Psi}^{(r+1-\ell)}_{j+\ell} - \sigma_{j-\ell}\hat{\Phi}^{(r+1-\ell)}_{j-\ell}\right).
\end{align*}
\end{proof}
We complete the presentation of the PML equations for the one-dimensional wave equation by transforming $\tilde{\sF v}$, $\hat{\Phi}^{(r)}$, and $\hat{\Psi}^{(r)}$, $r=1,\ldots,p$, back to the time domain. Doing so we arrive at the following semi-discrete system of equations for the finite-difference approximation $(\mathsf u)_j=u_j$, $j\in \mathbb Z$, of the wave-equation solution, together with the time-dependent auxiliary variables $\phi_j^{(r)}$ and $\psi_{j}^{(r)}$ for $j\in\Z$ and $r=1,\ldots,p$:
\begin{subequations}\label{eqn:pmltime}
\small
\begin{align}
    \frac{\de^2 u_j}{\de t^2} & = \mathcal L_p[\mathsf u]_{j} + h\sum_{r=1}^{p}\sum_{\ell=1}^{r}a_{r}\left( \sigma_{j+\ell-1}\psi_{j+\ell}^{(r+1-\ell)} - \sigma_{j-\ell} \phi_{j-\ell}^{(r+1-\ell)}\right), \label{eq:pmltime_pde}\\
    \frac{\de \phi_j^{(r)}}{\de t} & =  - 
    \frac{\sigma_j\phi_j^{(r)} + \sigma_{j-1}\phi_{j-1}^{(r)}}{2} - 
    \sum_{\ell=1}^{r-1}\frac{\sigma_{j-1-\ell} \phi_{j-1-\ell}^{(r-\ell)} - 
    \sigma_{j+1-\ell}\phi_{j+1-\ell}^{(r-\ell)}}{2}- \frac{u_{j-r+2} - u_{j-r}}{2h},
    \label{eq:pmltime_phi}\\
    \frac{\de \psi_j^{(r)}}{\de t} & = - \frac{\sigma_{j-1}\psi_j^{(r)} + \sigma_{j}\psi_{j+1}^{(r)}}{2} - 
    \sum_{\ell=1}^{r-1}\frac{\sigma_{j+\ell}\psi_{j+\ell+1}^{(r-\ell)} - 
    \sigma_{j+\ell-2}\psi_{j+\ell-1}^{(r-\ell)}}{2}- \frac{u_{j+r} - u_{j+r-2}}{2h}.
\end{align}
\end{subequations}

Note that the variables $\phi^{(r)}_j$ and $\psi^{(r)}_j$ for $j<0$ do not appear in the equations above because $\sigma_j = 0$ for $j < 0$, indicating that they are only active within the PML domain. Additionally, as the initial conditions for $u_j$ are supported only on the physical domain, the initial conditions for $\phi^{(r)}_j$ and $\psi^{(r)}_j$ must be taken as zero.

\begin{remark}\label{rmk:bidirectional_pml}
    Until now, we have only considered a PML region to the right of the physical domain. However, equations \eqref{eqn:pmltime} still hold for a left-sided PML. For instance, to formulate a layer that absorbs left-traveling waves from some index $j_0 < 0$ backward, it suffices to let $\sigma_j > 0$ for $j < j_0$.
\end{remark}

\section{RDPML equations in two dimensions}\label{sec:two_dims}

This section extends the one-dimensional results presented in the previous sections to two spatial dimensions. The procedure outlined here can be easily generalized to three and higher dimensions.

We start off by considering an equispaced grid $\{\bol{x}_{j, k} := 
(jh, kh): (j, k) \in \mathbb{Z}^2\} \subset \mathbb{R}^2$, $h>0$, on which we define the following 
$(4p+1)$-point symmetric finite difference approximation of the Laplacian:
\begin{align}
    \Delta v\big|_{\nex=\bol{x}_{j, k}} & = \sum_{r=-p}^p a_r\lf\{ v(\nex_{j+r, k}) +v(\nex_{j, k+r})\rg\} + 
    \mathcal{O}(h^{2p}),
    \label{eqn:wave_2d}
\end{align}
as $h\to 0$, where we have assumed that $v \in C^{2p}(\mathbb{R}^2)$, $p \in \mathbb{N}$. We use here the same coefficients
$\{a_r\}_{r=-p}^p$ as in the one-dimensional case (see Sec.~\ref{sec:one_dim}). Upon applying Fourier transform in time to the solution $u$ in \eqref{eqn:wave_1d} we arrive at the following discrete equation for the frequency domain solution~$v$ at the grid nodes $\bol{x}_{j, k}$:
\begin{align}
    \mathcal{L}_p[\sF v] + \omega^2 \sF v = 0, \label{eqn:finitedifferenceoperator_2d}
\end{align}
where abusing the notation we have set
\begin{align}
    \mathcal{L}_p[\sF v]_{j, k} := \sum_{r=-p}^pa_r\{v_{j+r, k}+v_{j, k+r}\}, \quad j, k \in \mathbb{Z},
    \label{eqn:disc_helm_2d}
\end{align}
with $\sF v = \sF v(\omega) \in \mathbb{C}^{\mathbb{Z} \times \mathbb{Z}}$ containing the approximate grid values $v_{j,k}$ of $v(\cdot, \omega)$ in \eqref{eqn:wave_2d}. Following the same procedure as 
in the one-dimensional case, we extend the two-dimensional grid $\{\bol{x}_{j, k}\}$ for $(j, k) \in \mathbb{Z}^2$, to a discrete complex domain $(\Lambda, \sF Z)$, with $\sF Z = (Z_{j, \ell, k, m}):\mathbb{Z}^2 \times \mathbb{Z}^2 \to \mathbb{C}^2$ 
defined by
\begin{equation*}
    Z_{j, \ell, k, m} = \left((j+\ell)h + \frac{ih}{\omega}\sum_{n=0}^{\ell-1}\sigma_n,\quad
    (k+m)h + \frac{ih}{\omega}\sum_{n=0}^{m-1}\sigma_n\right),
    \quad (j, \ell), (k, m) \in \mathbb{Z}^2,
\end{equation*}
where the damping coefficients $\{\sigma_n\}_{n \in \mathbb{Z}}$ are defined in the same way as in 
\eqref{eq:sigma_def}. We proceed by extending $\sF v \in \mathbb{C}^{\mathbb{Z}\times\Z}$ to a function 
$\sF V \in \mathbb{C}^{\mathbb{Z}^2 \times \mathbb{Z}^2}$, and we ask the extension $\sF V$ to be 
holomorphic when restricted to each of the indexes pairs $(j, \ell)$ and $(k, m)$. We then restrict $\sF V$ to the subset
\begin{equation*}%\label{eqn:stretching_path_2d}
    Z_{0, j, 0, k}  = \left(jh + \frac{ih}{\omega}\sum_{n=0}^{j-1}\sigma_n,
    kh + \frac{ih}{\omega}\sum_{n=0}^{k-1}\sigma_n\right), \quad
    j, k \in \mathbb{Z}.  
\end{equation*}
and denote by $\tilde{\sF v}$ such a restriction.

Introducing auxiliary variables for each dimension we define, for $j, k\in \mathbb Z$ and $r=1,\dots,p$, $\hat{\bol \Phi}$ and $\hat{\bol{\Psi}}$ as follows:
\begin{align}
    \hat \Phi_{j, k}^{(x, r)} & = \frac{1}{i \omega} \frac{V_{-r+1, j+1, 0, k} - 
    V_{-r, j, 0, k}}{2h + \frac{i\sigma_j h}{\omega}}, &
    \hat \Psi_{j, k}^{(x, r)} = \frac{1}{i\omega}\frac{V_{r, j, 0, k} - 
    V_{r-1, j-1, 0, k}}{2h + \frac{i\sigma_{j-1}h}{\omega}} \nonumber\\
    \hat \Phi_{j, k}^{(y, r)} & = \frac{1}{i\omega}\frac{V_{0, j, -r+1, k+1} - 
    V_{0, j, -r, k}}{2h + \frac{i\sigma_k h}{\omega}}, & 
    \hat \Psi_{j, k}^{(y, r)} = \frac{1}{i \omega}\frac{V_{0, j, r, k} - 
    V_{0, j, r-1, k-1}}{2h + \frac{i \sigma_{k-1}h}{\omega}}. \nonumber
\end{align}

Next, fixing first the index $k$ and then the index $j$, and following the procedure presented in Proposition~\ref{prop:identities} for the one-dimensional case, we obtain 
%\begin{subequations}
    \begin{align}
    V_{-r, j, 0, k} & = \tilde v_{j-r, k} - h\sum_{\ell=1}^r\sigma_{j-\ell}
    \hat \Phi_{j-\ell, k}^{(x, r+1-\ell)}, & 
    V_{r, j, 0, k} & = \tilde v_{j+r, k} + h\sum_{\ell=1}^r\sigma_{j+\ell-1}
    \hat \Psi_{j+\ell, k}^{(x, r+1-\ell)}. \\
    V_{0, j, -r, k} & = \tilde v_{j, k-r} - h\sum_{\ell=1}^r\sigma_{k-\ell}
    \hat \Phi_{j, k-\ell}^{(y, r+1-\ell)}, & 
    V_{0, j, r, k} & = \tilde v_{j, k+r} + h\sum_{\ell=1}^r\sigma_{k+\ell-1}
    \hat \Psi_{j, k+\ell}^{(y, r+1-\ell)}. 
\end{align}%\label{eq:inden_2d}\end{subequations}

The identities above allow us to get an explicit equation for $\tilde{\sF v}$. Indeed, using the fact that the holomorphic extension $\sf V$ satisfies 
    $\mathcal{L}_p[\sF V_{*, \ell, *, m}] = -\omega^2 \sF V_{*, \ell, *, m}$, for $(\ell,m)\in\Z\times\Z$, it follows 
    that the frequency-domain PML equations for the $\tilde{\sF v}$ can be expressed as:
{\footnotesize
\begin{align}
    -\omega^2 \tilde v_{j, k} & = \mathcal{L}_p[\tilde{\mathsf v}]_{j, k} + \nonumber \\
    & h\sum_{r=1}^p\sum_{\ell=1}^r a_r\left(\sigma_{j+\ell-1}\hat\psi^{(x, r+1-\ell)}_{j+\ell, k} + 
    \sigma_{k+\ell-1}\hat\psi^{(y, r+1-\ell)}_{j, k+\ell} - 
    \sigma_{j-\ell}\hat\phi^{(x, r+1-\ell)}_{j-\ell, k} - \sigma_{k-\ell}\hat\phi^{(y, r+1-\ell)}_{j, k-\ell}
    \right)
    \nonumber \\
    i\omega \hat\phi_{j, k}^{(x, r)} & = 
    \frac{\tilde v_{j-r+2, k} - \tilde v_{j-r, k}}{2h} + 
    \frac{\sigma_j \hat\phi_{j, k}^{(x, r)} + 
    \sigma_{j-1}\hat\phi_{j-1, k}^{(x, r)}}{2} + 
    \sum_{\ell=1}^{r-1}\frac{\sigma_{j-1-\ell}\hat\phi_{j-1-\ell, k}^{(x, r-\ell)} - 
    \sigma_{j+1-\ell}\hat\phi_{j+1-\ell, k}^{(x, r-\ell)}}{2}
    \nonumber\\
    i\omega \hat\psi_{j, k}^{(x, r)} & = 
    \frac{\tilde v_{j+r, k} - \tilde v_{j+r-2, k}}{2h} + 
    \frac{\sigma_{j-1}\hat\psi_{j, k}^{(x, r)} + 
    \sigma_{j}\hat\psi_{j+1, k}^{(x, r)}}{2} +
    \sum_{\ell=1}^{r-1}\frac{\sigma_{j+\ell}\hat\psi_{j+\ell+1, k}^{(x, r-\ell)} - 
    \sigma_{j+\ell-2}\hat\psi_{j+\ell-1, k}^{(x, r-\ell)}}{2}
    \nonumber\\
    i\omega \hat\phi_{j, k}^{(y, r)} & = 
    \frac{\tilde v_{j, k-r+2} - \tilde v_{j, k-r}}{2h} + 
    \frac{\sigma_k \hat\phi_{j, k}^{(y, r)} + 
    \sigma_{k-1}\hat\phi_{j, k-1}^{(y, r)}}{2} + 
    \sum_{\ell=1}^{r-1}\frac{\sigma_{k-1-\ell}\hat\phi_{j, k-1+\ell}^{(y, r-\ell)} - 
    \sigma_{k+1-\ell}\hat\phi_{j, k+1-\ell}^{(y, r-\ell)}}{2}
    \nonumber\\
    i\omega \hat\psi_{j, k}^{(y, r)} & = 
    \frac{\tilde v_{j, k+r} - \tilde v_{j, k+r-2}}{2h} + 
    \frac{\sigma_{k-1}\hat\psi_{j, k}^{(y, r)} + 
    \sigma_{k}\hat\psi_{j, k+1}^{(y, r)}}{2} +
    \sum_{\ell=1}^{r-1}\frac{\sigma_{k+\ell}\hat\psi_{j, k+\ell+1}^{(y, r-\ell)} - 
    \sigma_{k+\ell-2}\hat\psi_{j, k+\ell-1}^{(y, r-\ell)}}{2}.
    \nonumber
\end{align}
}

Transforming $\tilde{\mathsf v}$ back to time domain, we get the following semi-discrete system of equations for the finite difference approximation $(\mathsf u)_{j,k} = u_{j, k}$, $(j, k) \in \mathbb{Z}^2$, of the wave-equation solution, together with auxiliary variables $\phi_{j, k}^{(x, r)}$, $\phi_{j, k}^{(y, r)}$, $\psi_{j, k}^{(x, r)}$ and $\psi_{j, k}^{(y, r)}$ for $(j, k) \in \mathbb{Z}^2$ and $r=1, \dots, p$:

\begin{equation}
    \scriptsize
    \begin{aligned}
    \frac{\de^2 u_{j, k}}{\de t^2} & = \mathcal{L}_p[\mathsf u]_{j, k} + \\
    &h\sum_{r=1}^p\sum_{\ell=1}^r a_r\left(\sigma_{j+\ell-1}\psi^{(x, r+1-\ell)}_{j+\ell, k} + 
    \sigma_{k+\ell-1}\psi^{(y, r+1-\ell)}_{j, k+\ell} - 
    \sigma_{j-\ell}\phi^{(x, r+1-\ell)}_{j-\ell, k} - \sigma_{k-\ell}\phi^{(y, r+1-\ell)}_{j, k-\ell}
    \right),
     \\
    -\frac{\de \phi_{j, k}^{(x, r)}}{\de t} & = 
    \frac{u_{j-r+2, k} - u_{j-r, k}}{2h} + 
    \frac{\sigma_j \phi_{j, k}^{(x, r)} + 
    \sigma_{j-1}\phi_{j-1, k}^{(x, r)}}{2} + 
    \sum_{\ell=1}^{r-1}\frac{\sigma_{j-1-\ell}\phi_{j-1-\ell, k}^{(x, r-\ell)} - 
    \sigma_{j+1-\ell}\phi_{j+1-\ell, k}^{(x, r-\ell)}}{2},
    \\
    -\frac{\de \psi_{j, k}^{(x, r)}}{\de t} & = 
    \frac{u_{j+r, k} - u_{j+r-2, k}}{2h} + 
    \frac{\sigma_{j-1}\psi_{j, k}^{(x, r)} + 
    \sigma_{j}\psi_{j+1, k}^{(x, r)}}{2} +
    \sum_{\ell=1}^{r-1}\frac{\sigma_{j+\ell}\psi_{j+\ell+1, k}^{(x, r-\ell)} - 
    \sigma_{j+\ell-2}\psi_{j+\ell-1, k}^{(x, r-\ell)}}{2},
    \\
    -\frac{\de \phi_{j, k}^{(y, r)}}{\de t} & = 
    \frac{u_{j, k-r+2} - u_{j, k-r}}{2h} + 
    \frac{\sigma_k \phi_{j, k}^{(y, r)} + 
    \sigma_{k-1}\phi_{j, k-1}^{(y, r)}}{2} + 
    \sum_{\ell=1}^{r-1}\frac{\sigma_{k-1-\ell}\phi_{j, k-1+\ell}^{(y, r-\ell)} - 
    \sigma_{k+1-\ell}\phi_{j, k+1-\ell}^{(y, r-\ell)}}{2},
    \\
    -\frac{\de \psi_{j, k}^{(y, r)}}{\de t} & = 
    \frac{u_{j, k+r} - u_{j, k+r-2}}{2h} + 
    \frac{\sigma_{k-1}\psi_{j, k}^{(y, r)} + 
    \sigma_{k}\psi_{j, k+1}^{(y, r)}}{2} +
    \sum_{\ell=1}^{r-1}\frac{\sigma_{k+\ell}\psi_{j, k+\ell+1}^{(y, r-\ell)} - 
    \sigma_{k+\ell-2}\psi_{j, k+\ell-1}^{(y, r-\ell)}}{2}.    
\end{aligned}\label{eq:RDPML_2D}
\end{equation}

\begin{remark} Even though we developed the PML equations for a layer in the positive indexes,  \Cref{rmk:bidirectional_pml} is extended to two or more dimensions. The only change needed for a PML in the negative indexes is to let $\sigma_j > 0$ from the desired index $j_0 < 0$ backward.
\end{remark}

\section{Frequency-domain reduced system}\label{sec:linear_system}

The solution of the wave equation in the frequency domain, commonly referred to as the Helmholtz equation, holds significant relevance in numerous application areas. Standard formulations of PMLs for the Helmholtz equation typically do not incorporate auxiliary variables. 
Instead, in PML formulations for the Helmholtz equation, the problem is typically recast as a single scalar second-order PDE. This formulation introduces a term of the form $\operatorname{div} A\nabla$, where the matrix $A\in\mathbb{C}^{d\times d}$ encodes the PML's effect on the equation. Upon discretization using an FD scheme, the resulting linear system exhibits sparsity. This indicates that the application of the standard PML does not impact the sparsity character of the resulting linear system. In this section, we aim to investigate the sparsity of the system matrices that arise from our high-order RDPML schemes for solving the Helmholtz equation. For the sake of conciseness, we focus on the one-dimensional wave equation in our analysis. However, it is important to note that the findings can be extrapolated to higher dimensions as well.

We start off by deriving expressions for the auxiliary variables $\hat{\bol\Phi}^{(r)}$ and ${\hat{\bol\Psi}^{(r)}}$ in \eqref{eqn:derivatives} in terms of the discrete holomorphic function $\tilde{\mathsf v}$. The linear system~\eqref{eqn:PML1Dfreq} for the time-harmonic 
variables can be expressed in a more concise way as:
\begin{subequations}
    \begin{align}
    (\omega^2\mathcal{I} + \mathcal{L}_p)\tilde{\mathsf v} + 
    \sum_{r=1}^p \mathcal{A}_r\hat{\bol\Phi}^{(r)} + 
    \sum_{r=1}^p\mathcal{B}_r\hat{\bol\Psi}^{(r)} & = \mathbf 0,  \\
    \mathcal{D}_1^{(r)}\tilde{\mathsf v} + \mathcal{C}_1\hat{\bol\Phi}^{(r)} + 
    \sum_{\ell=1}^{r-1}\mathcal{E}_{r, \ell}\hat{\bol\Phi}^{(\ell)} & = \mathbf 0, \quad 
    r =1,\ldots, p, \label{eqn:system_1} \\
    \mathcal{D}_2^{(r)}\tilde{\mathsf v} + \mathcal{C}_2\hat{\bol\Psi}^{(r)} + 
    \sum_{\ell=1}^{r-1}\mathcal{F}_{r, \ell}\hat{\bol\Psi}^{(\ell)} & = \mathbf 0, 
    \quad r=1,\ldots, p, \label{eqn:system_2}
\end{align}\label{eq:system_1}\end{subequations}
where $\mathcal I$ and $\mathcal L_p$ denote the identity operator and the 
FD operator \eqref{eqn:finitedifferenceoperator}, respectively. The coefficients that define  the remaining operators $\mathcal A_r$, $\mathcal B_r$, $\mathcal D_1^{(r)}$, $\mathcal D_2^{(r)}$, $\mathcal C_1$, $\mathcal C_2$, $\mathcal E_{r,l}$, and $\mathcal F_{r,l}$ can be inferred from~\eqref{eqn:PML1Dfreq}. In particular, $\mathcal{C}_1$ and $\mathcal{C}_2$ are lower and upper bidiagonal and have the following non-zero entries
\begin{equation*}
(\mathcal{C}_1)_{j, j}  = -i\omega + \frac{\sigma_j}{2},\quad (\mathcal{C}_1)_{j, j-1}  = \frac{\sigma_{j-1}}{2},  \quad 
(\mathcal{C}_2)_{j, j}  = -i\omega + \frac{\sigma_{j-1}}{2},  \quad
(\mathcal{C}_2)_{j, j+1}  = \frac{\sigma_{j}}{2},
\end{equation*} 
for $j \geq 0$. We also note that, for $\omega \in \mathbb{R} \setminus \{0\}$, the inverses of these operators are lower and upper triangular with non-zero entries given explicitly by: 
\begin{equation}
    \begin{aligned}
    (\mathcal{C}_1^{-1})_{j, k} & = 
    \frac{2(-1)^{j+k}}{(-2i\omega + \sigma_j)}\prod_{l=k}^{j-1}
    \frac{\sigma_{l}}{(-2i\omega + \sigma_l)}, & \text{if } k \leq j,
     \\
    (\mathcal{C}_2^{-1})_{j, k} & = 
        \frac{2(-1)^{j+k}}{(-2i\omega + \sigma_{j-1})}
        \prod_{l=j}^{k-1}\frac{\sigma_{l}}{(-2i\omega + \sigma_{l})}, & 
        \text{if } k \geq j.
\end{aligned}\label{eqn:c_invs}
\end{equation}

In order to express~\eqref{eq:system_1} in a further compact form, we introduce the vectors $\hat{\bol\Phi} = (\hat{\bol\Phi}^{(1)}, ..., \hat{\bol\Phi}^{(p)})^\top$, 
$\hat{\bol\Psi} = (\hat{\bol\Psi}^{(p)}, ..., \hat{\bol\Psi}^{(1)})^\top$; the inverted order of the variables in $\hat{\mathbf{\Psi}}$ is solely for the purpose of simplifying the calculations that follow. Additionally, we define the block operators:
\begin{equation*}
\mathcal{E}  = 
    \begin{bmatrix}
        \mathcal{C}_1 & 0 & \dots & 0\\
        \mathcal{E}_{2, 1} & \mathcal{C}_1 &  \ddots & \vdots\\
        \vdots & \ddots & \ddots & 0\\
        \mathcal{E}_{p, 1} &  \dots & 
        \mathcal{E}_{p, p-1} & \mathcal{C}_1
    \end{bmatrix}, \quad
\mathcal{F}  = 
    \begin{bmatrix}
        \mathcal{C}_2 & \mathcal{F}_{p, p-1} &  \dots & \mathcal{F}_{p, 1}\\
        0             & \mathcal{C}_2        & \ddots & \vdots            \\
        \vdots        & \ddots               & \ddots & \mathcal{F}_{2, 1}\\
        0             & \dots                & 0      & \mathcal{C}_2
    \end{bmatrix}.
\end{equation*}

With the aid of these definitions, we can write~\eqref{eq:system_1} as
\begin{equation}
    \begin{aligned}
    (\omega^2\mathcal{I} + \mathcal L_p)\tilde{\mathsf v} + \mathcal{A}\hat{\bol\Phi} + 
    {B}\hat{\bol\Psi} & = \bol 0, \\
    \mathcal{D}_1\tilde{\mathsf v} + \mathcal{E}\hat{\bol\Phi} & = \bol 0,  \\
    \mathcal{D}_2\tilde{\mathsf v} + \mathcal{F}\hat{\bol\Psi} & = \bol 0, 
    \end{aligned}\label{eq:system_2}
\end{equation}
where $\mathcal A= [\mathcal A_1,\mathcal A_2,\ldots,\mathcal A_p],\quad\mathcal B= [\mathcal B_1,\mathcal B_2,\ldots,\mathcal B_p]$, $\quad  \mathcal{D}_1 =[\mathcal{D}_1^{(1)},\mathcal{D}_1^{(2)},\dots,\mathcal{D}_1^{(p)}]^\top$, and  $\mathcal{D}_2 = [\mathcal{D}_2^{(p)}, \mathcal{D}_2^{(p-1)},\dots, \mathcal{D}_2^{(1)}]^\top$.

It hence follows that $\hat{\bol\Phi}$ and $\hat{\bol\Psi}$ can be eliminated from~\eqref{eq:system_2}  provided $\mathcal{E}$ and $\mathcal{F}$ are invertible, and they are since $\mathcal C_1$ and $\mathcal C_2$ are so. Therefore, we arrive at the following reduced system linear system for $\tilde{\mathsf v}$:
\begin{align}
    (\omega^2 \mathcal{I} + \mathcal L_p)\tilde{\mathsf v} - 
    \mathcal{A}(\mathcal{E}^{-1}\mathcal{D}_1)\tilde{\mathsf v} - 
    \mathcal{B}(\mathcal{F}^{-1}\mathcal{D}_2)\tilde{\mathsf v} & = \bol 0.\label{eq:reduced_Helm}
\end{align}
Furthermore, we note that $\mathcal{E}$ and $\mathcal{F}$ can be easily inverted by forward and backward substitution, respectively, using formulae~\eqref{eqn:c_invs}. This shows that, as standard PML methods, the frequency-domain problem can be reduced to a single equation for the holomorphically extended PDE solution restricted to the stretching path. Unfortunately, unlike the full system that includes the auxiliary variables, the operators $\mathcal A( \mathcal E^{-1} \mathcal D_1)$ and $\mathcal B(\mathcal F^{-1}\mathcal D_2)$ in~\eqref{eq:reduced_Helm} give rise to a dense asymmetric matrix block when, in practice, the  PML is terminated by any suitable boundary condition.

To study the structure of the matrices stemming from~\eqref{eq:reduced_Helm} in more detail, we look into the sparsity patterns of the actual system matrices that arise from truncating the PML domain by means of periodic boundary conditions, as explained in greater detail in \cref{sec:numerical_results_1d} below. For the fourth-order ($p=2$) RDPML scheme, we consider both the full system derived from~\eqref{eq:system_2}, which includes the auxiliary variables, as well as the reduced linear system obtained from~\eqref{eq:reduced_Helm}. By ordering the unknowns as $(\tilde{\mathsf v}, \hat{\bol\Phi}^{(1)}, \hat{\bol\Phi}^{(2)}, \hat{\bol\Psi}^{(1)}, \hat{\bol\Psi}^{(2)})$ in the full-system case, we observe that the resulting linear system exhibits a sparse structure. Specifically, the system consists of band-limited matrix blocks, as shown in \Cref{fig:3a}. Conversely, in the reduced-system case, we observe a different matrix structure. Specifically, the resulting system matrix includes a dense block $(A_{22})$ associated with the $\tilde{\mathsf v}$-values corresponding to the nodes located within the PML domain. This dense block can be visualized in \Cref{fig:3b}. The presence of this dense block indicates that the reduced linear system requires careful consideration and efficient computational techniques to handle the increased complexity in solving for the $\tilde{\mathsf v}$-values within the PML region. Naturally, the matrix block ($A_{11}$) associated with the $\tilde{\mathsf v}$-values at the nodes located within the physical domain remain the same in both cases. 

Despite the smaller dimensions of the system matrix in the reduced linear system, its localized dense structure can have a significant impact on the efficiency of solving the linear system. Initial numerical experiments indicate that employing sparse direct solvers to solve the full system, with its sparsity pattern, can be faster compared to using a Schur-complement approach and sparse direct solvers to solve the reduced linear system. This suggests that the computational cost associated with the dense block in the reduced system outweighs the benefits gained from its reduced size. Finally, we note that we also explored the use of iterative solvers, such as CGS and GMRES, with standard 
preconditioning strategies such as SGS, ILU, and Jacobi. However, after conducting tests, we found that the most efficient approach for solving the system was still utilizing a sparse solver over the full system. 

\begin{figure}[ht!]
    \centering
    \subfloat[]{\includegraphics[scale=0.4, trim={2cm 0cm 0cm 0cm}]{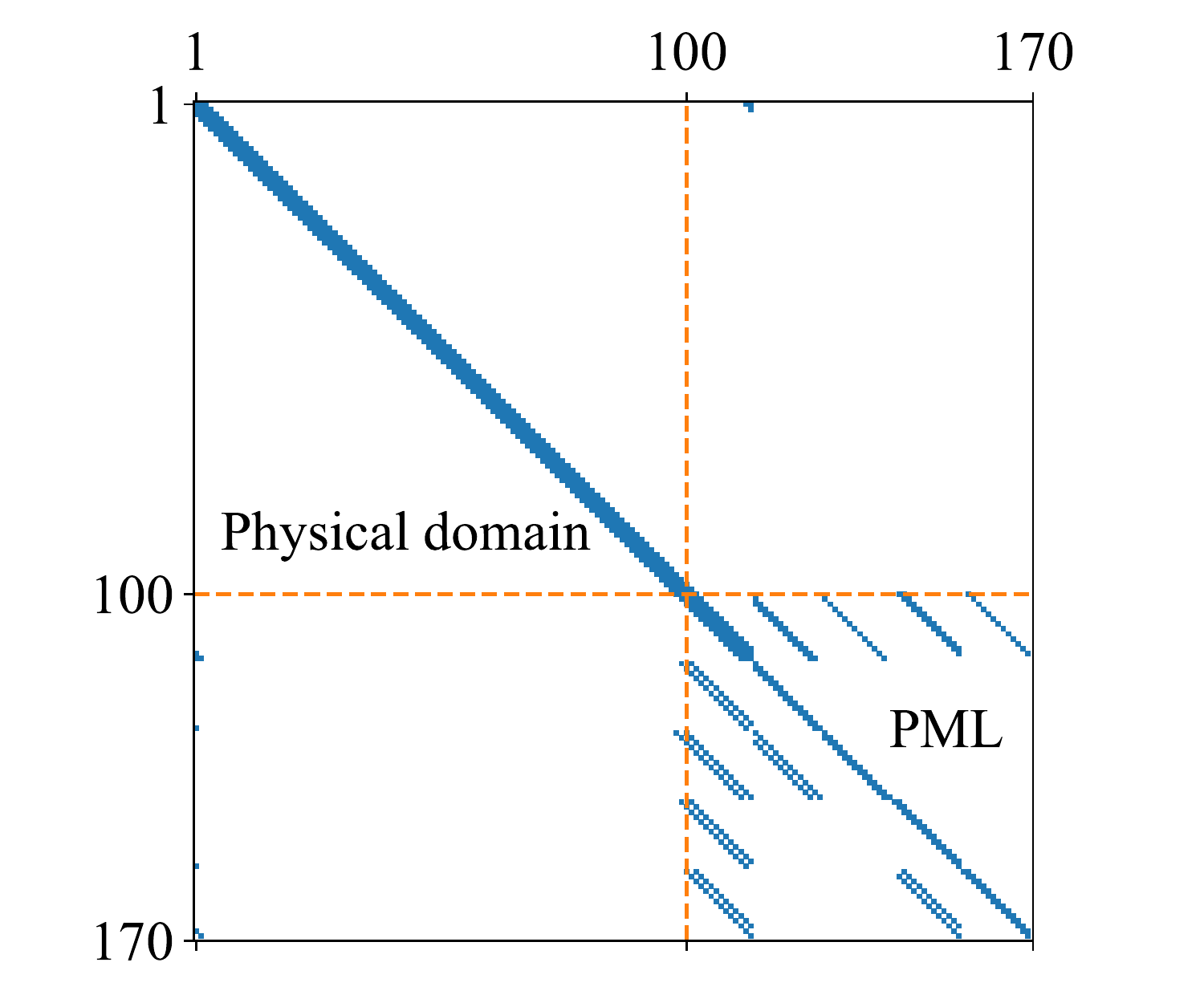}\label{fig:3a}}
    \subfloat[]{\includegraphics[scale=0.4]{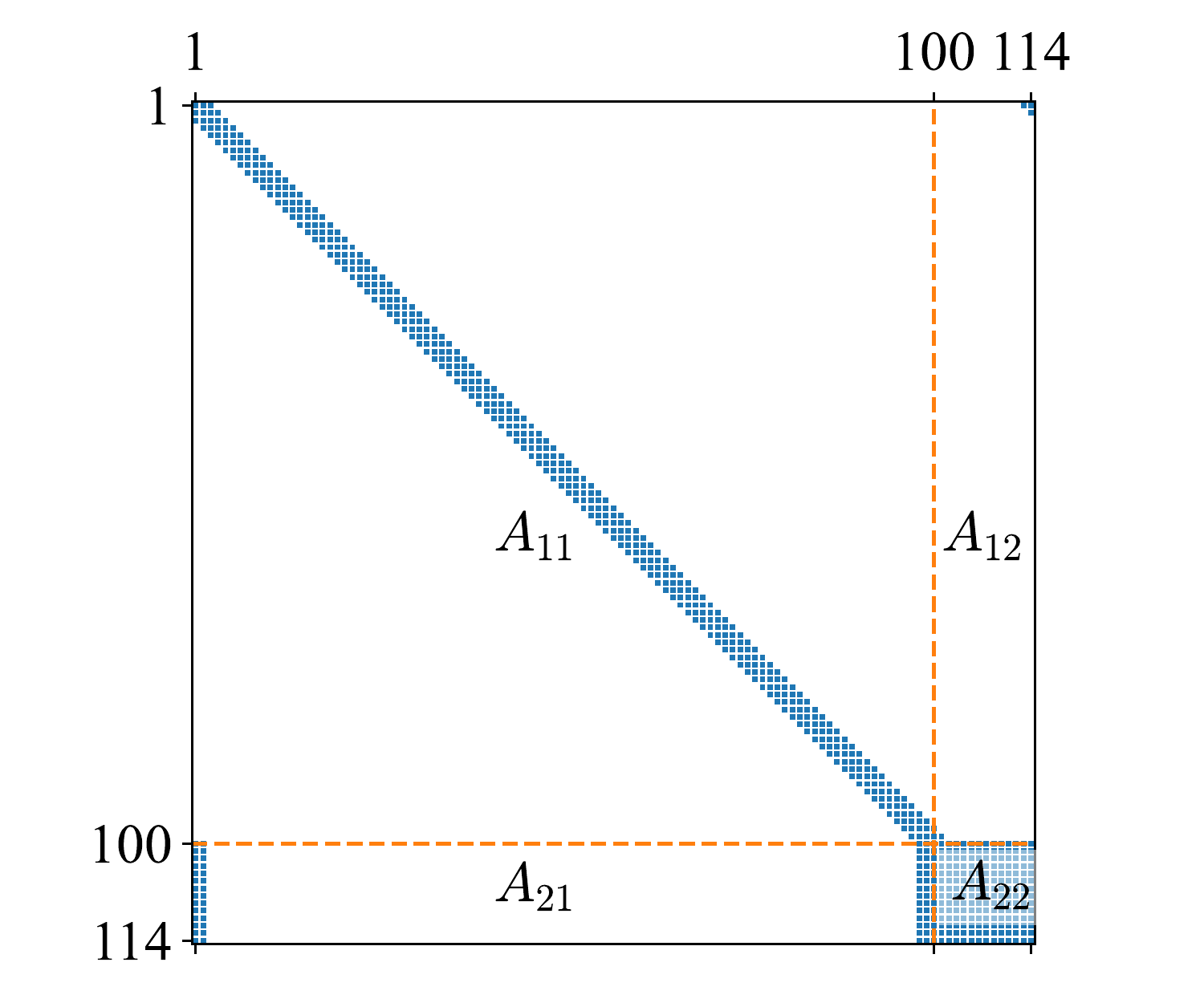}\label{fig:3b}}
    \caption{Comparison of the sparsity patterns of the full (a) and reduced (b) 
    linear systems matrices resulting from the fourth-order RDPML scheme for the 1D  time-harmonic wave (Helmholtz) equation. In 
    (b) the $A_{11}$ matrix block is associated with the nodes in the physical domain, while the dense 
    $A_{22}$ block corresponds to the nodes in the PML domain.}
    \label{fig:linear_system}
\end{figure}

\section{Numerical results}\label{sec:numerical_results}
This section introduces a range of numerical experiments designed to validate and demonstrate the effectiveness of the proposed methodology in the context of solving the wave equation in one- and two-spatial dimensions.

For the sake of clarity, unless stated otherwise, in the following sections, we employ the damping function $\sigma:\Z\to\R$:
$$
\sigma = 
\begin{cases} 
0 & \text{if } j<0,\\
\displaystyle 2/h & \text{if } j\geq 0,
\end{cases}
$$
in constructing the RDPML discrete stretching path. The constant value $2/h$ was determined by approximately minimizing the decay rate $\varrho(\sigma) =|\rho(\sigma,\xi,\omega)|$ in the case when the wavenumber satisfies the exact dispersion relation $\xi=\omega$ (refer to the supplementary material~\ref{sm:opti_sigma} for details).

\subsection{One-dimensional wave equation}\label{sec:numerical_results_1d}

We first set out to validate both the reflectionless and high-order characteristics of the proposed FD schemes in a one-dimensional spatial domain problem. To achieve this goal, we present two numerical experiments based on the following test problem. At the continuous level we consider the wave equation~\eqref{eqn:wave_1d} for $u:\R\times\R_+\mapsto\R$ with initial data
\begin{align}
    g_0(x) = \begin{cases}
        \e^{-10(x+3)^2}, & \text{if } |x+3| \leq 2,\\
        0, & \text{if } |x+3| > 2,
    \end{cases} \nonumber
\end{align}
and $g_1(x)=0$ for $x\in\R$.
It is worth noting that for numerical purposes, we can consider the truncated Gaussian function~$g_0$  used in the initial condition to be globally smooth. This assumption is valid since its jump discontinuity has a height that is in the same order as machine precision. 

As is well known~\cite{evans2022partial}, this initial value problem admits an exact solution given by D'Alembert's formula:
\begin{equation}
u(x,t) = \frac{g_0(x-t) + g_0(x+t)}{2},\quad t>0,\ x\in\R,
\label{eq:exat_sol}\end{equation}
which will be used as a reference solution to assess the numerical errors resulting from the proposed schemes in our first set of experiments. The constant damping coefficient $\sigma_j = 2/\sqrt{h}$ in the definition of the stretching path~\eqref{eq:complex_path} is used in all the numerical experiments presented in this section.
 
In order to numerically solve our test problem, we begin by confining the original unbounded spatial domain $\mathbb{R}$ to a finite interval $\Omega_0=(-6,0)$ and we augment this bounded domain with a finite-length PML domain, denoted as $\Omega_{\rm PML}=(0,4)$. To discretize the resulting computational domain $\Omega=(-6,4)$, we employ a uniform mesh ${x_j=-6+jh,\ j=0,\ldots,N}$, where $h=10/N$ and $N\in\mathbb{N}$.  Following~\cite{Chern_2019}, at the boundary of the computational domain $\Omega$, i.e., at $\p\Omega=\{-6,4\}$, we numerically impose periodic boundary conditions on the discrete variables $u_j$, $\phi_j^{(r)}$, and $\psi_j^{(r)}$ that make up the one-dimensional RDPML schemes~\eqref{eqn:pmltime}. In detail, the periodic boundary conditions used are given by:
{\begin{equation}
    \begin{aligned}
u_{N+j}=&~u_{j},&&\quad j=0,\ldots,p-1,\\
 \phi^{(r)}_{N+j}=&~\phi^{(r)}_{j}, \quad\psi^{(r)}_{N+j}=\psi^{(r)}_{j},&&\quad j=0,\ldots,p-r,\quad r=1,2,\ldots,p,
\end{aligned}\label{eq:per_bc_1d}
\end{equation}}
 where $u_{N+j}$, $\phi^{(r)}_{N+j}$, and $\psi^{(r)}_{N+j}$ for  $j\geq 1$ and $p\geq 2$ can be interpreted as the wave-equation solution and auxiliary function values at virtual nodes $x_{N+j}=(N+j)h$, $j=1,\ldots,p-1$ lying outside the computational domain $\Omega$.   
 
In order to demonstrate the attainment of the desired $2p$ order of accuracy by the proposed schemes, for $p=1,2,3$ and $4$ we impose the periodic boundary conditions~\eqref{eq:per_bc_1d} and evolve the system~\eqref{eqn:pmltime}  for different spatial discretization sizes $h>0$ by means of an explicit Runge-Kutta (RK) time integrator of order eight~\cite{Fehlberg1968} using a time step of $h/8$. This choice ensures a sufficiently accurate time evolution during the computations. Figure \ref{fig:conv_1d} displays the numerical errors obtained in these examples for grid sizes $h=2^{-n}$, $n=3,\ldots,7$. The error is measured as
\begin{equation}\label{eq:discrete_error}
E_h := \max_{t_\ell=h\ell/8\in[ 0,10]}\ \max_{x_j\in\overline\Omega_0} |u_j(t_\ell)-u(x_j,t_\ell)|,\end{equation}
where $t_\ell$ are the discrete times used by the time integrator, $u$ is the exact solution~\eqref{eq:exat_sol}, and $u_j$ the resulting RDPML solution.
 Clearly, the RDPML solutions achieve the expected $2p$ orders of convergence as $h$ decreases. Furthermore, these results serve as a first validation of the proposed RDPML equations~\eqref{eqn:pmltime}, as waves first hit the PML boundary at $x=0$ around $t=1$ and the error~\eqref{eq:discrete_error} accounts for all discrete times between $t=0$ and $t=10$.

\begin{figure}[ht!]
    \centering
    \includegraphics[scale=0.45]{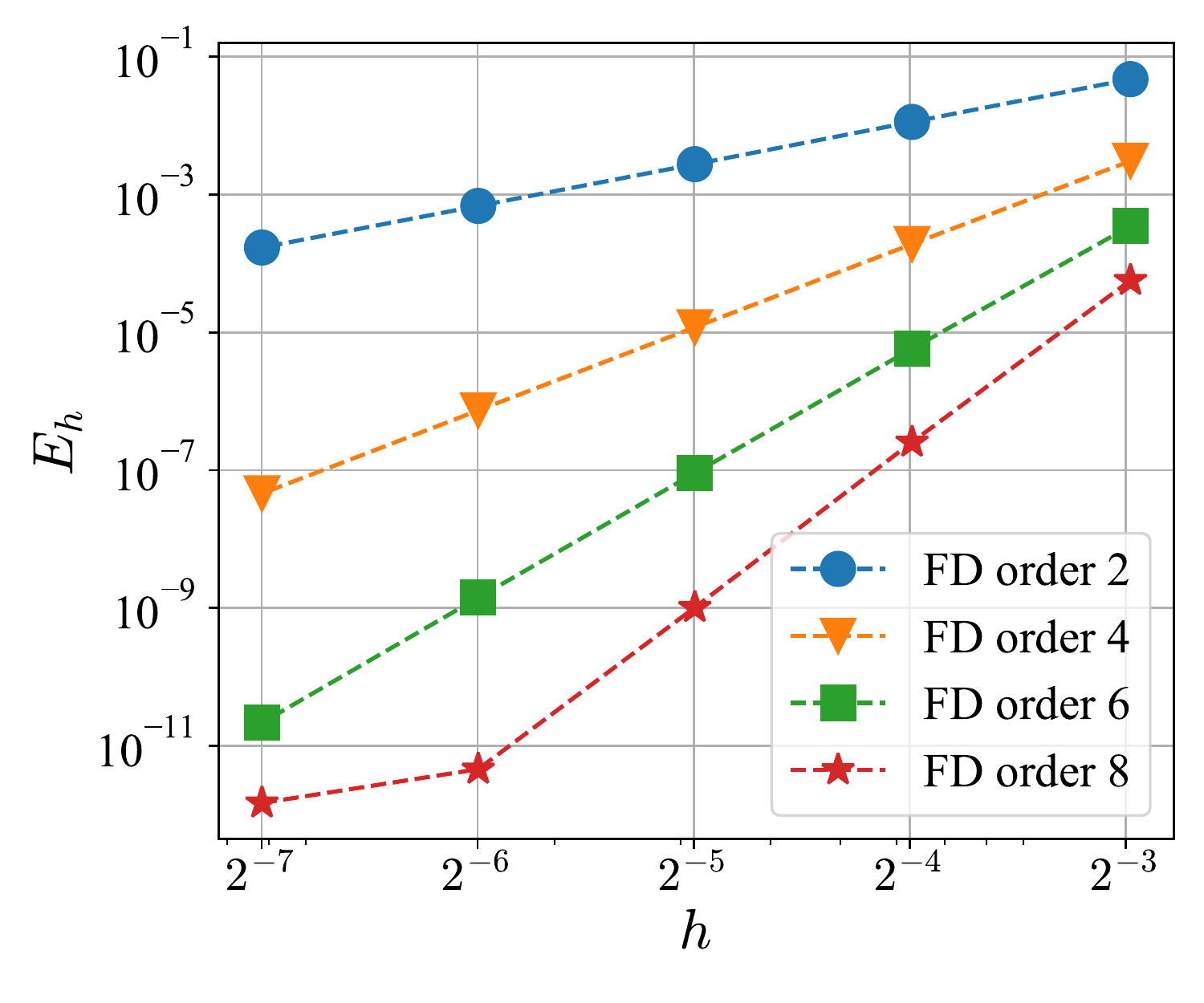}
    \caption{History of convergence of the approximate solutions using the RDPML schemes of order 2, 4, 6, and 8 for the one-dimensional wave equation. The errors are computed using the formula \eqref{eq:discrete_error}, and the marks in the lines correspond to the values of $h$ used in the computations.}
    \label{fig:conv_1d}
\end{figure}

%Unfortunately, the results showcased in Figure~\ref{fig:conv_1d} fall short of providing complete verification of the proposed schemes' reflectionless property. In fact, these results only serve to illustrate that the numerical reflection errors are comparable to the errors introduced by the finite difference discretization of the second-order derivative operator. To actually 
To verify the reflectionless property, we instead consider another experiment where for a fixed discretization size $h = 2^{-6}$, we compare the RDPML solutions of order $2p=2,4,6$ and $8$ with the corresponding numerical solutions $u^{\rm ref}_j$ of the wave equation obtained using the standard FD scheme of order $2p$ over the larger domain $\Omega_{\rm ref} = (-11,5)$. The grids of $\Omega_{\rm ref}$ and $\Omega$ are constructed in such a way that their nodes lying within $\Omega_0$ coincide. The motivation behind this experiment stems from the fact that the solutions $u_j$ and $u^{\rm ref}_j$ obtained using the RDPML method and the standard FD method, respectively, coincide at the nodes located within $\Omega_0$ for all time values in the range $t\in[0,1]$, i.e., before the waves reach the boundary points of the physical domain $\p\Omega_0=\{-6,0\}$. However, after this time interval, the artificial numerical reflections introduced by the PML are expected to propagate into the physical domain $\Omega_0$, allowing for their measurement. To quantify these reflections, we then introduce the following time-dependent error measure:
\begin{equation}\label{eq:discrete_error_2}
E_{\rm 1D}(t) :=\max_{x_j\in\overline\Omega_0} |u_j(t)-u^{\rm ref}_j(t)|.
\end{equation}

\Cref{fig:6a} displays the time evolution of this error measure at the discrete time instances obtained using the eight-order RK time integrator and spatial discretizations of order $2p=2,4,6,$ and $8$. The time $t=1$ is marked in this figure, which corresponds to the arrival time of the wave at the PML. We used here a large PML domain given by $\Omega_{\rm PML} = (0,10)$. The near-machine precision errors shown in this figure provide strong evidence of the effectiveness of the proposed high-order RDPML schemes and demonstrate its capability to accurately truncate the infinite domain while minimizing numerical reflection errors to an exceptional degree. This experiment, however, does not account for the effect of the (periodic) boundary conditions on the error $E_{\rm 1D}$. To take them into account, we instead consider the shorter PML domain $\Omega_{\rm PML} = (0,4)$ and we measure $E_{1D}$ over the same time interval $[0,10]$. \Cref{fig:6b} displays the evolution of $E_{\rm 1D}$ in this case, where the dashed line (marked in red) at $t=5$   indicates the round-trip time at which, due to the boundary condition used, the waves attenuated within the PML re-enter the physical domain $\Omega_0$. The fact that no differences can be observed between the two figures suggests that the boundary condition used has a negligible effect on the performance of the proposed RDPML. Similar, results have been observed using a two-sided PML domain in conjunction with standard homogeneous Dirichlet boundary conditions.

\begin{figure}[ht]
    \centering
    \subfloat[]{\includegraphics[scale=0.45, trim={2cm 0cm 0cm 0cm}]{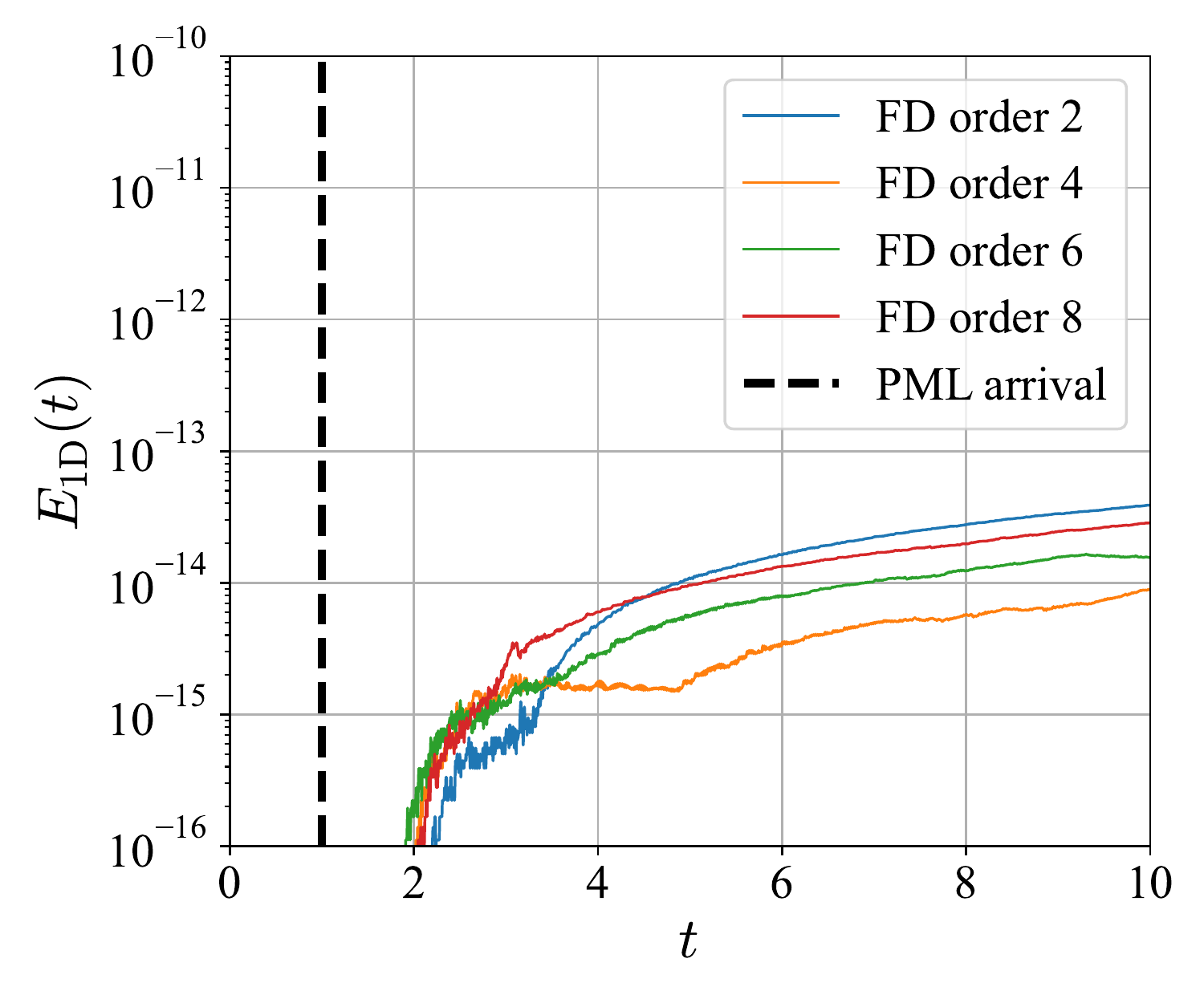}\label{fig:6a}}    
    \subfloat[]{\includegraphics[scale=0.45]{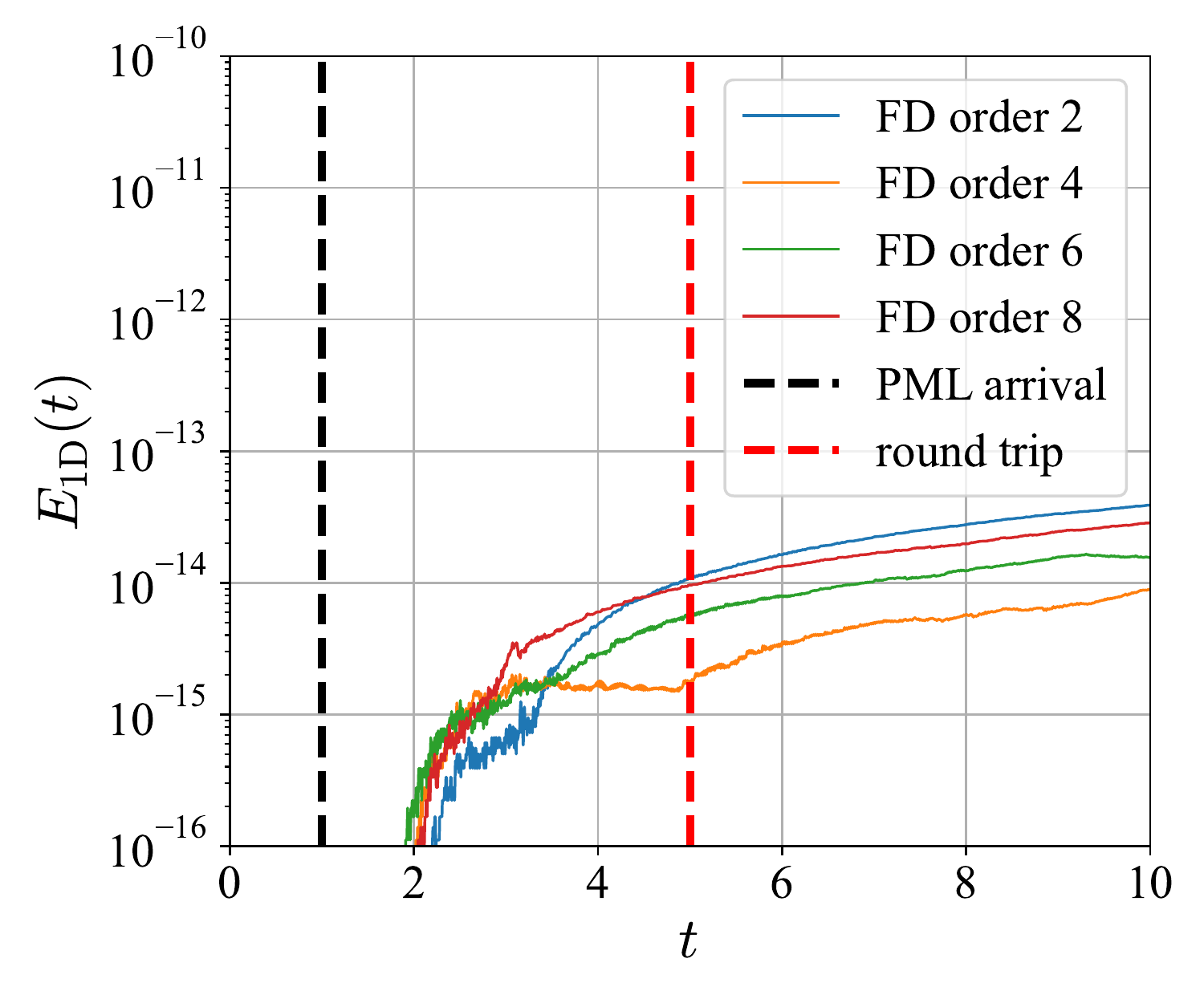}\label{fig:6b}}    
    \caption{Time evolution of the numerical error $E_{\rm 1D}$~\eqref{eq:discrete_error_2} that measures the spurious reflections introduced by the proposed RDPML schemes, for FD spatial discretizations of order 2, 4, 6, and 8. (a) A PML domain of length 10 is used for which waves do return to the physical domain during the time interval $[0,10]$ consider in the plot. (b) A shorter PML domain of length 4 is used, for which waves return to the physical domain at $t=5$ due to the periodic boundary conditions~\eqref{eq:per_bc_1d} employed.}
    \label{fig:reflection_1d}
\end{figure}

\subsection{Two-dimensional wave equation}\label{sec:2d_wave}
This section presents numerical experiments that serve to validate the effectiveness of the proposed high-order RDPML schemes in solving the wave equation in two spatial dimensions. In the lack of a directly available exact wave equation solution, we adopt an analogous experiment to the one shown in \Cref{fig:reflection_1d} for the one-dimensional wave equation.

We here consider the solution $u:\R^2\times\R_+\mapsto\R$ of~\eqref{eqn:wave_1d} for $d=2$ with initial data

% \begin{equation}    
%     g_0(\nex) = \begin{cases}
%         \e^{-140|\nex-\nex_0|^2}, & \text{if } |\nex-\nex_0| \leq 0.5,\\
%         0, & \text{if } |\nex-\nex_0| > 0.5,
%     \end{cases}\quad\nex_0=(-1,-1)\in \mathbb R^{2}, \label{eq:2d_bump}
% \end{equation}
\begin{equation}    
    g_0(\nex) = \begin{cases}
        \e^{-140|\nex-\nex_0|^2}, & \text{if } |\nex-\nex_0| \leq 0.5,\\
        0, & \text{if } |\nex-\nex_0| > 0.5,
    \end{cases}\quad{\nex_0=(-2,-2)}\in \mathbb R^{2}, \label{eq:2d_bump}
\end{equation}
% and $g_1(\nex) =0$ for $\nex\in\R^2$. As in the one-dimensional experiment presented above, we  approximate~$u$ using RDPML schemes~\eqref{eq:RDPML_2D} of orders $2p=2,4,6$ and $8$, over a bounded physical domain given in this case by $\Omega_{0} =(-2,0)\times(-2,0)$, which has attached a finite-size L-shaped PML region given by $\Omega_{\rm PML} = (0,0.36)\times(-2,0.36)\cup (-2,0.36)\times(0,0.36)$. The physical domain is discretized by uniform grids given by $\{\nex_{j,k} =(-2+hj,-2+hk), j,k=0,\ldots, N\}$, where $h = 2/N$, $N\in\mathbb N$. At the boundary of the computational domain $\Omega=(-2,0.36)\times(-2,0.36)$, which encompasses both  $\Omega_0$ and $\Omega_{\rm PML}$, we apply the following periodic boundary conditions:
and $g_1(\nex) =0$ for $\nex\in\R^2$. As in the one-dimensional experiment presented above, we  approximate~$u$ using RDPML schemes~\eqref{eq:RDPML_2D} of orders $2p=2,4,6$ and $8$, over a bounded physical domain given in this case by $\Omega_{0} = {(-4,0)\times(-4,0)}$, which has attached a finite-size L-shaped PML region given by $\Omega_{\rm PML} = {(0,0.2)\times(-4,0.2)}\cup {(-4,0.2)\times(0,0.2)}$. The physical domain is discretized by uniform grids given by $\{\nex_{j,k} = {(-4+hj,-4+hk)}, j,k=0,\ldots, N\}$, where $h = {4/N}$, $N\in\mathbb N$. At the boundary of the computational domain $\Omega={(-4,0.2)\times(-4,0.2)}$, which encompasses both  $\Omega_0$ and $\Omega_{\rm PML}$, we apply the following periodic boundary conditions:
% \begin{equation*}
%     \begin{aligned}
% u_{N+j,k}=&~u_{j,k},&& j=0,\ldots,p-1,\quad k=0,\ldots N,\\
% u_{j,N+k}=&~u_{j,k},&& j=0,\ldots,N,\quad k=0,\ldots,p-1,\\
% \phi^{(x, r)}_{N+j, k}=&~\phi^{(x, r)}_{j, k}, \remove{= \psi^{(x, r)}_{N+j, k}=\psi^{(x, r)}_{j, k}=0},&& j=0,\ldots,p-r, r=1,2,\ldots,p,  k=0, \ldots N,\\
% \phi^{(y, r)}_{j, N+k}=&~\phi^{(y, r)}_{j, k},\remove{=\psi^{(y, r)}_{j, N+k}=\psi^{(y, r)}_{j, k}=0},&& k=0,\ldots,p-r, r=1,2,\ldots,p,  j=0, \ldots N. \\
% \vah{\phi^{(y, r)}_{j, N+k}=}&~\phi^{(y, r)}_{j, k},&& j=0,\ldots,p-r, r=1,2,\ldots,p,  k=0, \ldots N,\\
% \vah{\psi^{(y, r)}_{j, N+k}=}&~\psi^{(y, r)}_{j, k},&& k=0,\ldots,p-r, r=1,2,\ldots,p,  j=0, \ldots N.
% \end{aligned}
% \end{equation*}
{\begin{equation*}
    \begin{aligned}
u_{N+j,k}=&~u_{j,k},&& j=0,\ldots,p-1,\quad k=0,\ldots N,\\
u_{j,N+k}=&~u_{j,k},&& j=0,\ldots,N,\quad k=0,\ldots,p-1,\\
\phi^{(x, r)}_{N+j, k}=&~\phi^{(x, r)}_{j, k}, && j=0,\ldots,p-r, r=1,2,\ldots,p,  k=0, \ldots N,\\
\phi^{(y, r)}_{j, N+k}=&~\phi^{(y, r)}_{j, k},&& k=0,\ldots,p-r, r=1,2,\ldots,p,  j=0, \ldots N, \\
\phi^{(y, r)}_{j, N+k}=&~\phi^{(y, r)}_{j, k},&& j=0,\ldots,p-r, r=1,2,\ldots,p,  k=0, \ldots N,\\
\psi^{(y, r)}_{j, N+k}=&~\psi^{(y, r)}_{j, k},&& k=0,\ldots,p-r, r=1,2,\ldots,p,  j=0, \ldots N.
\end{aligned}
\end{equation*}}

To measure the spurious numerical reflection errors generated by the proposed high-order RDPML schemes, we conduct a comparison between our numerical solution $u_{j,k}$ and a reference solution $u^{\rm ref}_{j,k}$. The reference solution is generated using the corresponding FD scheme of order $2p$ over the larger domain $\Omega_{\rm ref}={(-5,1)\times(-5,1)}$, which is discretized in such a manner that the nodes lying within $\Omega_0$ align with those where $u_{j,l}$ is defined. Specifically, the two solutions are compared using the following expression:
\begin{equation}\label{eq:discrete_error_3}
E_{\rm 2D}(t) :=\max_{x_{j,k}\in\overline\Omega_0} |u_{j,k}(t)-u^{\rm ref}_{j,k}(t)|
\end{equation}
at the discrete times resulting from the application of the eight-order RK time integrator used above, with a time step $h/16$.
The domains utilized in this example are depicted in \Cref{fig:domains}.  

\Cref{fig:errors_2D_evol} displays the evolution of $E_{\rm 2D}$ in~\eqref{eq:discrete_error_3} throughout the time interval ${[0, 3]}$ using the grid size ${h = 2^{-7}}$. The black line indicates the time at which the wave initially reaches the inner PML boundary, while the red line marks the moment when the waves start to re-enter $\Omega_0$ due to the periodic boundary conditions used.  Clearly, the overall amplitudes of the reflection and transmission errors remain close to the machine precision, demonstrating the effectiveness of the proposed schemes in minimizing these errors.
\begin{figure}[ht!]
    \centering
    \subfloat[]{\includegraphics[width=0.4\linewidth]{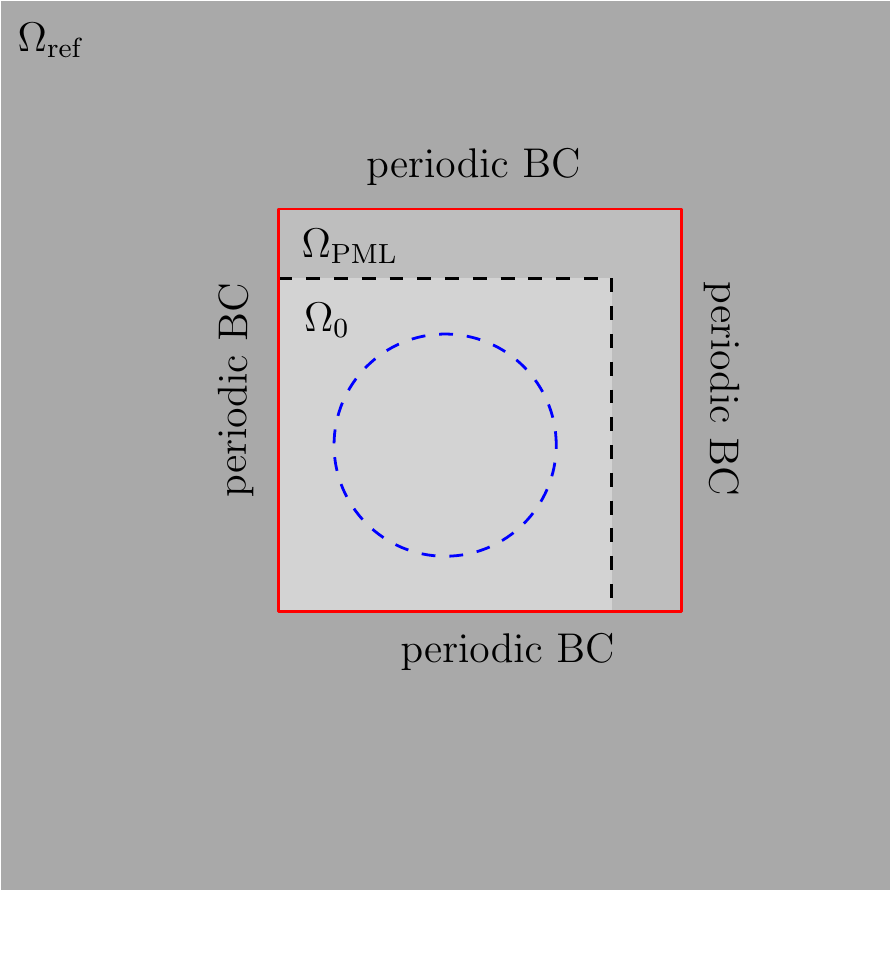}\label{fig:domains}}\qquad 
    \subfloat[]{\includegraphics[scale=0.45]{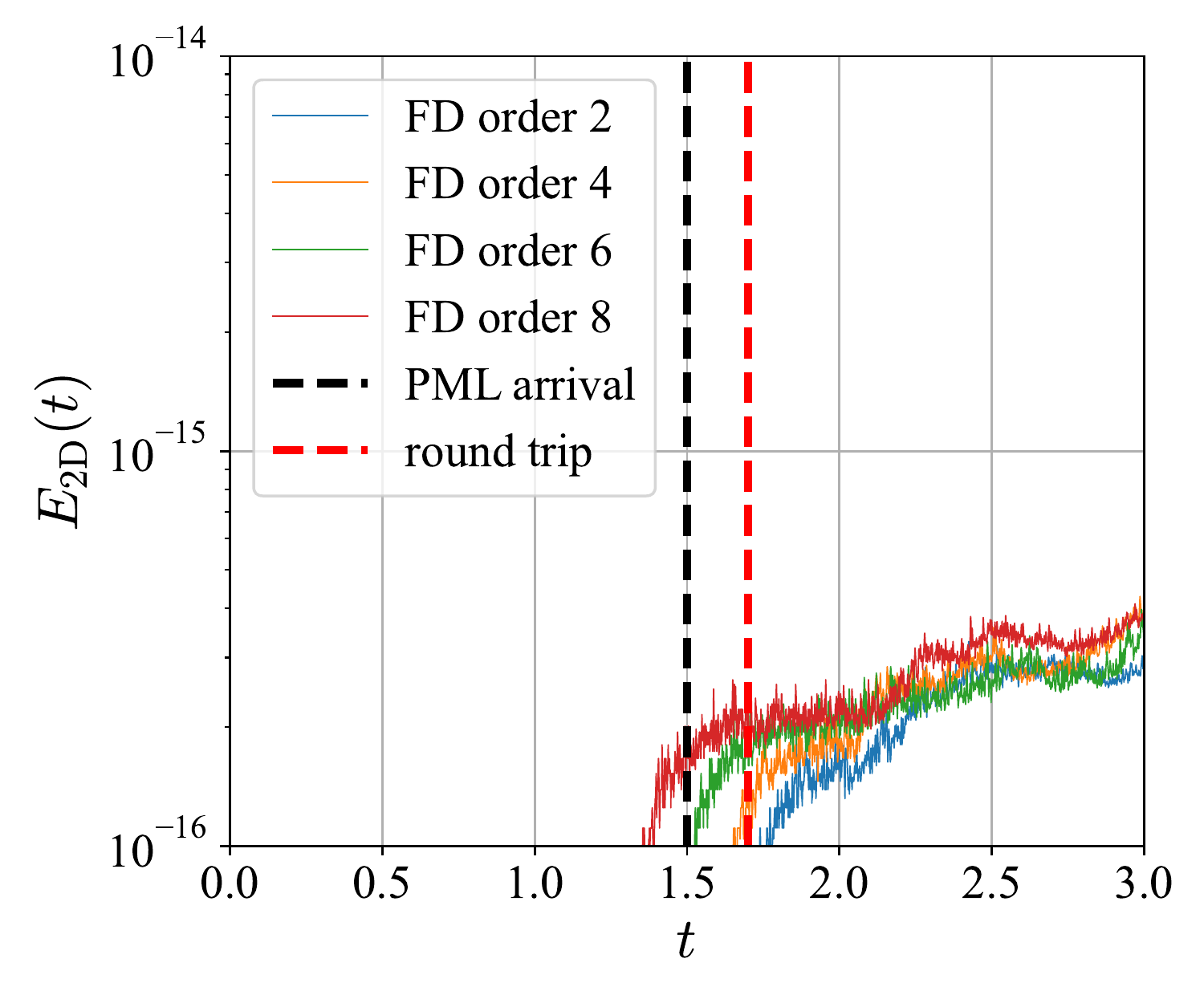}\label{fig:errors_2D_evol}}
    \caption{{(a) Diagram illustrating the domains utilized in the two-dimensional experiments of Section~\ref{sec:2d_wave}. The dotted blue curve marks the boundary of the support of the initial data~\eqref{eq:2d_bump} used in these examples. (b) Time evolution of the numerical error $E_{\rm 2D}$~\eqref{eq:discrete_error_3}, which quantifies the spurious reflections arising from the implementation of the proposed RDPML schemes for solving the two-dimensional wave equation. The error is evaluated for FD spatial discretizations of order 2, 4, 6, and 8, with a mesh refinement of $h=2^{-7}$.}}
    \label{fig:errors_2D}
\end{figure}
\begin{figure}[ht]
    \centering
    \includegraphics[scale=0.42]{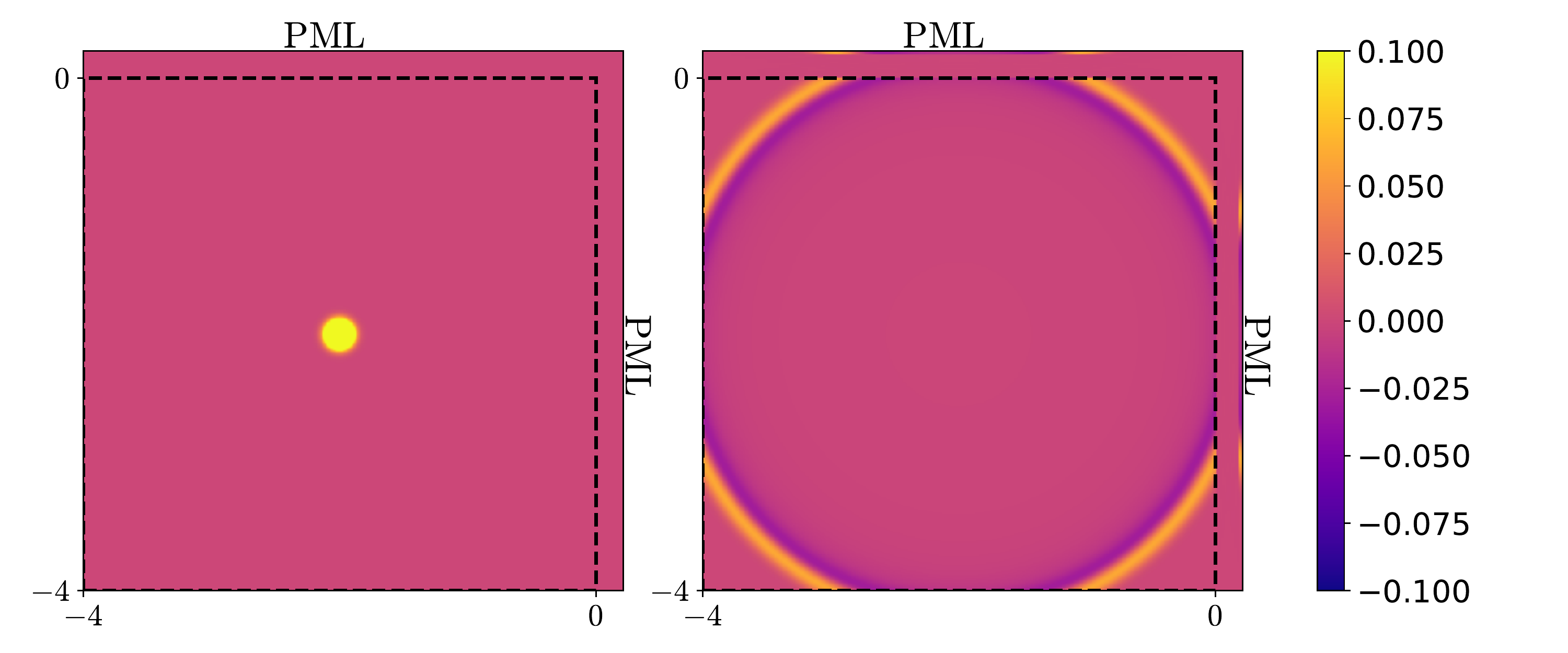} 
    \caption{{Color plots illustrating the initial condition (left) and the fourth-order RDPML solution at time $t=2.2$ (right) for the wave equation in two dimensions, corresponding to the experiment presented in \Cref{fig:errors_2D}. The width of the PML domain is just 5\% the width of the physical domain.\label{fig:epxeriment1_graphics}}}
\end{figure}

{Even though the reflectionless property of the PML ensures the absence of numerical reflection errors generated at the interface between the physical domain and the layer, residual waves can still re-enter the physical domain through the boundary condition at the end of the PML domain (in this case periodic BC's). However, our numerical experiments show that this error can also be driven to machine precision when refining the grid size $h$. Figure \ref{fig:pml5_convergence} shows the history of reflection errors $E_h$, defined as
\begin{equation}\label{eq:discrete_error_2d}
E_h := \max_{t_\ell=h\ell/16\in[0, 3]}\ \max_{x_{j, k}\in\overline\Omega_0} |u_{j, k}(t_\ell)-u^{\mathrm{ref}}(x_{j, k},t_\ell)|,\end{equation}
for grid sizes $h=2^{-n}$, $n=3, \dots, 7$. The experiments are performed with same initial conditions and domain sizes values as the experiment of \Cref{fig:errors_2D}. It is worth  noting that the size of the PML is only 5\% of the size of the physical domain.
}
\begin{figure}
    \centering
    \includegraphics[scale=0.45]{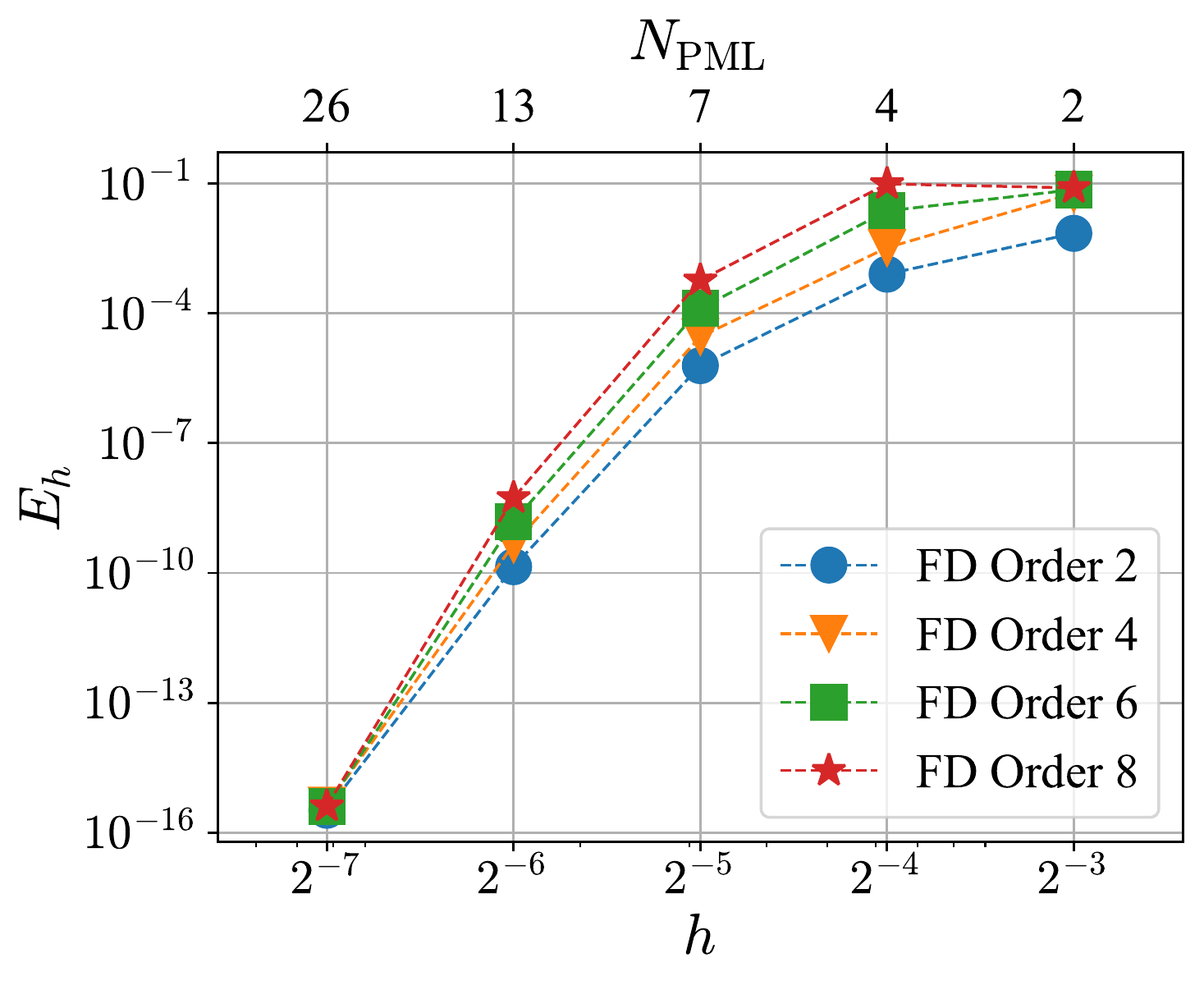}
    \caption{{History of the reflection error $E_h$ of the approximate solutions using the RDPML schemes of order 2, 4, 6, and 8 for the two-dimensional wave equation, where the errors are computed using the formula \eqref{eq:discrete_error_2d}. The marks in the graph correspond to the grid size $h$ (bottom x-axis) and to the number of PML nodes in each direction $N_{\rm PML}$ (top x-axis) used for each experiment.}}
    \label{fig:pml5_convergence}
\end{figure}

\subsection{Waveguide\label{sec:waveguide}}
To highlight the improved dispersion errors achieved by the high-order schemes compared to the classical second-order scheme, this section focuses on applying the fourth-order RDPML scheme to solve a simple waveguide problem. Our problem consists in finding $u:\Omega_\infty\times\R_+\to\R$, with $\Omega_\infty := (-30,\infty)\times(0,2)$, such that
\begin{subequations}
\begin{align}
    \frac{\partial^2 u}{\partial t^2}(\nex, t)-
    \Delta u(\nex, t) & = 0, \quad 
    t>0, \; \nex \in \Omega_\infty
    \label{eqn:waveguide_equation}
\end{align}
 subject to the homogeneous initial conditions:
\begin{equation}
u(\nex, 0) = 0 \andtext \frac{\partial u}{\partial t}(\nex, 0) = 0,    \qquad \nex\in\Omega_\infty,
\label{eqn:waveguide_ics}
\end{equation}
and the Dirichlet boundary conditions:
\begin{align}
    u(\nex,t) =&~ 0,&& \nex\in (-30,\infty)\times \{0,2\},\label{eq:hom_dir_BC}\\
    u(\nex,t) =&~  \sin\left(\frac{x_2\pi}2\right)
    \frac{\cos(\omega t)}{1+\e^{-5(t-1)}},&& \nex=(x_1,x_2)\in \{-30\}\times (0,2).\label{eq:Dir_source}
\end{align}
\label{eq:waveguide_problem}
\end{subequations}

To solve the problem using our RDPML schemes, we discretize an elongated rectangular domain $\Omega_0=(-30, 0)\times (0,2)$ with an attached PML region given by $\Omega_{\mathrm{PML}} = (0, 5) \times (0, 2)$. On the right boundary of the PML, ${5}\times(0,2)$, we impose a homogeneous Dirichlet boundary condition on the RDPML solution $u_{j,k}$. {A detailed description of the boundary conditions used to truncate the PML is given in \Cref{sec:boundary_conditions} in the supplementary material. In particular, to truncate the second and fourth-order RDPML with homogeneous Dirichlet boundary conditions, 2D versions of equations \eqref{eq:boundary_d2} and \eqref{eq:boundary_d4} were used, respectively.} 
%\remove{, as well as the corresponding auxiliary variables $\varphi_{j,k}$ and $\psi_{j,k}$ at the respective nodes}. 
%\remove{Similarly, we apply the same Dirichlet boundary condition to the FD solution $u_{j,k}$ on the top and bottom boundaries of $\Omega_0$, in accordance with~\eqref{eq:hom_dir_BC}. The same boundary condition is also imposed on $u_{j,k}$, $\varphi_{j,k}$, and $\psi_{j,k}$ on the top and bottom boundaries of $\Omega_{\rm PML}$}. 
For discretization, we employ a grid $\{\nex_{j, k} = (-30 + h_{1}j, h_{2}k), 0\leq j\leq 200,0\leq k\leq 10\}$ for $\Omega_0$, where $h_{1} = 0.15$ and $h_{2} = 0.2$. The same resolution is utilized for the discretization of $\Omega_{\mathrm{PML}}$. To evolve the system over time, we employ the explicit eight-order RK method utilized in the aforementioned experiments, employing a time step of $\min\{h_{1}, h_{2}\}/8$.
% To maintain the scheme's order of convergence, we utilize one-sided stencils at the left boundary to implement the Dirichlet boundary condition. However, employing this approach at the right boundary of $\Omega_{\rm PML}$ would necessitate the inclusion of additional auxiliary functions, which, though only required at the boundary, are defined over the entire domain $\Omega_{\rm PML}$. In order to circumvent the associated additional cost, we set $\sigma=0$ in the last \(4\) nodes of the PML.

Considering that the source term in~\eqref{eq:Dir_source} becomes time-harmonic exponentially fast as $t$ increases and assuming the applicability of some form of the limiting amplitude principle in this case, it is expected  that the exact solution $u$ will converge to the time-harmonic waveguide mode given by
\begin{align}\label{eq:waveguide_mode}
    v(\nex,t) = \sin\left( \frac{x_2\pi}{2}\right)\sin(\kappa x_1)\cos(\omega t),\qquad \nex\in\Omega_\infty,
\end{align}
where $\kappa = \sqrt{\omega^2 - \pi^2/4}$. We utilize this mode to observe the effect of dispersion errors on the RDPML solutions. Indeed, \Cref{fig:waveguide_a} presents the second-order and fourth-order solutions alongside the exact mode~\eqref{eq:waveguide_mode} for a large time $t=63$ and $\omega=5$. The vertical blue lines displayed in these plots correspond to two of the zero-level curves of the exact mode~\eqref{eq:waveguide_mode}. They are included to enable a qualitative comparison of the impact of dispersion errors in the numerical solutions. From the plots, it can be observed that the second-order RDPML solution deviates by approximately half a wavelength, whereas the fourth-order solution closely aligns with the time-harmonic mode. Additionally, the dashed line indicates the left boundary $\{0\}\times(0,2)$ of the PML. Lastly, \Cref{fig:waveguide_b} showcases the evolution of the error~\eqref{eq:discrete_error_3} for the waveguide problem, demonstrating the achievement of the reflectionless property of the proposed PML in this specific waveguide example.

\begin{figure}[ht!]
   \centering
        \subfloat[]{\includegraphics[trim={2cm 0cm 0cm 0cm}, scale=0.45]{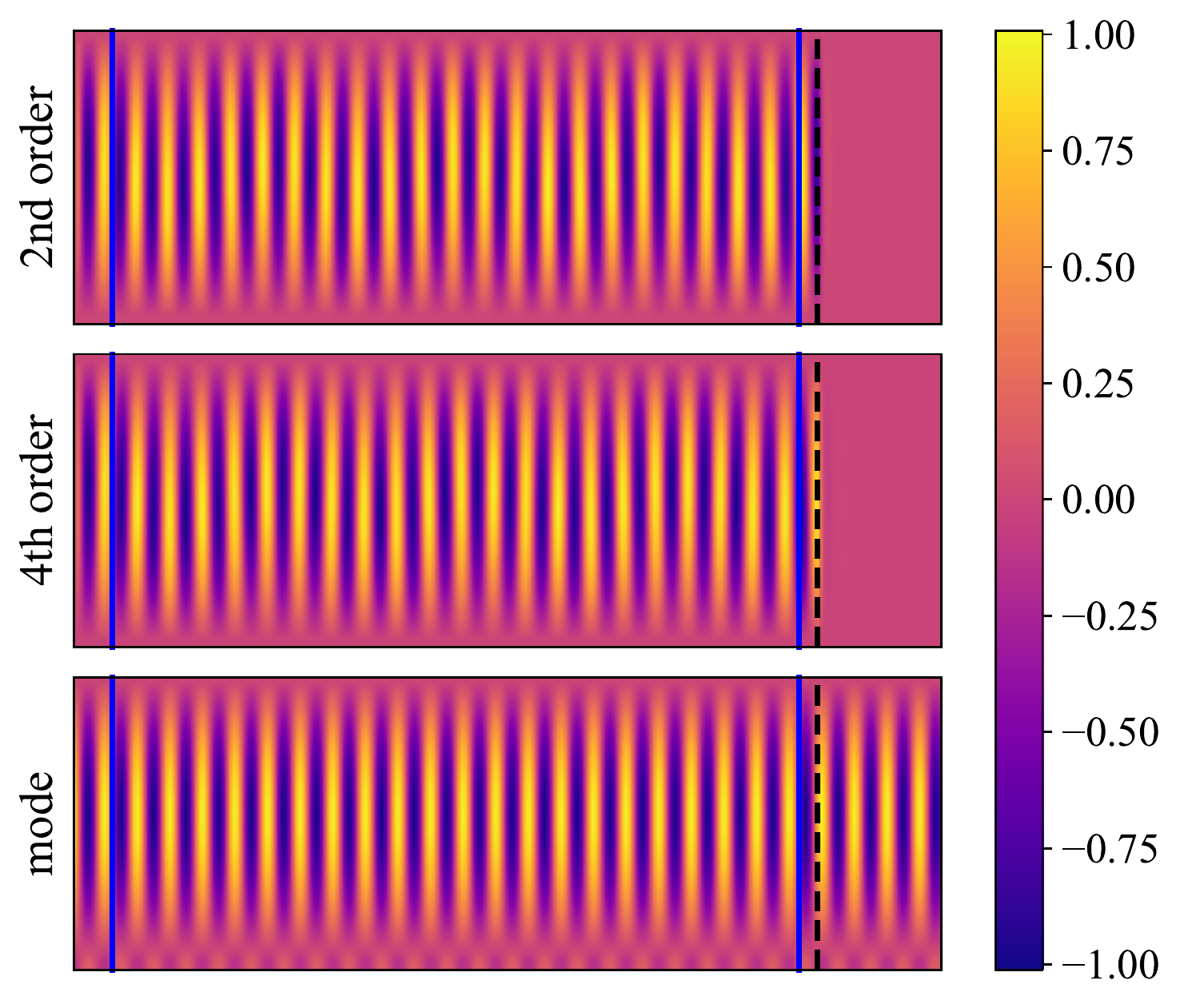}\label{fig:waveguide_a}}
        \subfloat[]{\includegraphics[trim={0cm 0cm 0cm 0cm}, clip, scale=0.45]{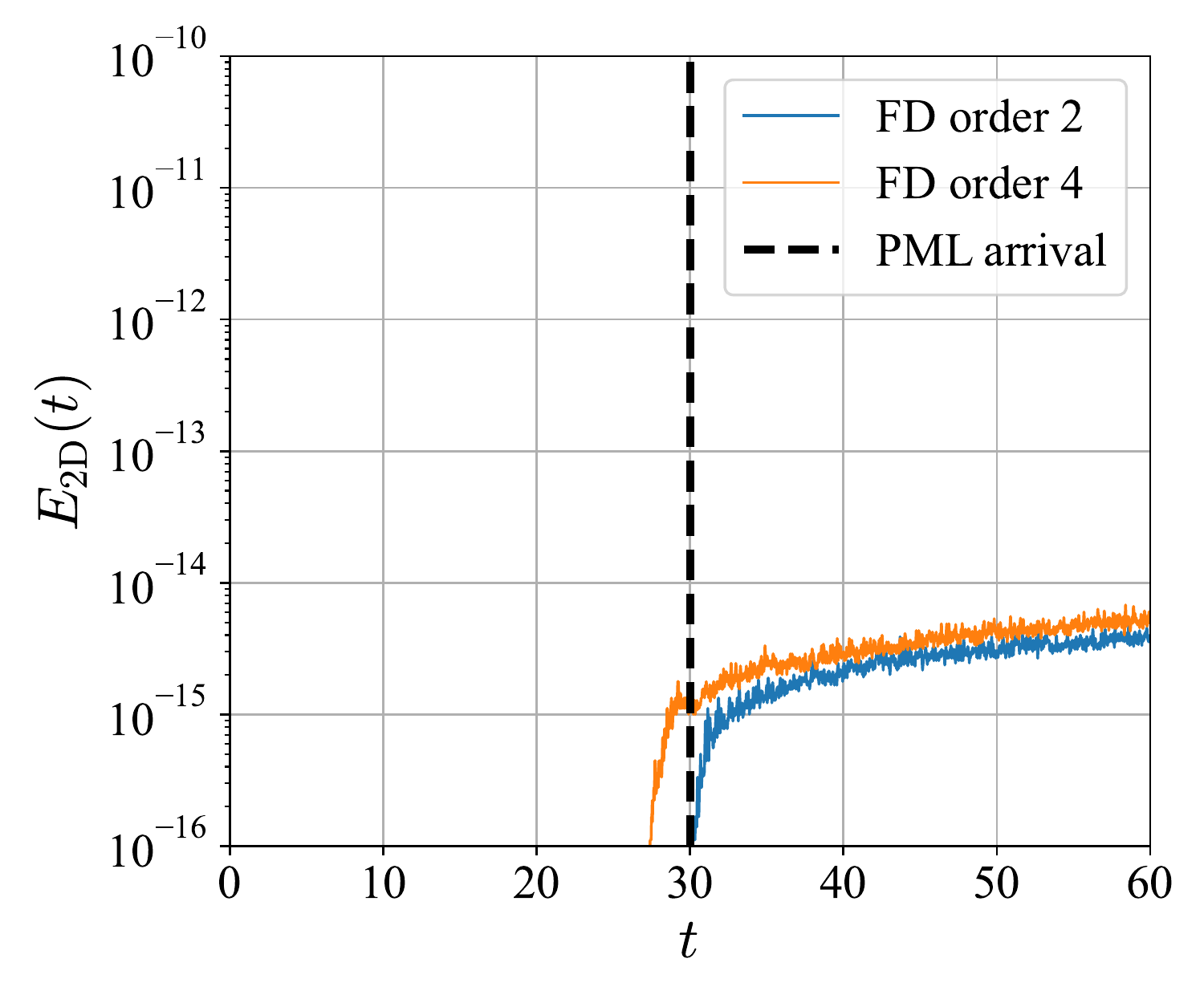}\label{fig:waveguide_b}}
   \caption{Second- and fourth-order RDPML solutions of the waveguide problem~\eqref{eq:waveguide_problem}. (a) Color plots of the numerical solutions together with the time-harmonic mode~\eqref{eq:waveguide_mode} at time $t=63$. The time-harmonic mode~\eqref{eq:waveguide_mode} serves as a reference to visually evaluate the impact of dispersion errors. These errors are considerably more pronounced in the second-order solution compared to the fourth-order solution. (b) Time evolution of the error $E_{\rm 2D}$ that measures the amplitude of the spurious reflections in the RDPML solutions. The machine precision level of $E_{\rm 2D}$ demonstrates the reflectionless property of the proposed FD schemes.}
   \label{fig:dispersion_error}
\end{figure}

{

\subsection{Evanescent waves\label{sec:evanescent}}

Classical PML techniques, along with the proposed RDPML presented herein, are specifically designed to target the absorption of slowly- and non-decaying propagating waves. These waves can be analyzed by considering ``modes" in the form of $v_{\boldsymbol{\xi}}(\nex, t) := \mathrm{e}^{i \boldsymbol{\xi}\cdot\nex - i\omega t}$ where $\boldsymbol{\xi}=(\xi_1,\xi_2)\in\mathbb{R}^2$ satisfies $\xi_1^2+\xi_2^2=\omega^2>0$. However, more general solutions of the wave equation also include evanescent modes $v_{\boldsymbol{\xi}}$ where $\boldsymbol{\xi}=(\xi_1,\xi_2)$ is complex, and these modes are not necessarily absorbed by the PML. For example, modes  $v_{\bol\xi}$ with $\bm{\xi} = (\xi_1, \xi_2)$  such that $\imag\xi_1 > 0$ decay exponentially fast within the physical domain $\Omega_0$ in the positive $x_1$-direction. However, \Cref{thm:reflectionless} does not guarantee their decay as they propagate within $\Omega_{\mathrm{PML}}$, since $\xi_1 \not \in \mathcal{K}$. Indeed, \Cref{fig:rho_amplitude}, which displays the amplitude of the damping factor $\rho(\sigma, \xi, \omega)$ for $\omega=5$, $\sigma=2/h$, $h=0.1$, and $\xi\in\mathbb{C}$ in the region $0\leq \mathrm{Re}(\xi)\leq \pi/h$, $-10\leq\mathrm{Im}(\xi)\leq10$, shows that $|\rho(\sigma, \xi, \omega)| > 1$ in parts of this region where $\mathrm{Im}(\xi) > 0$. This indicates that the corresponding evanescent modes are actually \emph{amplified} by the RDPML in this case.

To mitigate the undesired effects of spurious reflections and instabilities caused by such evanescent modes~\cite{diaz2006,dehoop2002}, a common strategy is to position the PML sufficiently far from the sources, with the hope that these modes have been adequately suppressed before entering the PML region. However, this approach comes at the cost of having to consider large computational domains. Few methods exist to suppress spurious reflections arising from both propagating and evanescent waves, while also avoiding the use of unnecessarily large computational domains~\cite{hagstrom2009,kreiss2016}. In this section, we briefly delve into the use of a two-stage RDPML that combines complex and real grid stretchings~\cite{kreiss2016}. Specifically, the initial stage of our two-stage RDPML considers $\sigma>0$, while the later stage employs $\sigma=0$. This choice ensures that lingering evanescent waves naturally decay before interacting with the boundary condition causing errors upon re-entering the physical domain.
 
% There are several ways to tackle this type of waves, here we show one way of addressing them, which is to implement a 2-stage PML, where the first half of the RDPML is left untouched, but we let $\sigma=0$ on the second half so that evanescent waves decay naturally before interacting with the boundary condition and causing errors when they re-enter the physical domain. 

To assess the effectiveness of the proposed two-stage approach, we revisit the waveguide problem outlined in Section~\ref{sec:waveguide}, but with a slightly modified setup displayed in~\Cref{fig:evanescent_plot}. In this new problem setup, the PML region is positioned much closer (approximately 60 times) to the left boundary, where the Dirichlet boundary condition is imposed. Consequently, the resulting physical domain is defined as $\Omega_0 = (-0.5, 0) \times (-2, 0)$. To evaluate the impact of real grid stretching, we examine two  different boundary conditions and two fourth-order RDPMLs. In detail, we take 
\begin{subequations}\label{eq:waveguide_bcs}
    \begin{align}
    u_1(\nex,t) =&\begin{cases} 0,& \nex\in (-0.5,\infty)\times \{0,2\}.\\
    \displaystyle\sin\left(\frac{x_2\pi}2\right)\frac{\cos(\omega t)}{1+\e^{-5(t-1)}},& \nex=(x_1,x_2)\in \{-0.5\}\times (0,2),\end{cases} \label{eq:dirBC1}\\
    u_2(\nex,t) =&\begin{cases} 0,& \nex\in (-0.5,\infty)\times \{0,2\},\\
     \displaystyle\sin\left(2\pi x_2\right)
    \frac{\cos(\omega t)}{1+\e^{-5(t-1)}},& \nex=(x_1,x_2)\in \{-0.5\}\times (0,2),\end{cases} \label{eq:dirBC2}
\end{align}\end{subequations}
as boundary conditions and RDPMLs given in terms of the damping functions
\begin{subequations} \label{eq:sigma_evanescent}
\begin{align}
\sigma^{(1)}_j & = \begin{cases}
0 & \text{if } j < 0,\\
2/h & \text{if } j \geq 0,
\end{cases}\label{eq:sigma_1} \\
\sigma^{(2)}_j & = \begin{cases}
0 & \text{if } j < 0 ,\\
2/h & \text{if } 0 \leq j \leq \lceil N_{\rm PML}/2\rceil,\\
0 &\text{if } \lceil N_{\rm PML}/2\rceil < j \leq N_{\rm PML},
\end{cases}\label{eq:sigma_2}
\end{align}
\end{subequations}
where $N_{\rm PML}$ denotes the total number of nodes in the discretization of the PML domain in the $x_1$-direction and {$h=0.08$ is the grid size}. This setup allows limited space for evanescent waves to decay before entering the PML domain, resulting in a much larger amplitude compared to the experiment in Section~\ref{sec:waveguide}. Additionally, note that $\sigma^{(1)}$ is the same damping function used in Section~\ref{sec:waveguide}, leading to a pure complex-grid stretching.

 \Cref{fig:trav_exp,fig:evan_exp}, which illustrate the error~\eqref{eq:discrete_error_3} arising from both stretchings for the boundary conditions \eqref{eq:dirBC1} and \eqref{eq:dirBC2}, respectively, reveal that in the case of complex stretching ($\sigma^{(1)}$), although no numerical reflections occur at the interface between the physical domain and the PML, numerical errors emerge over time due to the inadequate absorption of evanescent waves (see also~\Cref{fig:evanescent_plot}). These waves eventually interact with the Dirichlet boundary condition at the right boundary, giving rise to residual waves that re-enter the physical domain. Conversely, the two-stage stretching approach ($\sigma^{(2)}$) significantly reduces the spurious errors caused by evanescent waves, even though propagating waves are damped only in half of the PML domain.

\begin{figure}[ht!]
    \centering
    \hspace{15pt}
    \subfloat[]{\includegraphics[scale=0.45, trim={2cm 0cm 0cm 0cm}]{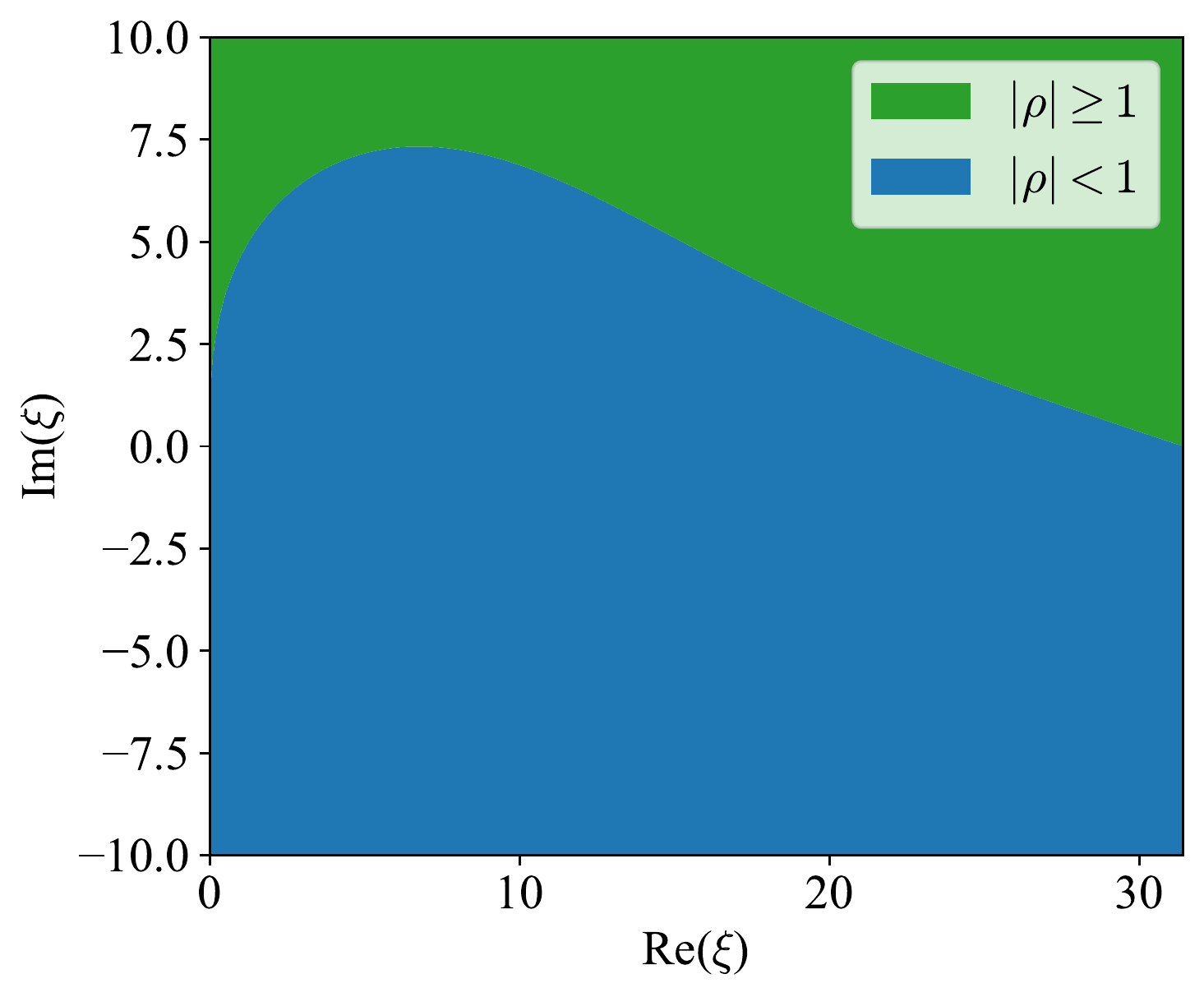}\label{fig:rho_amplitude}}\hspace{30pt}
    \subfloat[]{\includegraphics[scale=0.37, trim={0cm -2cm 0cm 0cm}]{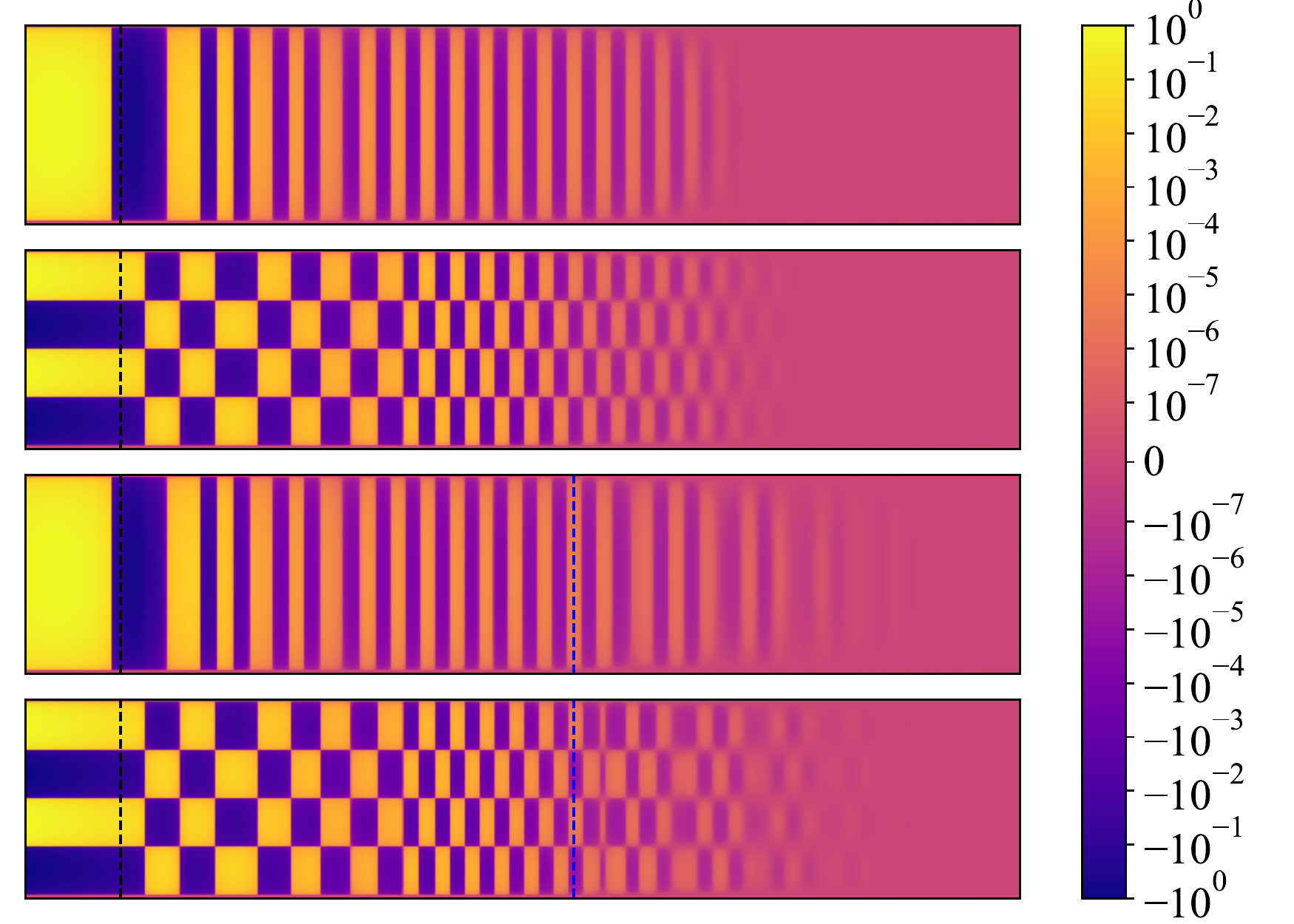}\label{fig:evanescent_plot}}\\
    \subfloat[]{\includegraphics[scale=0.45, trim={2cm 0cm 0cm 0cm}]{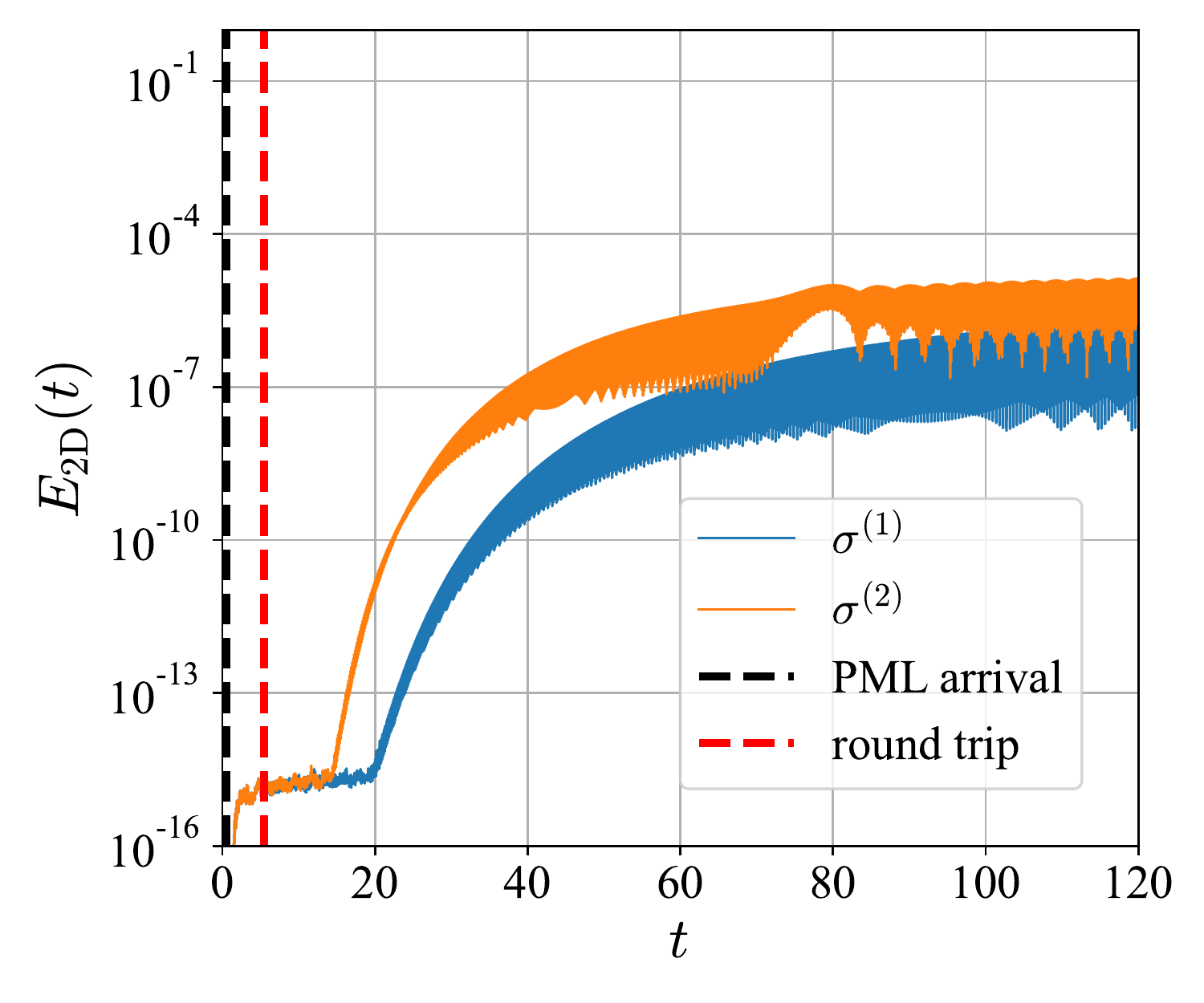}\label{fig:trav_exp}}
    \subfloat[]{\includegraphics[scale=0.45]{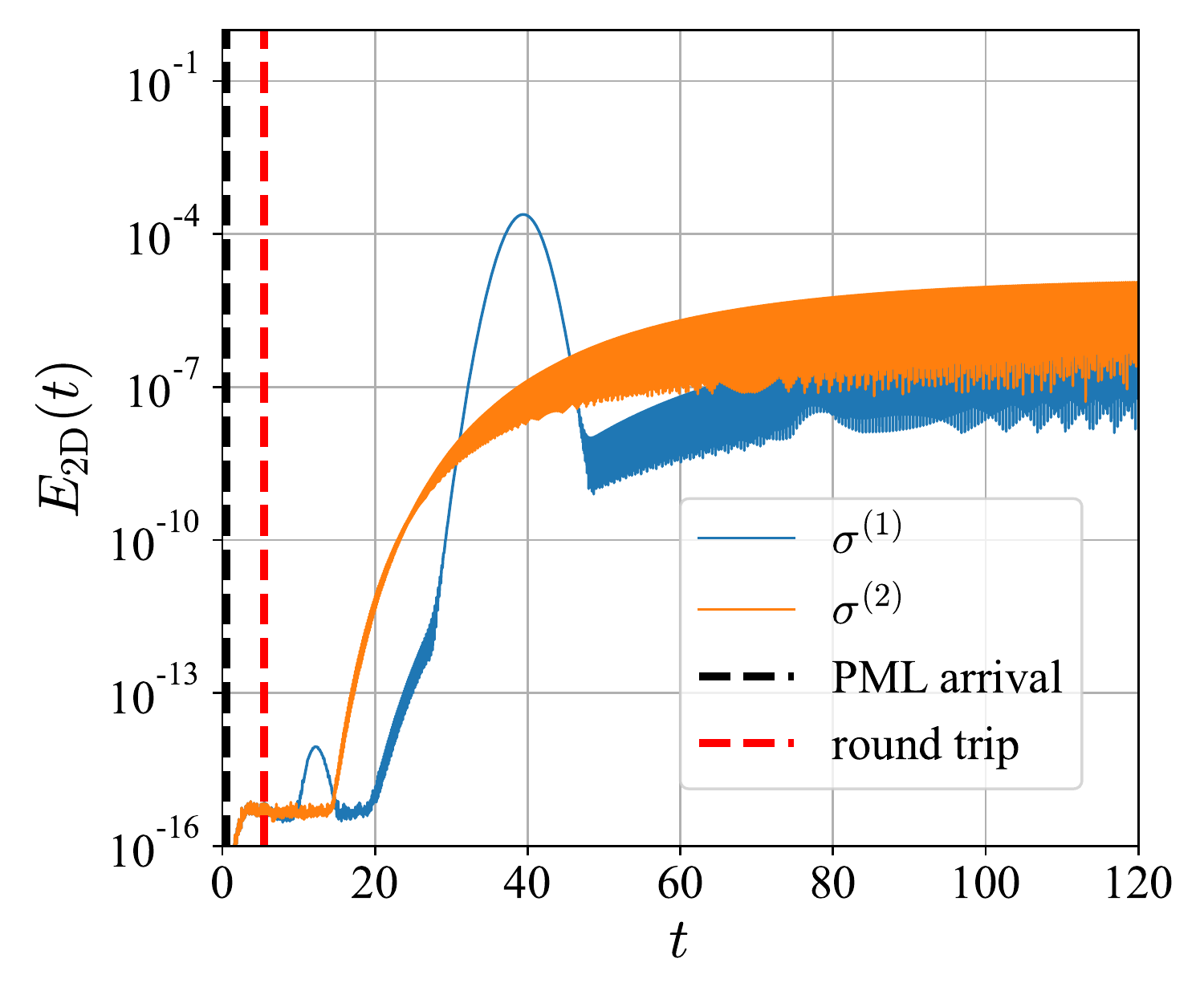}\label{fig:evan_exp}}
    \caption{{(a) Amplitude of decaying factor $\rho$~\eqref{eq:rho}, for $\omega=5$, $h=0.1$, $\sigma=2/h$, and complex $\xi$ values. (b) Plot of fourth-order RDPML solutions of the short waveguide problem of Section~\ref{sec:evanescent} with boundary conditions~\eqref{eq:waveguide_bcs} and damping functions~\eqref{eq:sigma_evanescent} given by $u_1$ and $\sigma^{(1)}$ (1st top), $u_2$ and $\sigma^{(1)}$ (2nd top), $u_1$ and $\sigma^{(2)}$ (3rd top), and $u_2$ and $\sigma^{(2)}$ (4th top), respectively. The vertical black line marks the left boundary of the PML, while the blue line marks the PML stage change. %Instabilities arise for both boundary conditions when using $\sigma^{(1)}$, and are partially mitigated by using $\sigma^{(2)}$ instead.
    (c)-(d)~Time evolution of $E_{\rm 2D}$~\eqref{eq:discrete_error_3} corresponding to the short waveguide problem using boundary conditions $u_1$ (c) and $u_2$ (d) in~\eqref{eq:waveguide_bcs}, and the damping functions $\sigma^{(1)}$ and $\sigma^{(2)}$ in~\eqref{eq:sigma_evanescent}.}} 
\end{figure}
}

\section{Conclusions and ongoing work}\label{sec:conclusion}

We presented a family of RDPMLs specifically designed for arbitrarily high-order  FD spatial discretizations of the scalar wave equation. These RDPMLs were shown to be highly effective in achieving geometric damping of incoming waves at the discrete level. To further validate their performance, we conducted a series of numerical experiments that demonstrated both the high-order accuracy and reflectionless property of the proposed schemes.

However, these achievements come at the expense of increased complexity. The equations governing the RDPMLs introduce additional local auxiliary variables which, when transformed into the frequency domain (as discussed in Section \ref{sec:linear_system}), result in a ``non-local boundary condition". Furthermore, the perfect match property at the discrete level implies that the developed layers are tailored to the specific PDE (in this case, the scalar wave equation) and the chosen high-order FD spatial discretization.

The exploration of RDPMLs for other hyperbolic equations and discretization schemes presents a wide range of possibilities for future research. While the schemes presented in this paper are specifically designed for a particular equation, the underlying theoretical tool of discrete complex analysis is not restricted to a single equation. Its applicability can be extended to investigate the proposed future work. In fact, we have already achieved success in developing a RDPML for continuous piecewise linear Finite Element Methods for the same equation examined in this paper. This serves as promising evidence for the potential application of the developed techniques in other contexts.
%\section*{Acknowledgements}

% ============================================================================
% \end{document}
\appendix
\section{Continuous limit of the RDMPL equations}\label{As:continuouslimit} In order to provide further insight into the RDPML techniques, in this section we explore the limit case $h\to 0+$ for the 2nd-order scheme in one dimension. Letting $x_j = jh = x$, $j \in \mathbb{Z}$, $h > 0$,  we assume there exist sufficiently smooth functions $\sigma(\cdot)$, $\Phi(\cdot,t)$, $\Psi(\cdot,t)$, and $U(\cdot,t)$ so that $\sigma_j = \sigma(x_j)$, $\Phi_j=\Phi(x_j,t)$, $\Psi_j=\Psi(x_j,t)$, and $U_j=U(x_j,t)$. Using these notations, the 2nd-order RDPML equations~\eqref{eqn:pmltime} can be expressed as:
\begin{align*}
U_{tt}(x,t)=&\frac{U(x+h,t)-2U(x,t)+U(x-h,t)}{h^2}+\\
&\hspace{3cm}\frac{1}{h}(\sigma(x)\Psi(x+h,t)-\sigma(x-h)\Phi(x-h,t)),\nonumber\\
\Phi_t(x,t)=&-\frac{\sigma(x)\Phi(x,t)+\sigma(x-h)\Phi(x-h,t)}{2}-\frac{U(x+h,t)-U(x-h,t)}{2h},\\
\Psi_t(x,t)=&-\frac{\sigma(x-h)\Psi(x,t)+\sigma(x)\Psi(x+h,t)}{2}-\frac{U(x+h,t)-U(x-h,t)}{2h}.
\end{align*}
Expanding the relevant quantities around $h=0$ up to the 2nd order using Taylor series, we obtain:
\begin{subequations}
    \begin{align}
U_{tt}-U_{xx}=&\frac{\sigma}{h}(\Psi-\Phi)+\sigma\Psi_x+(\sigma\Phi)_x+\frac{h}{2}(\sigma\Psi_{xx}-(\sigma\Phi)_{xx})+\mathcal O(h^2),\label{eq:td_1}\\
\Phi_t+\sigma\Phi=&-U_x+\frac{h}{2}(\sigma\Phi)_x+\mathcal O(h^2),\label{eq:td_2}\\
\Psi_t+\sigma\Psi=&-U_x-\frac{h\sigma^2}{2}\frac{(\sigma\Psi_x-\sigma'\Psi)}{\sigma^2}+\mathcal O(h^2).\label{eq:td_3}
\end{align}\label{eq:taylor}
\end{subequations}
Taking the limit as $h\to 0+$ in \eqref{eq:taylor}, we get
$$
\Phi_t+\sigma\Phi = \Psi_t+\sigma\Psi = -U_x
$$
from  \eqref{eq:td_1}  and \eqref{eq:td_2}, as well as
$$
U_{tt}-U_{xx}=\Upsilon+\sigma\Psi_x+(\sigma\Phi)_x
$$
% In order to eliminate the term of order $h^{-1}$ in~\eqref{eq:td_1} and considering
% $$\frac{\sigma}{h}(\psi-\phi)=-\frac{\sigma}{2(\sigma+s)}\left\{(\sigma\phi)'+\sigma^2\left(\frac{\psi}{\sigma}\right)'\right\}+O(h),$$ we introduce a new 
from~\eqref{eq:td_1}, where we have introduced the following new auxiliary variable
$$
\Upsilon:=\lim_{h\to0+}\frac{\sigma}{h}(\Psi-\Phi).
$$
To derive the additional relation required for introducing $\Upsilon$, we subtract~\eqref{eq:td_2} from~\eqref{eq:td_3}, multiply the result by $\sigma/h$, and then take the limit as $h\to 0+$ to obtain:
$$
\Upsilon_t+\sigma\Upsilon=-\frac{\sigma}{2}\left\{(\sigma\Phi)_x+\sigma^2\left(\frac{\Psi}{\sigma}\right)_x\right\} = -\sigma^2\Phi_x,
$$
where we have used the fact that $\Phi=\Psi$. Therefore, the time-dependent continuous RDPML equations are given by
\begin{subequations}
    \begin{align}
U_{tt}-U_{xx}=&\Upsilon+\sigma\Phi_x+(\sigma\Phi)_x,\\
\Phi_t+\sigma\Phi=&-U_{x},\\
\Upsilon_t+\sigma\Upsilon =& -\sigma^2\Phi_x.
\end{align}\label{eq:RDPML_we}
\end{subequations}

Finally, to show that the RDPML can be viewed as a specific numerical approximation of the standard PML concerning a particular set of auxiliary variables, we take the Laplace transform ($\partial / \partial t \rightarrow s$) of equations~\eqref{eq:RDPML_we}. By eliminating the transformed auxiliary variables $\widehat\Phi$ and $\widehat\Upsilon$ and assuming a constant PML coefficient $\sigma$, we obtain:
$$
s^2 \widehat{U}-\frac{1}{S^2} \widehat{U}_{xx}=0, \quad S:=1+\frac{\sigma}{s},
$$
where $\widehat U$ represents the Laplace transform of $U$.

\section{Optimizing the damping factor $\sigma$\label{sm:opti_sigma}}  To address the problem of selecting $\sigma$, we optimize its value assuming it remains constant over the entire PML domain. 
    To perform this optimization, we minimize the amplitude of the attenuation factor, $|\rho(\cdot, \xi, \omega)|$. The closer $|\rho(\cdot, \xi, \omega)|$ is to~$0$, the faster the PML damps the mode $\e^{i(\xi(\omega) x_j-\omega t)}$ as it propagates within the PML domain $j>0$. Among the multiple admissible values of $\xi=\xi_r(\omega)$, which depend on the discrete dispersion relation of the FD scheme, we consider the case when $\xi =\omega$, which corresponds to the dispersion relation of the continuous problem. As discussed in Lemma~\ref{thm:operator_sols}, depending on the order of the scheme, the leading  discrete wavenumer $\xi_r(\omega)$ lies in an $O(h^{2p})$  neighborhood of $\omega$ as $h\to0$, so an optimal $\sigma$ for $\xi=\omega$ provides, in principle, a suitable value for this parameter that remain valid for all the all schemes considered. Taking this in consideration, we then seek to minimize the function
    $$
        \varrho(\sigma) := \left|\rho(\sigma, \omega, \omega)\right|  = \left|\frac{\displaystyle 2 + \frac{i\sigma}{\omega}\left(1 - \e^{-i\omega h}\right)}{\displaystyle 2 + \frac{i\sigma}{\omega}\left(1 - \e^{i\omega h}\right)} \right|.
    $$    
    Solving $\varrho'(\sigma)=0$, where 
    \begin{align}
        \varrho'(\sigma) & = \frac{2\omega \sin(h\omega)(\sigma^2-2\omega^2-\sigma^2\cos(h\omega))}{(\sigma^2+2\omega^2-\sigma^2\cos(h\omega)+2\sigma\omega\sin(h\omega))^2\sqrt{-1+\frac{1}{\frac{1}{2}+\frac{\sigma\omega\sin(h\omega)}{\sigma^2+2\omega^2-\sigma^2\cos(h\omega)}}}}\nonumber
    \end{align}
    we  obtain 
    \begin{align}
        \sigma^*:= \frac{\omega\sqrt{2}}{\sqrt{1-\cos(\omega h)}}>0 \nonumber.
    \end{align}
    Since it can be directly checked that $\varrho''(\sigma^*)>0$ for sufficiently small $\omega h>0$ values, we conclude that $\sigma^*$  minimizes locally the attenuation factor $\varrho$.
    
    Unfortunately, $\sigma^*$ above depends on $\omega$, which renders it rather impractical for time-domain simulations. However, by expanding $h\sigma^*$ in a Taylor series around $\omega h=0$, we obtain:
    \begin{align}
        \sigma^{*} = \frac{2}{h} +O(\omega^2 h) \nonumber
    \end{align}
    as $\omega^2 h\to 0$. The value of $\sigma$ used in all our numerical examples, and the one that we advocate for the practical implementation of the proposed methodology, is then:
    $$
    \sigma = \frac{2}{h},
    $$
    which is independent of both $\omega$ and the order of convergence of the FD scheme. 
    
    Interestingly, a more refined analysis for the selection of $\sigma$ can be conducted for each scheme by utilizing the discrete dispersion relation $\xi=\xi_r(\omega)$; similar results are obtained. For instance, for the second-order scheme, we have
    \begin{align*}
        \xi_1 = \frac{1}{h}\cos^{-1}\left(1-\frac{h^2\omega^2}{2}\right). \nonumber
    \end{align*}
    Defining $\varrho_1(\sigma) = |\rho(\sigma, \xi_1, \omega)|$ and solving for $\sigma$ the equation $\varrho_1'(\sigma) = 0$ where
    \begin{align*}
         \varrho_1'(\sigma) & = 
        \frac{2(-4+h^2\sigma^2)\sqrt{4h^2\omega^2-h^4\omega^4}\sqrt{\frac{(4+h^2\sigma^2)\omega-2\sigma\sqrt{4h^2\omega^2-h^4\omega^4}}{(4+h^2\sigma^2)\omega+2\sigma\sqrt{4h^2\omega^2-h^4\omega^4}}}}{(-4+h^2\sigma^2)^2\omega + 4h^4\sigma^2\omega^3},
    \end{align*}
    surprisingly, we obtain
    \begin{align*}
        \sigma^* = \frac{2}{h},
    \end{align*}
    which also satisfies $\varrho_1''(\sigma^*) >0$ for $\omega \neq 2/h$. 
    
    Even though it is feasible to extend this line of reasoning to the exact discrete wavenumbers for higher-order approximations, the expressions for $\xi_p$ in terms of $h$ and $\omega$ become more intricate. Nevertheless, the analysis conducted for the continuous dispersion relation provides a good approximation. This is further supported by Figure~\ref{fig:rho_sigma}, which illustrates that $\sigma = 2/h$ closely approximates a minimizer of $|\rho|$ for all the admissible discrete wavenumbers $\xi_r$ associated with the second, fourth, sixth, and eighth-order schemes, considering fixed values of $h$ and $\omega$ (refer also to Table \ref{stab:exact_wavenumbers} for more precise values).

    \begin{table}[ht!]\small
        \centering
        \begin{tabular}{@{\hskip .1in}c@{\hskip .2in}c@{\hskip .2in}c@{\hskip .2in}c@{\hskip .2in}c@{\hskip .2in}}
        \toprule
        Order &  $\xi_1$ & $\xi_2$ & $\xi_3$ & $\xi_4$\\
        \midrule
        $2$ & $5.0536$ & -                  & -                  & -\\
        $4$ & $-26.5144i$ & $5.0017$  & -                  & -\\
        $6$ & $9.8894-23.6000i$ & $9.8894+23.6000i$  & $5.0000$ & -\\
        $8$ & $-23.5129i$ & $14.5883-21.1132i$ & $14.5883+21.1132i$ &  $5.0000$\\
        \bottomrule
        \end{tabular}
        \caption{Discrete wavenumbers $\xi_r$, as defined in Equation \eqref{eq:cos_rel}, correspond to finite difference schemes of orders 2, 4, 6, and 8, with $\omega=5$ and $h=0.1$. It is  worth noting that in this case the wavenumber $\xi_p$ of the scheme of order $2p$, where $p=1,2,3$ and $4$, approximates $\omega=5$.}
        \label{stab:exact_wavenumbers}
    \end{table}
    \begin{figure}[ht!]
        \centering
        \subfloat[]{\includegraphics[width=0.45\linewidth]{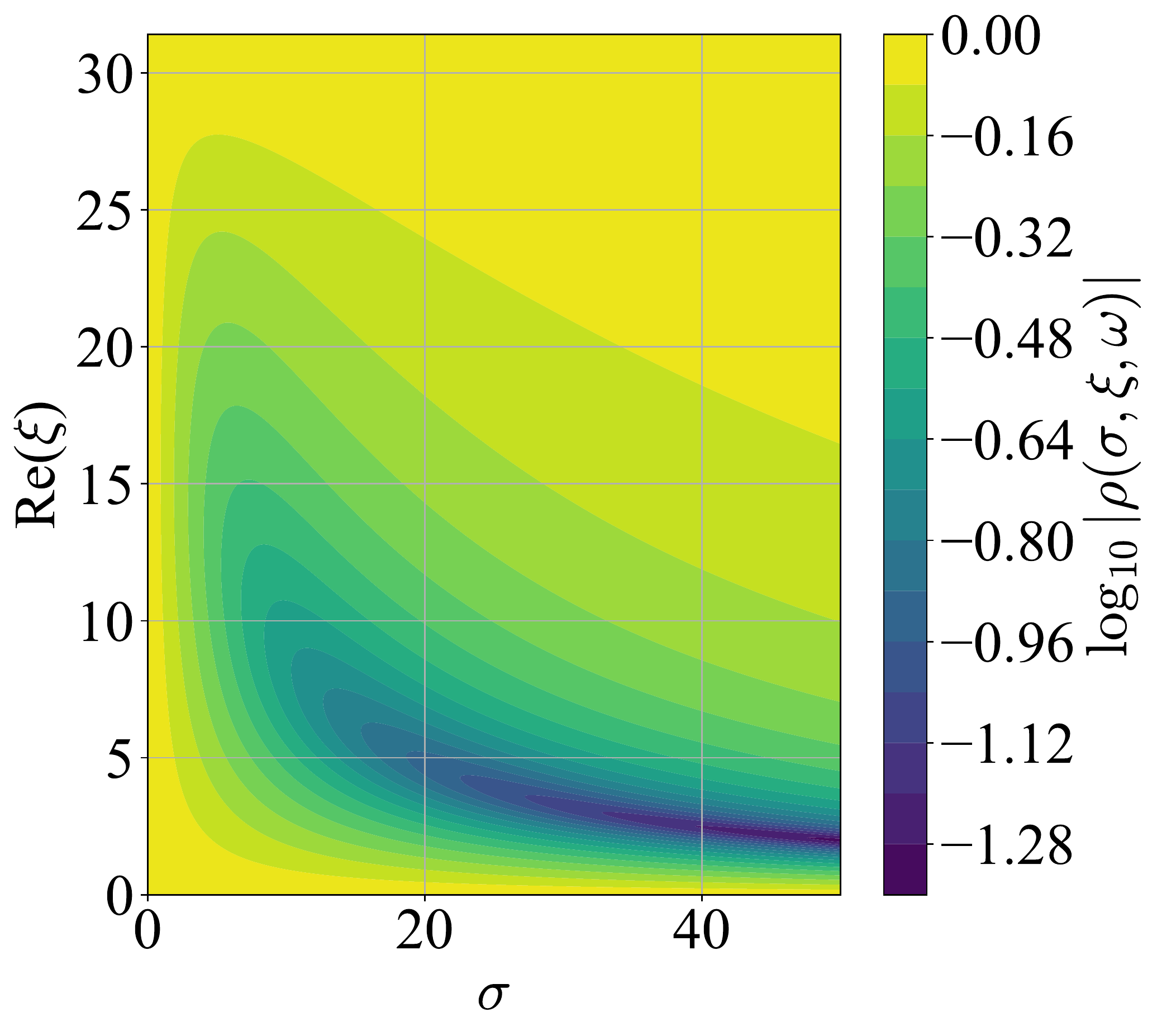}}
        \subfloat[]{\includegraphics[scale=0.45, trim={-0.5cm 0cm 0cm 0cm}]{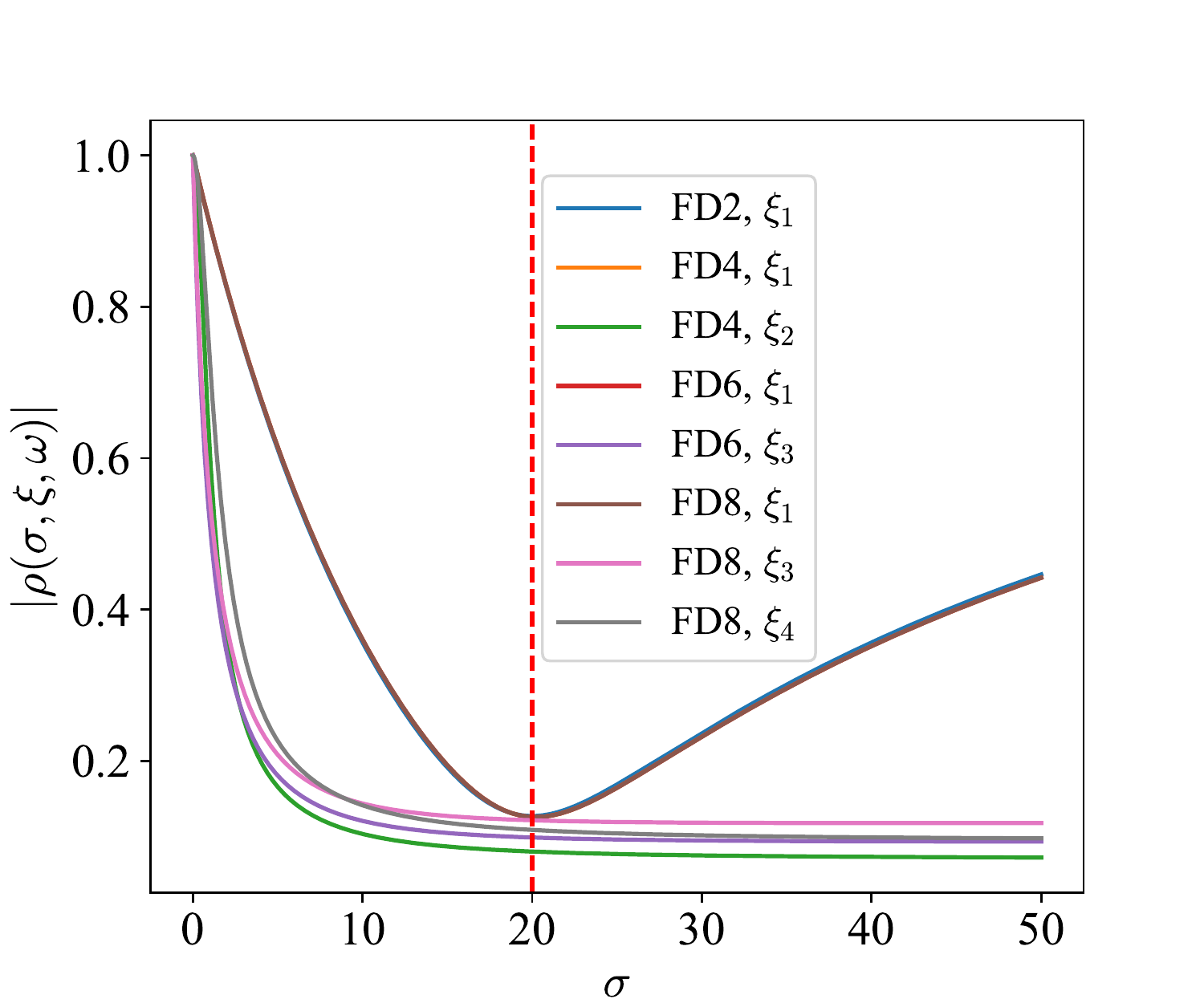}}
        \caption{Plot of the amplitude of the decaying factor $|\rho|$ as a function of $\sigma$ and $\real(\xi)$, for $\omega=5$, $h=0.1$, and $\imag(\xi)=0$. (a) The value $\sigma=2/h=20$ appears to minimize $|\rho|$ for $\xi=\omega=4$, which remains well below~1 at around that point. (b) Plot of $|\rho|$ as a function of $\sigma$ for the discrete wavenumbers $\xi=\xi_r$ in~\eqref{eq:cos_rel} of the finite difference schemes of order 2, 4, 6, and 8, and for $\omega=5$ and $h=0.1$. The dashed red line marks the value of $\sigma=2/h$ where $|\rho|$ appears to be minimal.}
        \label{fig:rho_sigma}
    \end{figure}
    
    % \begin{figure}[ht]
    %     \centering
    %     \includegraphics[width=\linewidth]{journal/Figures/sigma_tests.pdf}
    %     \caption{\vah{Reflection experiments in one dimension for different values of $\sigma$.}}
    %     \label{fig:sigma_tests}
    % \end{figure}
    
\section{Boundary conditions at RDPML boundaries\label{sec:boundary_conditions}} {In this section, we briefly describe the implementation of Dirichlet and Neumann boundary conditions for the second- and fourth-order RDPML used in the waveguide problem. The guiding principle behind these derivations is to maintain the schemes' high-order accuracy by suitably extending the complex grid function $\mathsf{U}$ in an odd (resp. even) manner with respect to the boundary node for the implementation of Dirichlet (resp. Neumann) boundary conditions. In particular, this extension enables the use of the fourth-order scheme at the final node within the PML domain. For simplicity, in what follows, the equations are stated in one dimension, but they can be readily extended to two or higher dimensions.
    
    \textbf{Second order Dirichlet boundary conditions:} For a grid size $h> 0$ and real grid points $x_j = jh$, $j \in \mathbb{Z}$, we consider the corresponding complex grid $\mathsf Z$ and impose homogeneous Dirichlet boundary conditions at the end of the truncated PML domain, that is, at $j=N$, $N > 0$. This can be directly achieved by setting $U_{1, N-1} = 0$ (see Figure \ref{fig:boundary_condition_repr}). Using the identity $U_{1, N-1}=0$ in the time-domain RDPML equations \eqref{eqn:pmltime} with $p=1$, we arrive at the following modified equations for the node $Z_{0, N-1}$:
    \begin{subequations} \label{eq:boundary_d2}
    \begin{align}
        \frac{\de^2 u_{N-1}}{\de t^2} & = \frac{a_0u_{N-1} + a_1u_{N-2}}{h^2} - a_1\frac{\sigma_{N-2}}{h}\phi^{(1)}_{N-2}\\
        -\frac{\de \psi^{(1)}_{N-1}}{\de t} & = \frac{-u_{N-2}}{2h} + \frac{\sigma_{N-2}}{2}\psi^{(1)}_{N-1}
    \end{align}
    \end{subequations}
    where the finite difference coefficients $a_r$ can be found in Table \ref{tab:fd_examples}. Note that the reaming equation \eqref{eq:pmltime_phi} for the auxiliary variable $\phi_{N-1}^{(1)}$  is not used, and the equations for $j \leq N-2$ remain unchanged from \eqref{eqn:pmltime}.}  
    
    \begin{figure}[hb!]
        \centering
        \includegraphics[width=0.9\linewidth]{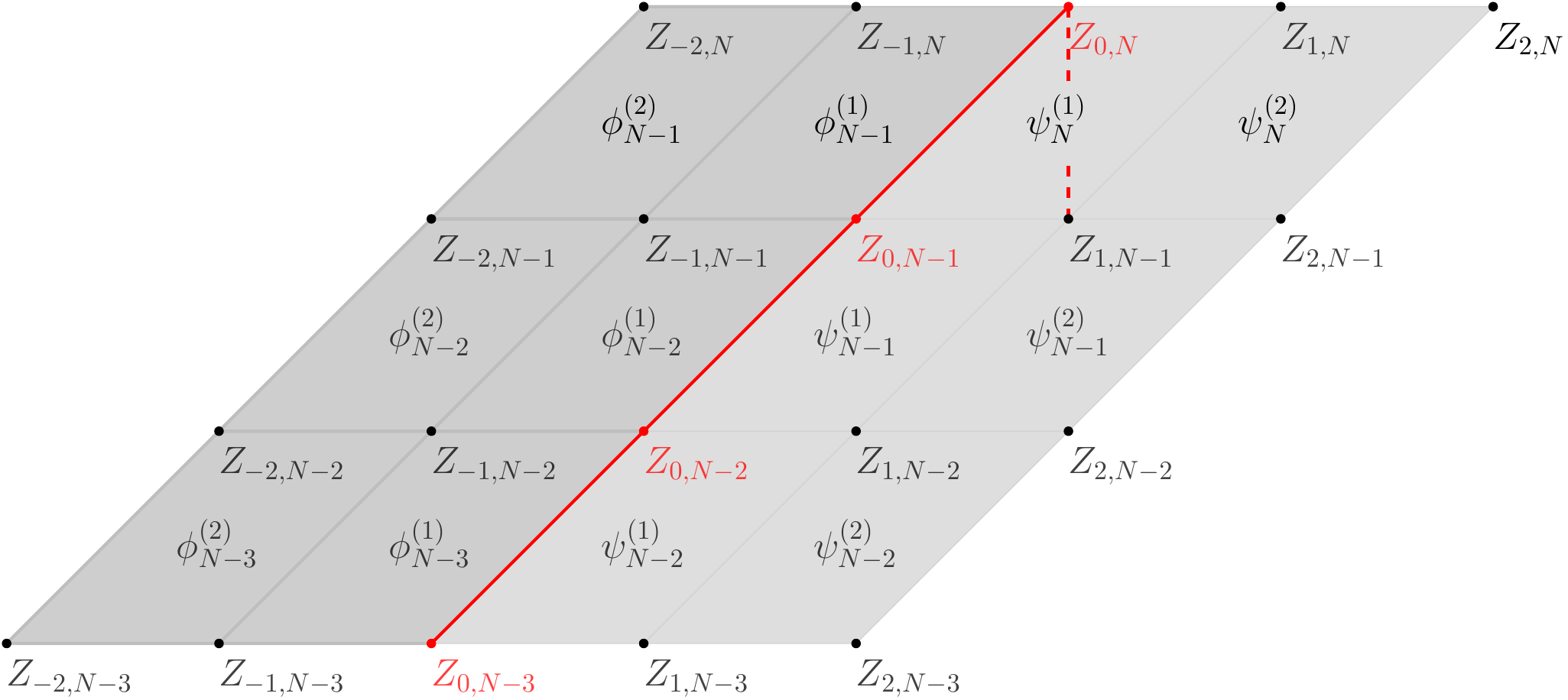}
        \caption{{Representation of boundary conditions in the complex grid $\mathsf Z$ for the fourth-order RDPML. The red dashed line indicates the points where the Dirichlet ($Z_{1, N-1}$) and Neumann ($Z_{0, N}$) boundary conditions are set. The red points indicate the stretching path \eqref{eq:complex_path}}.}
        \label{fig:boundary_condition_repr}
    \end{figure}
    
    {\textbf{Fourth order Dirichlet boundary conditions:} Similarly, for the fourth-order scheme, we enforce homogeneous Dirichlet boundary conditions by again setting $U_{1, N-1}=0$ and complementing it with $U_{2, N-1} = -U_{0, N-1}$. This simple odd extension allows us to write the fourth-order accurate scheme at the final nodes of the truncated PML, $Z_{0, N-1}$ and $Z_{0, N-2}$, without having to introduce additional variables. Indeed, after algebraic manipulations, we arrive at the following modifications of the corresponding RDPML \eqref{eqn:pmltime}:
    \begin{subequations} \label{eq:boundary_d4}
    \begin{align}
        \frac{\de^2 u_{N-1}}{\de t^2} & = \frac{a_2u_{N-3} + a_2u_{N-2} + (a_0-a_2)u_{N-1}}{h^2} - \\
        & \quad -a_2\left(\frac{\sigma_{N-2}}{h}\phi^{(2)}_{N-2} + \frac{\sigma_{N-3}}{h}\phi^{(1)}_{N-3}\right)  - a_1\frac{\sigma_{N-2}}{h}\phi^{(1)}_{N-2},\nonumber\\
        -\frac{\de \psi^{(1)}_{N-1}}{\de t} & = \frac{-u_{N-2}}{2h} + \frac{\sigma_{N-2}}{2}\psi^{(1)}_{N-1},\\
        -\frac{\de \psi^{(2)}_{N-1}}{\de t} & = \frac{\sigma_{N-2}}{2}\psi^{(2)}_{N-1} - \frac{u_{N-1}}{h} - \frac{\sigma_{N-2}}{2}\psi^{(1)}_{N-1},\label{eq:boundary_d4_n1_psi}
    \end{align}
    for $j=N-1$, and
    \begin{align}
        \frac{\de^2 u_{N-2}}{\de t^2} & = \frac{a_2 u_{N-4} + a_1 u_{N-3} + a_0 u_{N-2} + a_1 u_{N-1}}{h^2} - \\
        & \quad -a_2\left(\frac{\sigma_{N-3}}{h}\phi^{(2)}_{N-3} + \frac{\sigma_{N-4}}{h}\phi^{(1)}_{N-4}\right)  - a_1\frac{\sigma_{N-3}}{h}\phi^{(1)}_{N-3} + \nonumber\\
        & \quad  + a_1 \frac{\sigma_{N-2}}{h}\psi^{(1)}_{N-1} + a_2 \frac{\sigma_{N-2}}{h} \psi^{(2)}_{N-1}, \nonumber\\
        - \frac{\de \psi^{(2)}_{N-2}}{\de t} & = \frac{\sigma_{N-3}\psi^{(2)}_{N-2}+\sigma_{N-2}\psi^{(2)}_{N-1}}{2} - \frac{u_{N-2}}{2h} - \frac{\sigma_{N-3}}{2}\psi^{(1)}_{N-2},
    \end{align}
    \end{subequations}  for $j=N-2$.
    The remaining equations in \eqref{eqn:pmltime} for $\psi^{(1)}_{N-2}$, $\phi^{(1)}_{N-2}$, and $\phi^{(2)}_{N-2}$ remain unchanged. Similarly to the second-order scheme, the corresponding equation for $\phi^{(1)}_{N-1}$ is not used. 
    %well defined, this is not an issue since it does not appear in the PDE \eqref{eq:pmltime_pde}. 
    %Also, notice the factor $1/h$ accompanying $u_{N-1}$ in \eqref{eq:boundary_d4_n1_psi} instead of the usual factor of $1/2h$.

    \textbf{Second-order Neumann boundary conditions:} Similar to Dirichlet boundary conditions, we enforce second-order accurate Neumann boundary conditions by setting
    \begin{align}
        \frac{U_{1, N} - U_{-1, N}}{2h} = 0, \nonumber
    \end{align}
    which leads to  $U_{1, N} = U_{-1, N}$. Substituting this identity into \eqref{eqn:pmltime} yields the following changes to equation \eqref{eqn:pmltime} for $j=N$:
    \begin{subequations}
    \begin{align}
        \frac{\de^2 u_N}{\de t^2} & = \frac{2a_1 u_{N-1} + a_0u_N}{h^2} -2a_1 \frac{\sigma_{N-1}}{h}\phi^{(1)}_{N-1}, \nonumber \\
        -\frac{\de \psi^{(1)}_N}{\de t} & = \frac{\sigma_{N-1}}{2}\psi^{(1)}_N -\frac{\sigma_{N-1}}{2}\phi^{(1)}_{N-1}. \nonumber
    \end{align}
    \end{subequations}
    Again, the corresponding equation for $\phi_{N}^{(1)}$ is not used.
    
    \textbf{Fourth order Neumann boundary conditions:} Likewise, we set $U_{-1, N} = U_{1, N}$ and consider an even extension of $\mathsf{U}$ such that $U_{-2, N} = U_{2, N}$. This results in the following changes to equation~\eqref{eqn:pmltime}:
    \begin{subequations}\label{eq:neumann}
    \begin{align}
        \frac{\de^2 u_{N}}{\de t^2} & = \frac{2a_2u_{N-2}+2a_1u_{N-1}+a_0u_N}{h^2}- 2a_2\left(\frac{\sigma_{N-1}}{h}\phi^{(2)}_{N-1} + \frac{\sigma_{N-2}}{h}\phi^{(1)}_{N-2}\right) -\\
        & \quad - 2a_1 \frac{\sigma_{N-1}}{h}\phi^{(1)}_{N-1}, \nonumber \\
        -\frac{\de \psi^{(1)}_{N}}{\de t} & = \frac{\sigma_{N-1}}{2}\psi^{(1)}_N - \frac{\sigma_{N-1}}{2}\phi^{(1)}_{N-1},\\
        -\frac{\de \psi^{(2)}_{N}}{\de t} & = \frac{u_{N-2}-u_N}{2h} + \frac{\sigma_{N-1}}{2}\psi^{(2)}_{N} - \frac{\sigma_{N-1}}{2}\psi^{(1)}_N - \frac{\sigma_{N-1}}{2}\phi^{(2)}_{N-1} - \frac{\sigma_{N-2}}{2}\phi^{(1)}_{N-2},
    \end{align}
     for $j=N$, and 
    \begin{align}
        \frac{\de^2 u_{N-1}}{\de t^2} & = \frac{a_2 u_{N-3} + a_1 u_{N-2} + (a_0+a_2)u_{N-1} + a_1u_{N}}{h^2} -\\
        & \quad -a_2\left(\frac{\sigma_{N-2}}{h}\phi^{(2)}_{N-2} + \frac{\sigma_{N-3}}{h}\phi^{(1)}_{N-3}\right) - a_1\frac{\sigma_{N-2}}{h}\phi^{(1)}_{N-2} + a_1 \frac{\sigma_{N-1}}{h}\psi^{(1)}_N + \nonumber \\
        & \quad + a_2 \left(\frac{\sigma_{N-1}}{h}\psi^{(2)}_N - \frac{\sigma_{N-1}}{h}\phi^{(1)}_{N-1}\right), \nonumber \\
        -\frac{\de \psi^{(2)}_{N-1}}{\de t} & = \frac{\sigma_{N-1}\psi^{(2)}_N + \sigma_{N-2}\psi^{(2)}_{N-1}}{2} - \frac{\sigma_{N-2}}{2}\psi^{(1)}_{N-1} - \frac{\sigma_{N-1}}{2}\phi^{(1)}_{N-1},
    \end{align}
    \end{subequations}
    for $j=N-1$. Once again, the equation for $\phi^{(1)}_N$ is not used.
    
Finally, to validate these equations, we present Figures~\ref{fig:trav_exp_mixed} and~\ref{fig:evan_exp_mixed}, displaying the numerical reflection errors resulting from truncating the RDPML using Neumann boundary conditions in the short waveguide problem considered in Section \ref{sec:evanescent} of the paper. As expected, these results are similar to those obtained by truncating the RDPML using Dirichlet boundary conditions (cf. Figures~\ref{fig:trav_exp} and~\ref{fig:evan_exp} in the paper).
    }
\begin{figure}[ht!]
    \centering
    \subfloat[]{\includegraphics[scale=0.45, trim={2cm 0cm 0cm 0cm}]{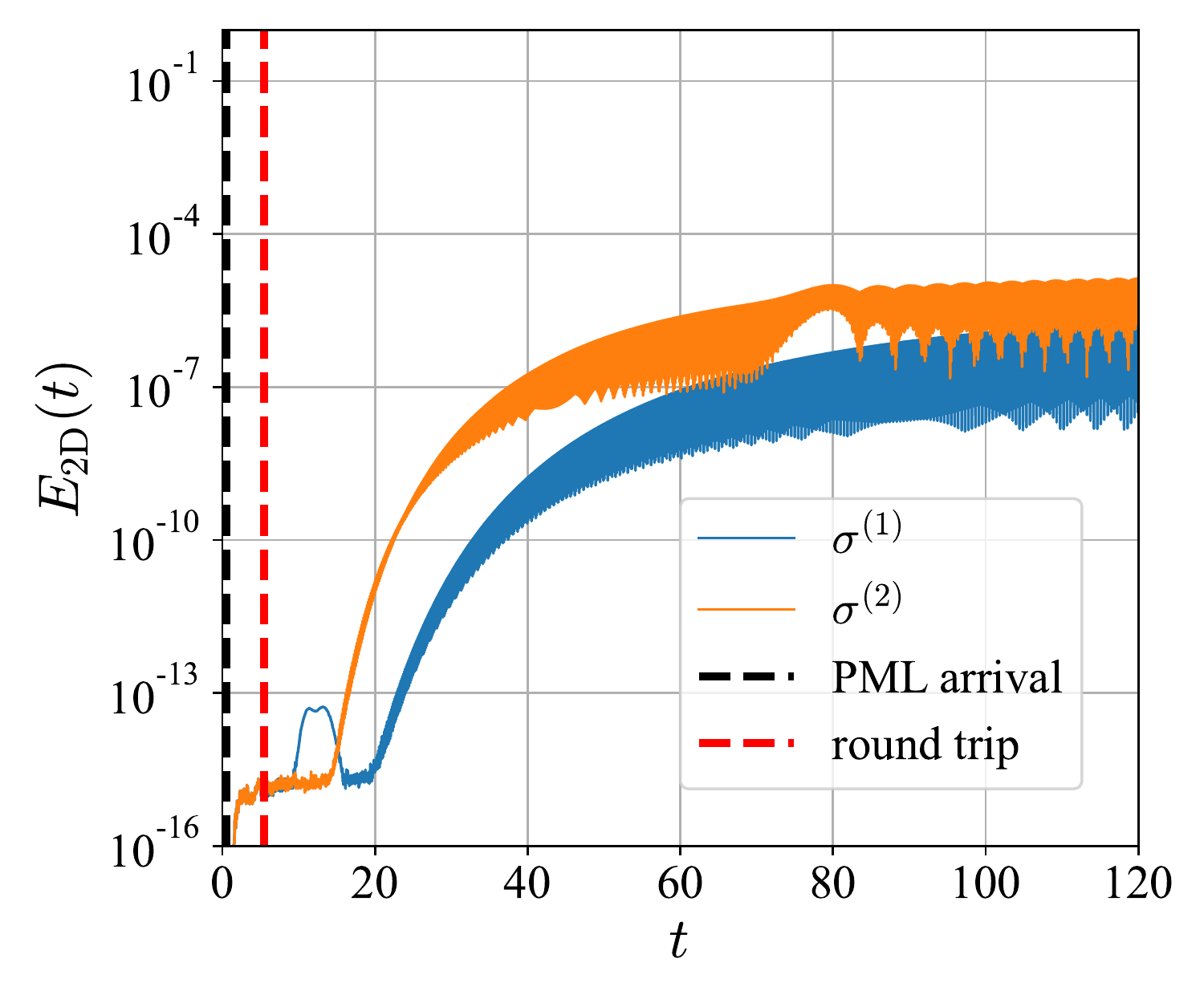}\label{fig:trav_exp_mixed}}
    \subfloat[]{\includegraphics[scale=0.45]{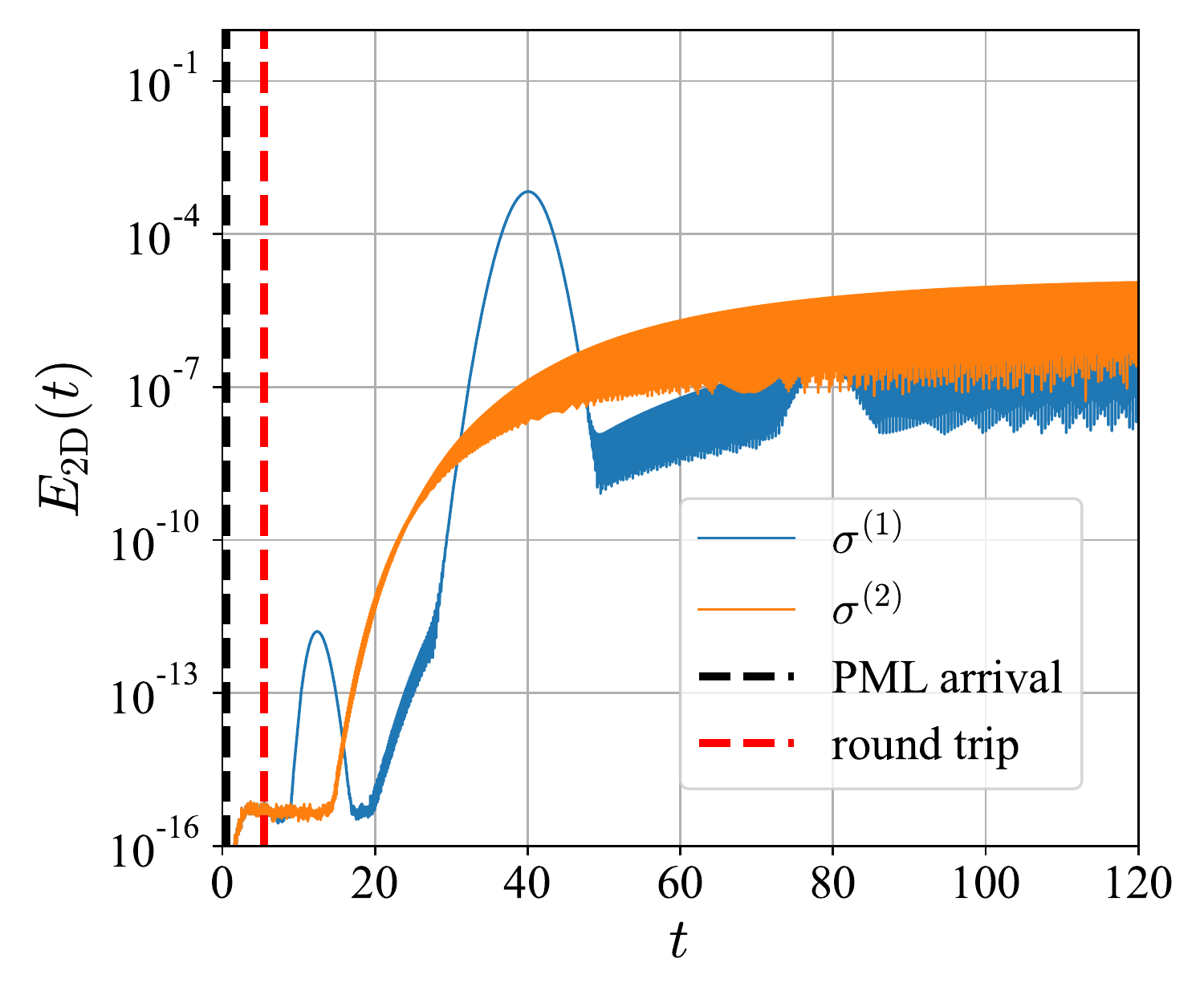}\label{fig:evan_exp_mixed}}
    \caption{{(a)-(b)~Time evolution of $E_{\rm 2D}$~\eqref{eq:discrete_error_3} corresponding to the short waveguide problem using boundary conditions $u_1$ (a) and $u_2$ (b) in~\eqref{eq:waveguide_bcs}, and the damping functions $\sigma^{(1)}$ and $\sigma^{(2)}$ in~\eqref{eq:sigma_evanescent}. The PML is terminated with Neumann boundary conditions \eqref{eq:neumann}. There is no significant difference to terminating the PML with Dirichlet boundary conditions (for comparison see Figures \ref{fig:trav_exp} and \ref{fig:evan_exp} from the main text).}} 
\end{figure}

\bibliographystyle{abbrv}
\bibliography{bibliography.bib}
% ============================================================================

\end{document}